\documentclass[11pt, reqno]{article}

\usepackage{amsmath, amsthm, amssymb}
\usepackage{enumerate}
\usepackage{pdflscape}
\usepackage{caption}
\usepackage{bm}
\usepackage[T1]{fontenc}

\usepackage{ifpdf}
\ifpdf
\usepackage[pdftex]{graphicx}
\else
\usepackage[dvips]{graphicx}\\
\fi
\usepackage{xcolor}
\usepackage{tikz}
\usetikzlibrary{arrows,backgrounds}
\usetikzlibrary{decorations.pathmorphing}
\usetikzlibrary{knots}

\usepackage[backend=biber,style=alphabetic,citestyle=alphabetic,url=false,isbn=false,maxnames=99,maxalphanames=99]{biblatex}
\bibliography{bibliography}

\usepackage[all]{xy}

\usepackage{multicol}
\usepackage{mathtools} % for "\DeclarePairedDelimiter" macro
\usepackage{tocvsec2}
\usepackage{bbm}
\usepackage{stmaryrd} % for \mapsfrom
\usepackage{quiver}
\usepackage[usestackEOL]{stackengine}
\input xy
\xyoption{all}

\usepackage[pdftex,plainpages=false,hypertexnames=false,pdfpagelabels]{hyperref}
\newcommand{\arxiv}[1]{\href{http://arxiv.org/abs/#1}{\tt arXiv:\nolinkurl{#1}}}
\newcommand{\arXiv}[1]{\href{http://arxiv.org/abs/#1}{\tt arXiv:\nolinkurl{#1}}}
\newcommand{\doi}[1]{\href{http://dx.doi.org/#1}{{\tt DOI:#1}}}

\newcommand{\googlebooks}[1]{(preview at \href{http://books.google.com/books?id=#1}{google books})}

\definecolor{dark-red}{rgb}{0.7,0.25,0.25}
\definecolor{dark-blue}{rgb}{0.15,0.15,0.55}
\definecolor{medium-blue}{rgb}{0,0,.8}
\definecolor{DarkGreen}{RGB}{0,150,0}
\definecolor{darkgreen}{rgb}{0,150,0}
\definecolor{YellowOrange}{HTML}{FAA21A}
\hypersetup{
	colorlinks, linkcolor={purple},
	citecolor={medium-blue}, urlcolor={medium-blue}
}
\usepackage{longtable}
\usepackage{fullpage}
\theoremstyle{plain}
\newtheorem{thm}{Theorem}[section]
\newtheorem*{thm*}{Theorem}

\newtheorem{cor}[thm]{Corollary}

\newtheorem*{cor*}{Corollary}

\newtheorem*{conj*}{Conjecture}
\newtheorem{lem}[thm]{Lemma}

\newtheorem{prop}[thm]{Proposition}

\newtheorem*{quest*}{Question}
\newtheorem*{claim*}{Claim}

\theoremstyle{definition}
\newtheorem{defn}[thm]{Definition}
\newtheorem{construction}[thm]{Construction}
\newtheorem{alg}[thm]{Algorithm}

\newtheorem{sub-ex}[thm]{Sub-Example}
\newtheorem{rem}[thm]{Remark}
\newtheorem*{rem*}{Remark}
\newtheorem{remark}[thm]{Remark}

\DeclareMathOperator{\End}{End}

\DeclareMathOperator{\ev}{ev}

\DeclareMathOperator{\Hom}{Hom}

\DeclareMathOperator{\op}{op}

\DeclareMathOperator{\id}{id}

\newcommand{\comment}[1]{}

\newcommand{\be}{\begin{enumerate}[label=(\arabic*)]}
	\newcommand{\ee}{\end{enumerate}}

\newcommand{\set}[2]{\left\{#1 \middle| #2\right\}}

\def\semicolon{;}
\def\applytolist#1{
	\expandafter\def\csname multi#1\endcsname##1{
		\def\multiack{##1}\ifx\multiack\semicolon
		\def\next{\relax}
		\else
		\csname #1\endcsname{##1}
		\def\next{\csname multi#1\endcsname}
		\fi
		\next}
	\csname multi#1\endcsname}

\def\calc#1{\expandafter\def\csname c#1\endcsname{{\mathcal #1}}}
\applytolist{calc}QWERTYUIOPLKJHGFDSAZXCVBNM;
\def\bbc#1{\expandafter\def\csname bb#1\endcsname{{\mathbb #1}}}
\applytolist{bbc}QWERTYUIOPLKJHGFDSAZXCVBNM;
\def\bfc#1{\expandafter\def\csname bf#1\endcsname{{\mathbf #1}}}
\applytolist{bfc}QWERTYUIOPLKJHGFDSAZXCVBNM;
\def\sfc#1{\expandafter\def\csname s#1\endcsname{{\sf #1}}}
\applytolist{sfc}QWERTYUIOPLKJHGFDSAZXCVBNM;
\def\fc#1{\expandafter\def\csname f#1\endcsname{{\mathfrak #1}}}
\applytolist{fc}QWERTYUIOPLKJHGFDSAZXCVBNM;

\newcommand{\Rep}{{\sf Rep}}

\newcommand{\Mod}{{\sf Mod}}

\renewcommand{\Vec}{{\sf Vec}}
\newcommand{\fdVec}{{\sf Vec_{fd}}}

\newcommand{\sVec}{{\sf sVec}}
\newcommand{\TwoVec}{{\sf 2Vec}}
\newcommand{\ThreeVec}{{\sf 3Vec}}

\newcommand{\Slice}{{\sf Slice}}
\newcommand{\TwoSlice}{{\sf 2Slice}}
\newcommand{\Alg}{{\sf Alg}}
\newcommand{\TwoAlg}{{\sf 2Alg}}

\newcommand{\noshow}[1]{}
\usetikzlibrary{shapes}
\usetikzlibrary{cd}
\usetikzlibrary{backgrounds}
\usetikzlibrary{decorations,decorations.pathreplacing,decorations.markings}
\usetikzlibrary{fit,calc,through}
\usetikzlibrary{external}
\usetikzlibrary{arrows}
\tikzset{vertex/.style = {shape=circle,draw,fill=black,inner sep=0pt,minimum size=5pt}}
\tikzset{edge/.style = {->,> = latex', bend right}}
\tikzset{
	super thick/.style={line width=3pt}
}
\tikzset{
	quadruple/.style args={[#1] in [#2] in [#3] in [#4]}{
		#1,preaction={preaction={preaction={draw,#4},draw,#3}, draw,#2}
	}
}
\tikzstyle{shaded}=[fill=red!10!blue!20!gray!30!white]
\tikzstyle{unshaded}=[fill=white]
\tikzstyle{empty box}=[circle, draw, thick, fill=white, opaque, inner sep=2mm]
\tikzstyle{annular}=[scale=.7, inner sep=1mm, baseline]
\tikzstyle{rectangular}=[scale=.75, inner sep=1mm, baseline=-.1cm]
\tikzstyle{late>}=[decoration={markings, mark=at position 0.75 with {\arrow{>}}}, postaction={decorate}]
\tikzstyle{late<}=[decoration={markings, mark=at position 0.75 with {\arrow{<}}}, postaction={decorate}]
\tikzstyle{mid>}=[decoration={markings, mark=at position 0.5 with {\arrow{>}}}, postaction={decorate}]
\tikzstyle{mid<}=[decoration={markings, mark=at position 0.5 with {\arrow{<}}}, postaction={decorate}]
\tikzstyle{over}=[double, draw=white, super thick, double=]
\tikzstyle{box} = [rectangle,draw,rounded corners=5pt,very thick]
\tikzset{Rightarrow/.style={double equal sign distance,>={Implies},->},
	triple/.style={-,preaction={draw,Rightarrow}},
	quadruple/.style={preaction={draw,Rightarrow,shorten >=0pt},shorten >=1pt,-,double,double
		distance=0.2pt}}
\tikzstyle{knot}=[preaction={super thick, white, draw}]

\newcommand{\roundNbox}[6]{
	\draw[rounded corners=5pt, very thick, #1] ($#2+(-#3,-#3)+(-#4,0)$) rectangle ($#2+(#3,#3)+(#5,0)$);
	\coordinate (ZZa) at ($#2+(-#4,0)$);
	\coordinate (ZZb) at ($#2+(#5,0)$);
	\node at ($1/2*(ZZa)+1/2*(ZZb)$) {#6};
}

\newcommand{\tikzmath}[2][]
{\vcenter{\hbox{\begin{tikzpicture}[#1]#2
	\end{tikzpicture}}}
}

\begin{document}
	
	\title{Tannaka-Krein reconstruction for fusion 2-categories}
	\author{David Green}
	\date{\today}
	\maketitle
	\begin{abstract}
		We reprove the classical Tannaka-Krein reconstruction theorem by finding a monoidal equivalence of categories between a 1-truncated sub-2-category of the slice 2-category $\Mod(\Vec)/\Vec$ and the category of algebras. We then immediately generalize this approach to find a monoidal equivalence of 2-categories between a 2-truncated sub-3-category of the slice 3-category $\Mod(\TwoVec)/\TwoVec$ and the category of algebras.
		  
		As an immediate consequence, a finite semisimple 2-Hopf algebra $C$ can be recovered from its fusion 2-category of modules together with the monoidal fiber 2-functor to $\TwoVec$. Moreover, every fusion 2-category equipped with a monoidal functor to $\TwoVec$ is of this form.
	\end{abstract}
	\section{Introduction} 
	The classical Tannaka reconstruction theorem \cite{Tannaka} recovers a compact topological group from its category of finite dimensional complex representations and forgetful functor $F$. The counterpart to this theorem, due to Krein \cite{kreuin1949principle} recovers a pair $(C, F)$ of a semisimple linear category and faithful functor $F$ as the representations of $\End(F)$.
	These ideas have been subsequently generalized by many authors to a variety of settings, both with and without analytic concerns. See for example \cite{ Deligne2007, Woronowicz1988, MR1010160}.  Of particular interest is the case when $H$ is a semisimple Hopf algebra, $C = \Mod(H)$ is its fusion category of representations, together with forgetful functor to the category of vector spaces. The text \cite{Joyal1991ANIT} contains an account of the applications to mathematics. 
	
	The fiber functor $F$ plays an essential role; without it (given only the \textit{equivalence class} of $\Mod(H)$), one reconstructs a Morita equivalent object not necessarily isomorphic to $H$. Despite the usefulness of Morita theory (including in this work), there is a natural desire to reconstruct the original object up to isomorphism. Meanwhile, in physical contexts, one expects the fiber functor to describe an explicit symmetry breaking process \cite{PhysRevResearch.2.043086}.
	
	A vertical categorification of the reconstruction theorem requires two definitions. First, we need the higher version of a fusion category; the \textit{fusion 2-category} originally defined in \cite{DR18} and subsequently reformulated in \cite{weakfusion}. This notion has already been applied in both  mathematics to solve the minimal nondegenerate extensions problem \cite{JFR23} and in physics to construct examples of (3 + 1)D topological quantum field theories (TQFTs) \cite{PhysRevB.96.045136, MR4239374}.
	
	Second, we need a higher dimensional version of Hopf algebras: \textit{Hopf categories}.  The definition and representation theory of Hopf categories has also already received attention in mathematics  \cite{DAY199799, neuchl, chen2023categorified} and physics, again to construct (3 + 1)D TQFT's \cite{crane1994}. 

	While pieces have appeared \cite{ Schauenburg1992TannakaDF, DAY199799,Pfeiffer_2007,huang2023tannakakrein,}, there has not yet been a full reconstruction theorem for fusion 2-categories.  We prove one as Theorem \ref{2HopfReconstruction}. In order to state the theorem, we use a higher slice category $\ThreeVec/\TwoVec$. Objects of this category are semisimple 2-categories equipped with a functor into $\TwoVec$, 1-morphisms are pairs consisting of a functor and a natural transformation making the obvious triangle commute, and higher morphisms are defined similarly. 
	
	\begin{thm*}
	There is a symmetric monoidal equivalence, contravariant at the level of 1-morphisms, between the full subcategory of $\ThreeVec/\TwoVec$ consisting of locally faithful functors and the 2-category of 2-Hopf Algebras.
	The natural transformations associated to this equivalence reconstruct a semisimple Hopf category from its fusion 2-category of representations and fiber functor, and a fusion 2-category with fiber functor $F$ from the Hopf category $\End(F)$.
	\end{thm*}	
	This proves the 2-categorical case of the  conjecture of Baez and Neuchl \cite[\S1.1]{Baez1995HigherDA}.
	 We remark that with the recent result \cite{DecYu}, which classifies the  fusion 2-categories admitting fiber functors, a complete classification of Hopf categories is possible.
	 One immediate consequence of the work \cite{DecYu} is that all fusion 2-categories admitting a fiber functor to 2-Vec, and thus all finite semisimple Hopf 1-categories are group-theoretical. This situation differs from the case for fusion 1-categories and Hopf algebras, where non group-theoretical examples are known to exist \cite{MR2480712}.  
	The reconstruction theorem likely remains true in far more generality, as it does for 1-categories: beyond the semisimple case, in $\cV$-enriched contexts, for non-coassociative bialgebras corresponding to fiber functors which fail to be monoidal. This last generalization gives ``weak'' Hopf 1-category reconstruction. We confine our interest to the unenriched case with a monoidal functor, in order to avoid many technical difficulties relating to enrichment such as axiomatizing enriched monoidal 2-categories. This has the disadvantage of excluding the interesting and important case when the target is $\Mod(\sVec)$, which we hope to remedy in the near future. 
	
	Section \ref{AgnosticSection} provides the conceptual underpinning of the work, containing an outline of an approach to Tannaka-Krein reconstruction for fusion $n$-categories and highlighting which results are needed to push our method through. Sections \ref{1CategoryReconstructionSection} and \ref{2CategoryReconstructionSection} contain full proofs of the reconstruction theorem using the outlined approach for fusion 1- and 2- categories respectively, with certain details and computations pushed to the appendix. 
	
	 In particular, the proofs of reconstruction for 1-/2-categories we give are both \textit{natural} and \textit{monoidal}, i.e, the reconstruction procedures are the components of natural transformations associated to certain monoidal equivalences of monoidal 1-/2-categories. While this perspective may be independently useful in the future, it has the immediate advantage of avoiding the need to check certain coherence relations. For instance, when reconstructing a Hopf algebra from its category of representations and forgetful functor, the comultiplication on the reconstructed algebra is immediately seen to be both coassociative and an algebra morphism.  \\

	\subsection{Prior work, enrichment, and the many object case}
	This subsection is not intended to be rigorous, but to expand upon an instance of the ``categorical ladder''. The basic reconstruction theorem recalled in Section \ref{1CategoryReconstructionSection} has been generalized across a series of articles \cite{DAY199617, DAY199799, Coalgebroids2000BalancedCP, McCrudden2000CategoriesOR,  Pfeiffer_2007} to the ``many object case enriched in $V$'', where $V$ is any symmetric monoidal 1-category. As $V$ is symmetric monoidal, it is naturally a symmetric monoid in $\Mod(V)$, so that $\Mod(V)/V$ is again symmetric monoidal. 
	
	 Following these authors, we may thus view a Tannaka-Krein style reconstruction theorem as an \textit{equivalence of monoidal categories} from a given subcategory of $\Mod(V)/V$ to a category of algebraic objects. We say the subcategory of $\Mod(V)/V$ \textit{reconstructs} the category of algebraic objects. Then we can summarize the situation in the following list: 
	 \begin{itemize}
		\item The full subcategory of consisting faithful functors reconstructs the category of algebras in $V$.
		\item The full subcategory of consisting of faithful and monoidal functors reconstructs the category of bialgebras in $V$.
		\item $\Mod(V)/V$ reconstructs the category of ``Hopf Algebroids \cite{DAY199799}''.
	 \end{itemize}
	We could add many more items to this list, by adding adjectives to the domain subcategory of $\Mod(V)$ and reconstructed algebraic category in pairs. Examples include rigid/Hopf, braided/quasitriangular and others (including $C^*$).
	
	In general, $\Mod(V)/V$ is a monoidal 2-category, but certain cases of interest (the first and second items above) have the property that the corresponding full subcategory of $\Mod(V)/V$ is 1-truncated. This makes the reconstruction procedure much easier, as less complex axioms capture the full behavior of $\Mod(V)/V$. 
	
	We now climb the categorical ladder to obtain the following list, with $\cV$ a symmetric monoidal 2-category. As before, $\cV$ is a symmetric monoid in $\Mod(\cV)$ which induces a monoidal structure on $\Mod(\cV)/\cV$.
 \begin{itemize}
	\item The full subcategory of consisting locally faithful 2-functors reconstructs the category of algebras in $V$.
	\item The full subcategory of consisting of faithful and monoidal 2-functors reconstructs the category of bialgebras in $V$.
\end{itemize}
 	Here, $\Mod(\cV)/\cV$ is in general a monoidal 3-category. Axiomatizing such objects (or Hopf 2-algebroids) is beyond the scope of this work (the definition as a 1-object tetracategory in the sense of Trimble \cite{trimbleTetra} works fine for objects, but the correct notion of pointed 4-functors and higher cells is currently unknown to the author). However, the subcategories of $\Mod(\cV)/\cV$ corresponding to (bi)algebras are again truncated, and so the existing axiomatization  of symmetric monoidal 2-categories, 2-functors and their associated higher morphisms is sufficient to make the second and third columns above rigorous.
	\subsection*{Acknowledgements}
	I would like to thank Dave Penneys for his guidance and patience, as well as Thibault D\'eccopet, Giovanni Ferrer, Niles Johnson, Theo Johnson-Freyd, and Sean Sanford for many helpful conversations. The author was supported by NSF grants DMS 1547357, DMS 1654159, and DMS 2154389.
	% \textbf{TODO}
	% \begin{itemize}
		%     \item Fix notation for tensorators, should all be $J_F$.
		%     \item Sort out notation more generally, avoid using primes. 
		%     \item Sources for intro. What else to add?
		% \end{itemize}
	\subsection{The Reconstruction Procedure} \label{AgnosticSection}
	Here we provide a ``height-agnostic'' approach to bialgebra/Hopf algebra reconstruction. This section informs the approach of the remainder of this article, but is not required. We let $\cV$ be a semisimple symmetric monoidal $n$-category. By induction, as well as the existence of a universal target for symmetric monoidal functors \cite{JFRUpcoming}, $\cV$ admits a fiber functor and so we may assume that $\cV$ is a subcategory of the category of semisimple modules for an $n$-algebra in the universal target. The cases we consider in this article are $\cV = \Vec$ and $\cV = \TwoVec$, providing two levels of reconstruction. 
	\begin{enumerate}
		\item Form the full sub $(n+1)$-category ${\sf{nSlice}}_\cV$ from the slice $(n+1)$-category $\sf{\Mod(\cV)}/\sf{\cV}$ consisting of faithful $\cV$-functors. This full sub $(n+1)$-category is $n$ truncated, i.e equivalent to an $n$-category. 
		\item Show that an appropriate Deligne tensor product induces a monoidal $\cV$-category structure on ${\sf{nSlice}}_\cV$, equivalent to the natural structure given by the monoid structure of $\cV$ as a module over itself.  
		\item Recognize that the assignment $A \mapsto (\Rep(A), \textbf{Forget}_A)$ is a monoidal $\cV$-functor $\sf{(n-1)Alg}_\cV \to {\sf{nSlice}_\cV}$, contravariant at the level of 1-morphisms.
		\item Show the assignment $(C, F) \to \End(F)$ is a monoidal $n$-functor $\End(-) \colon \sf{nSlice}_\cV\to \sf{Alg}_\cV$, contravariant at the level of 1-morphisms, and that these two functors are inverse equivalences. 
		\item Conclude the category $\cM$ of algebra objects in ${\sf{nSlice}_\cV}$ is equivalent to the category of $n$-bialgebras. 
		\item Show any pair $(C, F)$, where $F$ is a monoidal functor, is canonically an algebra object in ${\sf{nSlice}_\cV}$, so that it's image $\End(F)$ is a bialgebra object (bialgebra reconstruction). 
		\item Construct duality morphisms compatible with the equivalence between $\cM$ and the category of bialgebras, realizing the equivalence between the $n$ categories of fusion $\cV$-module categories with fiber functor and semisimple Hopf $n$-algebras. 
	\end{enumerate}
 Some requirements of the above approach (beyond locating correct \textit{and} manageable definitions of the objects under consideration) are the following:
	\begin{itemize}
	\item Step 1 requires a theory of $n$ categories enriched in a symmetric monoidal $n-1$ category $\cV$, in particular so that $\cV$ module categories are $\cV$-categories. Moreover, for the first part of step 4 to be well defined, the endomorphisms of a $\cV$-functor must again lie in $\cV$.
	\item The content of step 4 in one direction is essentially the $\cV$-Yoneda lemma. The other direction should follow once it is known that $\Mod(-)$ is an equivalence from the Morita $(n +1)$-category of algebras in $\cV$ to $\Mod(\cV)$, together with a categorification of the double centralizer theorem. 
	\end{itemize}
	\subsection{Notation, Conventions and Truncation}
	\begin{itemize}
		\item Objects in a 1-, or 2- category will be denoted $c, c', ...$, always in Roman lowercase.
		\item Morphisms in a 1- or 2- category will be denoted $f, f', ...$, generally in Roman lowercase. 
		\item 1-categories, 1-functors will be denoted by capital roman letters, $C, C',...$ for categories and $F, F', \dots$ for functors. For categories of functors, this rule will take precedence.
		\item 2-morphisms and natural transformations will be generally denoted by lowercase Greek letters. 
		\item 3-morphisms and modifications will be generally denoted by capital Greek letters.
		\item 2-categories and 2-functors will be given capital script lettering $(\cC, \cC', \dots)$. 
		\item We use the oplax convention for natural transformations.
		\item We use the coherence theorem for 2-categories \cite[Theorem 3.6.6]{johnson20202dimensional} to suppress compositors and unitors. 
	\end{itemize}
	 All categories, functors, transformations, and modifications are linear over an algebraically closed field $\bbK$. 
	 We will frequently work with structures in $n-1$ truncated $n$-categories. We verify the axioms for these structures only up to a necessarily unique, invertible, $n$ cell. This includes the composition functors. We will also repeatedly refer to the following objects:
	 \begin{itemize}
	 	\item $\TwoVec$, the 2-category of finite semisimple $\bbK$-linear 1-categories, functors, and natural transformations.
	 	\item $\ThreeVec$, the 3-category of finite semisimple $\bbK$-linear 2-categories, functors, natural transformations, and modifications.
	 	\item $\TwoAlg$, the 2-category of algebra objects in $\TwoVec$, which is equivalent(\cite{1509.06811}) to the 2-category of monoidal categories, monoidal functors, and monoidal natural transformations.  
	 \end{itemize}
	 
	\section{Reconstruction for 1-categories} \label{1CategoryReconstructionSection}
	In this section we provide a proof of the reconstruction theorem for fusion 1-categories in terms of finite dimensional semsimple Hopf algebras. The purpose of this section is to emphasize an approach which is both natural (i.e, proves a statement about the \textit{category} of Hopf algebras) and categorifiable, so as to serve as a reference for the later sections of this paper. 
	\subsection{Monoidal Slice Categories and Coalgebra structure} 
	% \nn{wrote everything here, can move some things to appendix A. Probably also need more subsections}
	We define the 2-category $\Slice$ as the full subcategory of the slice 2-category $\TwoVec/\Vec$ consisting of faithful functors. Unpacked, the category $\Slice$ has:
	% \nn{using ' for different things in the same space, and numbers for similar things in different spaces}. 
	\begin{itemize}
		\item Objects given by pairs $(C, F)$ where $C$ is a finite semisimple linear 1-category and $F \colon C \to \Vec$ is a linear functor, which is injective on $\Hom$-sets (faithful).
		\item Morphisms $(C_1, F_1)$ to $(C_2, F_2)$ are pairs $(T, \tau)$ where $T \colon C_1 \to C_2$ is a linear functor and $\tau \colon F_1 \to F_2 \circ T$ is a natural isomorphism. 
		\item 2-morphisms $(T, \tau) \Rightarrow (T', \tau')$ are natural transformations $\sigma \colon T \to T'$ satisfying the ``ice cream cone" condition: 
		
		% https://q.uiver.app/?q=WzAsNyxbMSwyLCJcXFZlYyJdLFswLDAsIkNfMSJdLFsyLDAsIkNfMiJdLFszLDEsIj0iXSxbNCwwLCJDXzEiXSxbNiwwLCJDXzIiXSxbNSwyLCJcXFZlYyJdLFsxLDAsIkZfMSIsMl0sWzIsMCwiRl8yIl0sWzQsNiwiRl8xIiwyXSxbNSw2LCJGXzIiXSxbMSwyLCJUJyIsMix7ImN1cnZlIjoyfV0sWzEsMiwiVCIsMCx7ImN1cnZlIjotMn1dLFs0LDUsIlQiXSxbMTEsMTIsIlxcc2lnbWEiLDIseyJzaG9ydGVuIjp7InNvdXJjZSI6MjAsInRhcmdldCI6MjB9fV0sWzcsMiwiXFx0YXUnIiwyLHsib2Zmc2V0IjozLCJzaG9ydGVuIjp7InNvdXJjZSI6MjAsInRhcmdldCI6NDB9fV0sWzksNSwiXFx0YXUiLDIseyJzaG9ydGVuIjp7InNvdXJjZSI6MjAsInRhcmdldCI6NDB9fV1d
		\begin{equation}\label{IceCreamCondition1Cat}\begin{tikzcd} 
				{C_1} && {C_2} && {C_1} && {C_2} \\
				&&& {=} \\
				& \Vec &&&& \Vec
				\arrow[""{name=0, anchor=center, inner sep=0}, "{F_1}"', from=1-1, to=3-2]
				\arrow["{F_2}", from=1-3, to=3-2]
				\arrow[""{name=1, anchor=center, inner sep=0}, "{F_1}"', from=1-5, to=3-6]
				\arrow["{F_2}", from=1-7, to=3-6]
				\arrow[""{name=2, anchor=center, inner sep=0}, "{T}"', curve={height=12pt}, from=1-1, to=1-3]
				\arrow[""{name=3, anchor=center, inner sep=0}, "{T'}", curve={height=-12pt}, from=1-1, to=1-3]
				\arrow["{T'}", from=1-5, to=1-7]
				\arrow["\sigma"', shorten <=3pt, shorten >=3pt, Rightarrow, from=2, to=3]
				\arrow["{\tau}"', shift right=3, shorten <=8pt, shorten >=17pt, Rightarrow, from=0, to=1-3]
				\arrow["{\tau'}"', shorten <=8pt, shorten >=17pt, Rightarrow, from=1, to=1-7]
			\end{tikzcd}.
		\end{equation}
		At $c \in C_1$ the above pasting diagram has the equational form
		\begin{equation} \label{IceCreamEquationOnLegs}
			\tau'_c = F_2(\sigma_{c})\tau_c
		\end{equation}
		\item Composition is given by pasting.
	\end{itemize}
	A 1-category is said to be \textit{balanced} if every morphism which is both monic and epic is an isomorphism. All abelian, and thus all semisimple categories, are balanced. 
	\begin{lem}\label{FaithfulBalancedImpliesConservative}
	Any faithful functor with balanced domain reflects isomorphisms, i.e $F(f)$ is an isomorphism if and only if $f$ is.  
	\end{lem}
	\begin{proof}
	Suppose $F(f)$ is an isomorphism. Since $F$ is faithful, $F$ reflects both monics and epics. Therefore $f$ is monic and epic, and as the domain is balanced, $f$ is an isomorphism.  
	\end{proof}
	\begin{lem} \label{1Slice1Truncated}
		$\Slice$ is $1$-truncated, i.e there is at most one 2-morphism $(T, \tau) \Rightarrow (T', \tau')$, and it is invertible if it exists. 
	\end{lem}
	\begin{proof}
		We use notation as in \eqref{IceCreamCondition1Cat}.
		Since $\tau'$ is invertible, we rearrange \eqref{IceCreamEquationOnLegs} to $\tau_c^{-1}\tau'_c = F_2(\sigma_{c})$. Since $F_2$ is faithful and $C_2$ is semisimple, $\sigma_{c}$ is uniquely determined and is moreover an  isomorphism by the previous lemma. 
	\end{proof}
	We observe that $T$ must be faithful as well. With the above lemma in mind, when defining monoidal structures on $\Slice$, we will verify the axioms for a monoidal 1-category (up to isomorphism), and likewise for monoids. The category $\Slice$ is a natural target for the (contraviariant) functor $\Mod(-) \colon \Alg \to \Slice$, defined by $A \mapsto (\Mod(A), \textbf{Forget})$. There is another map in the opposite direction:
	\begin{lem}  The assignments \label{EndFunctorWellDefined}
		\begin{itemize} 
			\item $Q(C, F) \coloneqq \End(F)$
			\item $Q(T, \tau)(\eta_2) \coloneqq \tau^{-1} \circ (\eta_2 \circ T) \circ \tau$
			\item $Q(\sigma) = \id_{Q(T, \tau)}$
		\end{itemize}
		provide a well defined, contravariant functor $Q \colon \Slice \to \sf{Alg}$, the 1-category of algebras. 
	\end{lem}
	\begin{proof}
		We first show that if there exists $\sigma \colon (T, \tau) \Rightarrow (T', \tau')$, then $Q(T, \tau) = Q(T', \tau')$, i.e the functor $Q$ is well defined. We have the following diagram:
		% https://q.uiver.app/?q=WzAsOCxbMCwwLCJGXzEoYykiXSxbMCwyLCJGXzEoYykiXSxbMSwwLCJGXzJUKGMpIl0sWzEsMiwiRl8yVCcoYykiXSxbMiwwLCJGXzJUKGMpIl0sWzMsMCwiRihjKSJdLFszLDIsIkYoYykiXSxbMiwyLCJGXzJUJyhjKSJdLFswLDEsIiIsMCx7ImxldmVsIjoyLCJzdHlsZSI6eyJoZWFkIjp7Im5hbWUiOiJub25lIn19fV0sWzAsMiwiXFx0YXUiXSxbMSwzLCJcXHRhdSciLDJdLFsyLDMsIkZfMihcXHNpZ21hKSIsMl0sWzIsNCwiRl8yKFxcZXRhXzIpIl0sWzQsNSwiXFx0YXVeey0xfSJdLFs1LDYsIiIsMCx7ImxldmVsIjoyLCJzdHlsZSI6eyJoZWFkIjp7Im5hbWUiOiJub25lIn19fV0sWzMsNywiRl8yKFxcZXRhXzIpIiwyXSxbNyw2LCIoXFx0YXUnKV57LTF9IiwyXSxbNCw3LCJGXzIoXFxzaWdtYSkiXV0=
		\[\begin{tikzcd}
			{F_1(c)} & {F_2T(c)} & {F_2T(c)} & {F_1(c)} \\
			\\
			{F_1(c)} & {F_2T'(c)} & {F_2T'(c)} & {F_1(c)}
			\arrow[Rightarrow, no head, from=1-1, to=3-1]
			\arrow["\tau", from=1-1, to=1-2]
			\arrow["{\tau'}"', from=3-1, to=3-2]
			\arrow["{F_2(\sigma)}"', from=1-2, to=3-2]
			\arrow["{F_2(\eta_2)}", from=1-2, to=1-3]
			\arrow["{\tau^{-1}}", from=1-3, to=1-4]
			\arrow[Rightarrow, no head, from=1-4, to=3-4]
			\arrow["{F_2(\eta_2)}"', from=3-2, to=3-3]
			\arrow["{(\tau')^{-1}}"', from=3-3, to=3-4]
			\arrow["{F_2(\sigma)}", from=1-3, to=3-3]
		\end{tikzcd}\]
		The composite along the top row is $Q(T, \tau)$ and the bottom is $Q(T', \tau')$. The outer squares commute by \eqref{IceCreamEquationOnLegs}, and the inner square is the naturality of $\sigma$. That $Q$ preserves identities is clear. Let $(T_1, \tau_1) \colon (C_1, F_1) \to (C_2, F_2)$ and $(T_2, \tau_2) \colon (C_2, F_2) \to C_3, F_3)$. The verification that $Q$ is a contravariant functor to $\Vec$ is: 
		\begin{align*}
			Q(T_1, \tau_1) \circ Q(T_2, \tau_2)(\eta_3) &= \tau_1^{-1} \circ (Q(T_2, \tau_2) \circ T_1) \circ \tau_1 \\
			&= \tau_1^{-1} \circ (T_1 \circ \tau_2)^{-1}\circ ((T_2 \circ T_1) \circ \eta_3) \circ \tau_2 \circ \tau_1 \\
			&= Q(T_1 \circ T_2, \tau_1  \circ T_1\tau_2)
		\end{align*}
		Finally, each morphism $Q(T, \tau)$ is given by conjugation and is therefore an algebra morphism. 
	\end{proof}
	\begin{rem}
	Previous iterations of this approach \cite{Schauenburg1992TannakaDF, Pfeiffer_2007} define $\Slice$ as a 1-category with morphisms given by equivalence classes in lieu of considering $\Alg$ as a locally discrete 2-category. A refinement of the statement that $Q$ is well defined, working with coalgebras instead of algebras appears as Proposition 2.1.2.1 in \cite{rivano1972categories}.
	\end{rem}
\begin{lem}[{\cite[Lemma 2.1.3]{Schauenburg1992TannakaDF}}] \label{EquivalenceInSlice}
	Let $(T, \tau) \in \Slice$. Then $(T, \tau)$ is an isomorphism if and only if $T$ is an equivalence.
\end{lem}
	\begin{lem} \label{1EquivLegs}
		The functors $\Mod(-)$ and $Q$ are inverse equivalences. 
	\end{lem}
\begin{proof}

	We begin by defining the natural isomorphism $\gamma \colon \id_{\sf{Alg}} \to Q \circ \Mod$. Let $A$ be an algebra with forgetful functor $F_A \colon \Mod(A) \to \Vec$. To define a map $A \to \End(F_A)$, for each module $(V, \rho)$ we require a map $A \to \End(V)$. We choose $\rho$, the action of $A$ on $V$. Since $\Mod(A)$ is the Cauchy completion of $\mathbf{B}A$ and $\Vec$ is Cauchy complete, restriction to $\mathbf{B}A$ gives an isomorphism $\End(F_A) \cong \End(F_A |_{\mathbf{B}A}) \cong \Hom_{\mathbf{B}A}(*, -)$. Thus, by the Yoneda Lemma, we have 
	\[
	\End(F_A) \cong \End(Hom_{\mathbf{B}A}(*, -)) \cong A
	\]  
	As constructed, $\gamma$ is manifestly natural. 
	%  We next verify that $\gamma$ is monoidal. Expanding the definition of the composite tensorator, the diagram is 
	% \[
	% \begin{tikzcd}
		%  A \otimes A' \ar[d, "\gamma_A \otimes \gamma_{A'}"] \ar[rr, Rightarrow, no head] & & A \otimes A' \ar[d, "\gamma_{A \otimes A'}"]\\
		%  \End(F_A) \otimes \End(F_{A'}) \ar[r, "J"] & \End(\otimes \circ F_A \boxtimes F_{A'}) \ar[r, "\End(J')"] & \End(F_{A \otimes A'})
		% \end{tikzcd}
	% \]
	% Both composites are equal to  $\rho(a) \otimes  \rho'(a')$, which is easiest to see after inverting the map $\End(J')$. Next, we define the legs of the equivalence $1 \simeq \Mod \circ Q$. 
	Next, we define a natural transformation $\zeta \colon 1 \simeq \Mod \circ Q$ as follows. For $(C, F)$ in $\Slice$, define a functor $\zeta_C \colon C \to \Mod(\End(F))$  as follows: 
	\begin{itemize}
		\item On objects: $c \mapsto (F(c), \End(F)|_c)$; the vector space $F(c)$ with $\End(F)$ action given by taking the leg at $c$.
		\item On morphisms: $c \mapsto Ff$; this intertwines the $\End(F)$ action by definition of $\End(F)$.
	\end{itemize}
	Then $(\zeta_C, =)$ is a map $(C, F) \to (\Mod(\End(F)), \textbf{Forget}_{\End(F)})$ in $\Slice$, which is manifestly natural in $(C, F)$. It remains to verify that $\zeta_C$ is an equivalence of categories, and therefore an equivalence in $\Slice$. Choose an algebra $A$ such that $C = \Mod(A)$, and a bimodule $_AM_\bbC$ such that 
	\[
	\begin{tikzcd}
		\Mod(A) \ar[d, "\sim"] \ar[r, " - \otimes M"] & \bbC -\Mod  \ar[d, "\sim"] \\ C \ar[r, "F"] & \Vec 
	\end{tikzcd}
	\]
	commutes exactly. Since $\Mod(-)$ is an equivalence from the Morita 2-category to $\TwoVec$, we have 
	\[
	\Mod(\End(F))\simeq \Mod(\End_{\Mod(A)}(_A M_\bbC)) = \Mod(\End_{\Mod(A)}(_A M))
	\]
	Since $F$ is a faithful functor, $M$ is a faithful module. The bimodule $_A M_{\End_{\Mod(A)}(M)}$ corresponds to a functor $\Mod(A) \to \Mod(\End_{\Mod(A)}(_A M))$ which is exactly $F \colon C \to \Mod(\End(F))$. 
	
	We claim the bimodule $_A M_{\End_{\Mod(A)}(M)}$ induces a Morita equivalence. This is a consequence of Morita III( see \cite[\S18]{lam2012lectures}), the faithfulness of $M$, and the double centralizer theorem. 
	\end{proof}
	We next give $\Slice$ the structure of a monoidal category such that $Q$ is a monoidal functor to the category of algebras. To do this we will need to recall the 2-universal property of the Deligne tensor product from \cite{décoppet20212deligne}, specialized to the case when the categories involved are semisimple, as well as some further properties. 
	\begin{thm}[{\cite[\S1]{décoppet20212deligne}}]
		% \nn{is the wording here suitably changed from 2Tensor, Def 1.1 to avoid text overlap?} \nn{notation}
		Given $C$ and $D$ two finite semisimple linear categories, there exists a finite semisimple linear category $C \boxtimes D$ and linear functor $\boxtimes \colon C \times D \to C \boxtimes D$ such that precomposition with $\boxtimes$ induces an equivalence 
		\[
		\Hom(C \boxtimes D, E) \simeq \Hom_\textit{bilin}(C \times D, E)
		\]
		for all finite $E$. This equivalence is natural in all three variables. Unpacked, this means:
		\begin{itemize} 
			\item For every finite $E$ and bilinear bifunctor $F \colon C \times D $ there exists a functor $\bar F \colon C \boxtimes D \to E$ and natural isomorphism $u \colon \bar F \circ \boxtimes \Rightarrow F$.
			\item For every two functors $G, H \colon C \boxtimes D \to E$ and natural transformation $t \colon G \circ \boxtimes \Rightarrow H \circ \boxtimes$, there exists a \textit{unique} natural transformation $t' \colon G \to H$ such that $t' \circ \boxtimes = t$.
		\end{itemize}
		Furthermore: 
		\begin{itemize}
			\item If $C$ and $D$ are monoidal, then so is $C \boxtimes D$. With this monoidal structure, the functor $\boxtimes$ is monoidal. 
			\item If $F \colon C \times D \to E$ is monoidal, then so is the induced functor $\bar F \colon C \boxtimes D \to E$
		\end{itemize}
		Finally, for algebras $A, B$, then 
		\[
		\Mod(A) \boxtimes \Mod(B) \simeq \Mod(A \otimes B)
		\]
		and this equivalence is natural in $A$ in $B$.
	\end{thm}
	\begin{construction} \label{Monoidal1SliceConstruction}
		Let $(C, F)$ and $(C', F')$ be objects in $\Slice$. We define their tensor product $\boxtimes$ as 
		\[
		(C, F) \boxtimes (C', F') := (C \boxtimes C', \otimes_\Vec \circ (F \boxtimes F')).
		\]
		This assignment extends to a functor $\Slice \times \Slice \to \Slice$ by the assignment:
		\[
		(T, \tau) \boxtimes (T', \tau') := (T \boxtimes T', \otimes_\Vec \circ (\tau \boxtimes \tau')).
		\]
		The monoidal unit is $\id_\Vec$. To define the associator we will add additional data to the associator on $\TwoVec$, which is essentially the Cartesian monoidal structure. We need a $2$-morphism $\sim$ as below:
		% https://q.uiver.app/?q=WzAsMyxbMCwwLCIoQyBcXGJveHRpbWVzIEMnKSBcXGJveHRpbWVzIEMnJyJdLFsxLDIsIlxcVmVjIl0sWzIsMCwiQyBcXGJveHRpbWVzIChDJyBcXGJveHRpbWVzIEMnJykiXSxbMCwyLCJcXGFscGhhX1xcVHdvVmVjIl0sWzAsMSwiXFxvdGltZXMgXFxjaXJjICgoXFxvdGltZXMgXFxjaXJjIChGIFxcYm94dGltZXMgRicpKSBcXGJveHRpbWVzIEYnJykiLDJdLFsyLDEsIlxcb3RpbWVzIFxcY2lyYyAoRiBcXGJveHRpbWVzIChcXG90aW1lcyBcXGNpcmMgKEYnIFxcYm94dGltZXMgRicnKSkpIl0sWzQsMiwiXFxzaW0gIiwyLHsic2hvcnRlbiI6eyJzb3VyY2UiOjIwLCJ0YXJnZXQiOjMwfX1dXQ==
		\[\begin{tikzcd}
			{(C \boxtimes C') \boxtimes C''} && {C \boxtimes (C' \boxtimes C'')} \\
			\\
			& \Vec
			\arrow["{\alpha_\TwoVec}", from=1-1, to=1-3]
			\arrow[""{name=0, anchor=center, inner sep=0}, "{\otimes \circ ((\otimes \circ (F \boxtimes F')) \boxtimes F'')}"', from=1-1, to=3-2]
			\arrow["{\otimes \circ (F \boxtimes (\otimes \circ (F' \boxtimes F'')))}", from=1-3, to=3-2]
			\arrow["{\sim }"', shorten <=11pt, shorten >=16pt, Rightarrow, from=0, to=1-3]
		\end{tikzcd}.\]
		After expanding and rearranging the tensor factors on the functors into $\Vec$, we may define the associator using the 2-naturality of the Deligne tensor product as the morphism: 
		% https://q.uiver.app/?q=WzAsNyxbMCwwLCIoQyBcXGJveHRpbWVzIEMnKSBcXGJveHRpbWVzIEMnJyJdLFsyLDAsIkMgXFxib3h0aW1lcyAoQycgXFxib3h0aW1lcyBDJycpIl0sWzAsMSwiKFxcVmVjIFxcYm94dGltZXMgXFxWZWMpIFxcYm94dGltZXMgXFxWZWMiXSxbMiwxLCJcXFZlYyBcXGJveHRpbWVzIChcXFZlYyBcXGJveHRpbWVzIFxcVmVjKSJdLFswLDIsIlxcVmVjIFxcYm94dGltZXMgXFxWZWMiXSxbMiwyLCJcXFZlYyBcXGJveHRpbWVzIFxcVmVjIl0sWzEsMywiXFxWZWMiXSxbMCwxLCJcXGFscGhhX1xcVHdvVmVjIl0sWzAsMiwiKEYgXFxib3h0aW1lcyBGJykgXFxib3h0aW1lcyBGJyIsMl0sWzIsMywiXFxhbHBoYV9cXFR3b1ZlYyIsMl0sWzEsMywiRiBcXGJveHRpbWVzIChGJyBcXGJveHRpbWVzIEYnJykiXSxbMiwxLCJcXGFscGhhX1xcVHdvVmVjIiwwLHsic2hvcnRlbiI6eyJzb3VyY2UiOjIwLCJ0YXJnZXQiOjMwfSwibGV2ZWwiOjJ9XSxbMiw0LCJcXG90aW1lcyBcXGJveHRpbWVzIDEiLDJdLFszLDUsIlxcb3RpbWVzIFxcYm94dGltZXMxIl0sWzQsNiwiXFxvdGltZXMiLDJdLFs1LDYsIlxcb3RpbWVzIl0sWzQsMywiXFxhbHBoYV9cXFZlYyIsMix7InNob3J0ZW4iOnsic291cmNlIjoyMCwidGFyZ2V0IjoyMH0sImxldmVsIjoyfV1d
		\begin{equation}
			\label{1SliceAssociator}\begin{tikzcd} 
				{(C \boxtimes C') \boxtimes C''} && {C \boxtimes (C' \boxtimes C'')} \\
				{(\Vec \boxtimes \Vec) \boxtimes \Vec} && {\Vec \boxtimes (\Vec \boxtimes \Vec)} \\
				{\Vec \boxtimes \Vec} && {\Vec \boxtimes \Vec} \\
				& \Vec
				\arrow["{\alpha_\TwoVec}", from=1-1, to=1-3]
				\arrow["{(F \boxtimes F') \boxtimes F'}"', from=1-1, to=2-1]
				\arrow["{\alpha_\TwoVec}"', from=2-1, to=2-3]
				\arrow["{F \boxtimes (F' \boxtimes F'')}", from=1-3, to=2-3]
				\arrow["{\alpha_\TwoVec}", shorten <=13pt, shorten >=20pt, Rightarrow, from=2-1, to=1-3]
				\arrow["{\otimes \boxtimes 1}"', from=2-1, to=3-1]
				\arrow["{1\boxtimes \otimes}", from=2-3, to=3-3]
				\arrow["\otimes"', from=3-1, to=4-2]
				\arrow["\otimes", from=3-3, to=4-2]
				\arrow["{\alpha_\Vec}"', shorten <=14pt, shorten >=14pt, Rightarrow, from=3-1, to=2-3]
			\end{tikzcd}.
		\end{equation}
		We will verify the pentagon axiom somewhat indirectly, to best mirror the approach for 2-categories.  Because the functors $Q \colon \Slice \to \Alg$ and $A \mapsto (\Mod(A), \textbf{Forget}_A)$ are equivalences, we know that the functor $\tilde \boxtimes$ given by 
		\[
		(C, F) \tilde \boxtimes (C', F') := (\Mod(\End(F) \otimes \End(F')), \textbf{Forget}_{\End(F) \otimes \End(F')})
		\]
		defines a symmetric monoidal structure on $\Slice$. There is a natural isomorphism 
			\[
			(C \boxtimes C', \otimes_\Vec \circ (F \boxtimes F')) \to  (\Mod(\End(F) \otimes \End(F')), \textbf{Forget}_{\End(F) \otimes \End(F')})
			\]
		 given by $c \boxtimes c' \mapsto F(c) \otimes F'(c')$ with the evident $\End(F) \otimes \End(F')$ action.
			This evidently natural map, together with the identity component natural transformation, is an equivalence in $\Slice$ by Lemmas \ref{EquivalenceInSlice} and \ref{1EquivLegs}. 
			Transporting the morphism \eqref{1SliceAssociator} across this equivalence results in the associator for $\TwoVec$, by construction.  As a consequence, both $Q$ and $\Mod$ may be  enhanced to symmetric monoidal equivalences. 
	\end{construction}
	\begin{rem*} 
		An important feature of the above construction is that $\alpha_{\TwoVec} \colon (\Vec \boxtimes \Vec) \boxtimes \Vec \to \Vec \boxtimes (\Vec \boxtimes \Vec)$ is the induced morphism from the associator on $\Vec$, and likewise for the structures on the monoidal unit $\id_\Vec$. 
	\end{rem*}
	\begin{rem*} \label{InducedFromHigherMonoidalStructure}
		This monoidal structure on $\Slice$ can alternatively be obtained as follows. Consider the monoidal 2-category structure on $\TwoVec$. As $\Vec$ is a monoid object in $\TwoVec$, the 2-category $\TwoVec/\Vec$ is monoidal. 
	\end{rem*}
	\begin{cor} \label{1CategoriesOfMonoidsEquivalent}
		$Q$ induces a symmetric monoidal equivalence between the category of monoids in $\Slice$ and the category of bialgebras.
	\end{cor}
%	\begin{rem*}
%		Let $\cM$ be the 2-category of monoids in $\TwoVec$. This category is known \cite{1509.06811} to be $\sf{MonVec}$. We see that the subcategory of $\sf{MonVec}/\Vec$ consisting of pairs $(C, F)$ where both $C$ and $F$ are monoidal and $F$ is faithful, with 1-morphisms consisting of pairs $(T, \tau)$ where again both $T$ and $\tau$ are monoidal is precisely the slice category $\Slice$.
%	\end{rem*}
	\begin{prop} \label{1MonoidIn1Slice}
		Every monoidal functor $(C, F)$ in $\Slice$ is canonically a monoid, and with respect to this structure, every morphism $(T, \tau)$ where $T$ and $\tau$ are monoidal is a monoid homomorphism. Additionally, every monoid in $\Slice$ is isomorphic to one of this form.  
	\end{prop}
	\begin{proof}
		The final part of this lemma is a consequence of Corollary \ref{1CategoriesOfMonoidsEquivalent}. We let $\chi$ be the tensorator for $F = (F, \chi)$. The multiplication $\mu^{(F, \chi)}$ and unit $\iota^{(F, \chi)}$ are given by the morphisms (induced by universal property of the Deligne tensor product): 
		% https://q.uiver.app/?q=WzAsOCxbMCwxLCJcXG11XnsoRiwgSil9PSJdLFsxLDAsIkNcXGJveHRpbWVzIEMiXSxbMywwLCJDIl0sWzIsMiwiXFxWZWMiXSxbNCwxLCJcXGlvdGFeeyhGLEopfSA9Il0sWzUsMCwiXFxWZWMiXSxbNiwyLCJcXFZlYyJdLFs3LDAsIkMiXSxbMSwyLCJcXG90aW1lcyJdLFsyLDMsIkYiXSxbMSwzLCJcXG90aW1lcyBcXGNpcmMgKEZcXGJveHRpbWVzIEYpIiwyXSxbNSw3LCJcXGJiQyBcXG1hcHN0byAxIl0sWzcsNiwiRiJdLFs1LDYsIiIsMix7ImxldmVsIjoyLCJzdHlsZSI6eyJoZWFkIjp7Im5hbWUiOiJub25lIn19fV0sWzEwLDIsIkoiLDIseyJzaG9ydGVuIjp7InNvdXJjZSI6MjAsInRhcmdldCI6MjB9fV0sWzEzLDcsIlxcaW90YV9GIiwyLHsic2hvcnRlbiI6eyJzb3VyY2UiOjIwLCJ0YXJnZXQiOjMwfX1dXQ==
		\[\begin{tikzcd}
			& {C\boxtimes C} && C && \Vec && C \\
			{\mu^{(F, \chi)}=} &&&& {\iota^{(F,\chi)} =} \\
			&& \Vec &&&& \Vec
			\arrow["\otimes", from=1-2, to=1-4]
			\arrow["F", from=1-4, to=3-3]
			\arrow[""{name=0, anchor=center, inner sep=0}, "{\otimes \circ (F\boxtimes F)}"', from=1-2, to=3-3]
			\arrow["{\bbC \mapsto 1}", from=1-6, to=1-8]
			\arrow["F", from=1-8, to=3-7]
			\arrow[""{name=1, anchor=center, inner sep=0}, Rightarrow, no head, from=1-6, to=3-7]
			\arrow["\chi"', shorten <=9pt, shorten >=9pt, Rightarrow, from=0, to=1-4]
			\arrow["{\iota_F}"', shorten <=8pt, shorten >=13pt, Rightarrow, from=1, to=1-8]
		\end{tikzcd}\]
		Expanding composites as in \eqref{1SliceAssociator}, the shorter composite of the pentagon axiom for $(C, F)$ is: 
		% https://q.uiver.app/?q=WzAsOCxbMCwwLCIoQyBcXGJveHRpbWVzIEMpIFxcYm94dGltZXMgQyJdLFszLDMsIlxcVmVjIl0sWzMsMCwiQyBcXGJveHRpbWVzIEMiXSxbNSwwLCJDIl0sWzMsMSwiXFxWZWMgXFxib3h0aW1lcyBcXFZlYyJdLFsxLDEsIihcXFZlYyBcXGJveHRpbWVzIFxcVmVjKSBcXGJveHRpbWVzXFxWZWMiXSxbMiwyLCJcXFZlYyBcXGJveHRpbWVzIFxcVmVjIl0sWzUsMV0sWzAsMiwiXFxvdGltZXMgXFxib3h0aW1lcyAxIl0sWzIsMywiXFxvdGltZXMiXSxbMywxLCJGIl0sWzIsNCwiRiBcXGJveHRpbWVzIEYiXSxbNCwxLCJcXG90aW1lcyJdLFswLDUsIihGIFxcYm94dGltZXMgRikgXFxib3h0aW1lcyBGIiwyLHsibGFiZWxfcG9zaXRpb24iOjIwfV0sWzUsNiwiXFxvdGltZXMgXFxib3h0aW1lcyAxIiwyXSxbNiwxLCJcXG90aW1lcyIsMl0sWzQsNywiSiIsMCx7ImxhYmVsX3Bvc2l0aW9uIjoyMCwic2hvcnRlbiI6eyJzb3VyY2UiOjEwLCJ0YXJnZXQiOjYwfSwibGV2ZWwiOjJ9XSxbNSw0LCJcXG90aW1lcyBcXGNpcmMgKEogXFxib3h0aW1lcyAxKSIsMCx7InNob3J0ZW4iOnsic291cmNlIjo0MCwidGFyZ2V0Ijo0MH0sImxldmVsIjoyfV1d
		\[\begin{tikzcd}
			{(C \boxtimes C) \boxtimes C} &&& {C \boxtimes C} && C \\
			& {(\Vec \boxtimes \Vec) \boxtimes\Vec} && {\Vec \boxtimes \Vec} && {} \\
			&& {\Vec \boxtimes \Vec} \\
			&&& \Vec
			\arrow["{\otimes \boxtimes 1}", from=1-1, to=1-4]
			\arrow["\otimes", from=1-4, to=1-6]
			\arrow["F", from=1-6, to=4-4]
			\arrow["{F \boxtimes F}", from=1-4, to=2-4]
			\arrow["\otimes", from=2-4, to=4-4]
			\arrow["{(F \boxtimes F) \boxtimes F}"'{pos=0.2}, from=1-1, to=2-2]
			\arrow["{\otimes \boxtimes 1}"', from=2-2, to=3-3]
			\arrow["\otimes"', from=3-3, to=4-4]
			\arrow["\chi"{pos=0.2}, shorten <=5pt, shorten >=28pt, Rightarrow, from=2-4, to=2-6]
			\arrow["{\otimes \circ (\chi \boxtimes 1)}", shorten <=29pt, shorten >=29pt, Rightarrow, from=2-2, to=2-4]
		\end{tikzcd}\]
		and the longer composite is: 
		% https://q.uiver.app/?q=WzAsMTIsWzAsMCwiKEMgXFxib3h0aW1lcyBDKSBcXGJveHRpbWVzIEMiXSxbMiwwLCJDIFxcYm94dGltZXMgKEMgXFxib3h0aW1lcyBDKSJdLFswLDEsIihcXFZlYyBcXGJveHRpbWVzIFxcVmVjKSBcXGJveHRpbWVzIFxcVmVjIl0sWzIsMSwiXFxWZWMgXFxib3h0aW1lcyAoXFxWZWMgXFxib3h0aW1lcyBcXFZlYykiXSxbMCwyLCJcXFZlYyBcXGJveHRpbWVzIFxcVmVjIl0sWzIsMiwiXFxWZWMgXFxib3h0aW1lcyBcXFZlYyJdLFsyLDQsIlxcVmVjIl0sWzUsMCwiQyBcXGJveHRpbWVzIEMiXSxbNywwLCJDIl0sWzMsMiwiXFxWZWMiXSxbNCwxXSxbNiwxXSxbMCwxLCJcXGFscGhhX1xcVHdvVmVjIl0sWzAsMiwiKEYgXFxib3h0aW1lcyBGKSBcXGJveHRpbWVzIEYiLDJdLFsyLDMsIlxcYWxwaGFfXFxUd29WZWMiLDJdLFsxLDMsIkYgXFxib3h0aW1lcyAoRiBcXGJveHRpbWVzIEYpIl0sWzIsNCwiXFxvdGltZXMgXFxib3h0aW1lcyAxIiwyXSxbMyw1LCIxIFxcYm94dGltZXMgXFxvdGltZXMiXSxbNCw2LCJcXG90aW1lcyIsMl0sWzUsNiwiXFxvdGltZXMiXSxbMSw3LCIxIFxcYm94dGltZXMgXFxvdGltZXMiXSxbNyw4LCJcXG90aW1lcyJdLFs5LDYsIlxcb3RpbWVzIl0sWzgsNiwiXFxvdGltZXMiLDAseyJjdXJ2ZSI6LTV9XSxbNyw5LCJGIFxcYm94dGltZXMgRiJdLFs0LDUsIlxcYWxwaGFfXFxWZWMiLDIseyJzaG9ydGVuIjp7InNvdXJjZSI6MjAsInRhcmdldCI6MjB9LCJsZXZlbCI6Mn1dLFsxMCwxMSwiSiIsMix7InNob3J0ZW4iOnsic291cmNlIjo0MCwidGFyZ2V0Ijo0MH0sImxldmVsIjoyfV0sWzMsMTAsIlxcb3RpbWVzIFxcY2lyYyAoMSBcXGJveHRpbWVzIEopIiwyLHsic2hvcnRlbiI6eyJzb3VyY2UiOjQwLCJ0YXJnZXQiOjQwfSwibGV2ZWwiOjJ9XSxbMTMsMTUsIlxcYWxwaGFfXFxUd29WZWMiLDAseyJzaG9ydGVuIjp7InNvdXJjZSI6MjAsInRhcmdldCI6MzB9fV1d
		\[\begin{tikzcd}
			{(C \boxtimes C) \boxtimes C} && {C \boxtimes (C \boxtimes C)} &&& {C \boxtimes C} && C \\
			{(\Vec \boxtimes \Vec) \boxtimes \Vec} && {\Vec \boxtimes (\Vec \boxtimes \Vec)} && {} && {} \\
			{\Vec \boxtimes \Vec} && {\Vec \boxtimes \Vec} & \Vec \\
			\\
			&& \Vec
			\arrow["{\alpha_\TwoVec}", from=1-1, to=1-3]
			\arrow[""{name=0, anchor=center, inner sep=0}, "{(F \boxtimes F) \boxtimes F}"', from=1-1, to=2-1]
			\arrow["{\alpha_\TwoVec}"', from=2-1, to=2-3]
			\arrow[""{name=1, anchor=center, inner sep=0}, "{F \boxtimes (F \boxtimes F)}", from=1-3, to=2-3]
			\arrow["{\otimes \boxtimes 1}"', from=2-1, to=3-1]
			\arrow["{1 \boxtimes \otimes}", from=2-3, to=3-3]
			\arrow["\otimes"', from=3-1, to=5-3]
			\arrow["\otimes", from=3-3, to=5-3]
			\arrow["{1 \boxtimes \otimes}", from=1-3, to=1-6]
			\arrow["\otimes", from=1-6, to=1-8]
			\arrow["\otimes", from=3-4, to=5-3]
			\arrow["\otimes", curve={height=-30pt}, from=1-8, to=5-3]
			\arrow["{F \boxtimes F}", from=1-6, to=3-4]
			\arrow["{\alpha_\Vec}"', shorten <=15pt, shorten >=15pt, Rightarrow, from=3-1, to=3-3]
			\arrow["\chi"', shorten <=22pt, shorten >=22pt, Rightarrow, from=2-5, to=2-7]
			\arrow["{\otimes \circ (1 \boxtimes \chi)}"', shorten <=19pt, shorten >=19pt, Rightarrow, from=2-3, to=2-5]
			\arrow["{\alpha_\TwoVec}", shorten <=25pt, shorten >=37pt, Rightarrow, from=0, to=1]
		\end{tikzcd}\]
		We show these are isomorphic morphisms in $\Slice$. The requisite natural isomorphism $\otimes \circ (\otimes \boxtimes 1) \simeq \otimes \circ (1 \boxtimes \otimes) \circ \alpha_\TwoVec$ is given by $\alpha_C$. Then the ice-cream cone condition \eqref{IceCreamCondition1Cat} is precisely the hexagon axiom for $F$. 
		% While this is straightforward to verify directly, the coherence theorem for monoidal functors implies that any pasting diagram built from $F$ and the structure morphisms from the domain and target has a unique value. 
		The unit axioms are proven in an identical way. Now let $(T, \tau) \colon (C, F) \to (C', F')$ be a morphism where both $T$ and $\tau$ are monoidal. We show $(T, \tau)$ is a monoid homomorphism. We need to show the following morphisms are isomorphic in $\Slice$:
		% https://q.uiver.app/?q=WzAsOSxbNSwxLCJcXHRleHR7YW5kfSJdLFsyLDIsIlxcVmVjIl0sWzAsMCwiQyBcXGJveHRpbWVzIEMiXSxbMiwwLCJDIl0sWzQsMCwiQyciXSxbNiwwLCJDIFxcYm94dGltZXMgQyJdLFs4LDAsIkMnIFxcYm94dGltZXMgQyciXSxbMTAsMCwiQyciXSxbOCwyLCJcXFZlYyJdLFsyLDEsIlxcb3RpbWVzIFxcY2lyYyAoRiBcXGJveHRpbWVzIEYpIiwyXSxbMywxLCJGIiwyLHsibGFiZWxfcG9zaXRpb24iOjcwfV0sWzQsMSwiRiciXSxbMiwzLCJcXG90aW1lcyJdLFszLDQsIlQiXSxbNSw2LCJUIFxcYm94dGltZXMgVCJdLFs2LDcsIlxcb3RpbWVzIl0sWzUsOCwiXFxvdGltZXMgXFxjaXJjIChGIFxcYm94dGltZXMgRikiLDJdLFs2LDgsIlxcb3RpbWVzIFxcY2lyYyAoRicgXFxib3h0aW1lcyBGJykiLDEseyJsYWJlbF9wb3NpdGlvbiI6NzB9XSxbNyw4LCJGJyJdLFsxMCw0LCJcXHRhdSIsMix7InNob3J0ZW4iOnsic291cmNlIjoyMH19XSxbOSwzLCJKIiwyLHsic2hvcnRlbiI6eyJzb3VyY2UiOjIwfX1dLFsxNiw2LCJcXG90aW1lcyBcXGNpcmMgKFxcdGF1IFxcYm94dGltZXMgXFx0YXUpIiwwLHsic2hvcnRlbiI6eyJzb3VyY2UiOjIwfX1dLFsxNyw3LCJKJyIsMix7InNob3J0ZW4iOnsic291cmNlIjoyMH19XV0=
		\[\begin{tikzcd}
			{C \boxtimes C} && C && {C'} && {C \boxtimes C} && {C' \boxtimes C'} && {C'} \\
			&&&&& {\text{and}} \\
			&& \Vec &&&&&& \Vec
			\arrow[""{name=0, anchor=center, inner sep=0}, "{\otimes \circ (F \boxtimes F)}"', from=1-1, to=3-3]
			\arrow[""{name=1, anchor=center, inner sep=0}, "F"'{pos=0.7}, from=1-3, to=3-3]
			\arrow["{F'}", from=1-5, to=3-3]
			\arrow["\otimes", from=1-1, to=1-3]
			\arrow["T", from=1-3, to=1-5]
			\arrow["{T \boxtimes T}", from=1-7, to=1-9]
			\arrow["\otimes", from=1-9, to=1-11]
			\arrow[""{name=2, anchor=center, inner sep=0}, "{\otimes \circ (F \boxtimes F)}"', from=1-7, to=3-9]
			\arrow[""{name=3, anchor=center, inner sep=0}, "{\otimes \circ (F' \boxtimes F')}"{description, pos=0.7}, from=1-9, to=3-9]
			\arrow["{F'}", from=1-11, to=3-9]
			\arrow["\tau"', shorten <=11pt, Rightarrow, from=1, to=1-5]
			\arrow["\chi"', shorten <=6pt, Rightarrow, from=0, to=1-3]
			\arrow["{\otimes \circ (\tau \boxtimes \tau)}", shorten <=6pt, Rightarrow, from=2, to=1-9]
			\arrow["{\chi'}"', shorten <=12pt, Rightarrow, from=3, to=1-11]
		\end{tikzcd}\]
		The requisite natural isomorphism $\otimes \circ (T \boxtimes T) \simeq T \circ \otimes$ is precisely the tensorator of $T$, and monoidality of $\tau$ gives the ice cream cone condition. 
	\end{proof}
	% \begin{rem*}
		%     The above lemma \nn{is similar to? is? }a result of Day \nn{cite} about monoids in the the monoidal category $\Hom(\cC, \cV)$, where $\cC$ is a $\cV$ category and the monoidal product is given by Day Convolution. We include it here for reference later. 
		% \end{rem*}
	\begin{cor}[Bialgebra Reconstruction Theorem]
		The functors $(C, F) \mapsto \End(F)$ and $A \to (\Mod(A), \textbf{Forget}_A)$ are contravariant monoidal equivalences between the category of semisimple bialgebras and the category of monoids in $\Slice$. Here, the monoidal structure on $\Mod(A)$ is the standard one induced from the comultiplication, so that the forgetful functor is monoidal. 
	\end{cor}

	% \begin{rem*}
		%     Many \nn{all?} of the results of this section remain valid when $\Vec$ is replaced by any monoid object in a monoidal 2-category. \nn{figure out details, cite day and street} \nn{save for thesis?}
		% \end{rem*}
	\subsection{Duals}
	Finally, we enhance the above to the following: 
	
	\begin{thm}\label{1BialgReconstruction}
		The functors $(C, F) \mapsto \End(F)$ and $A \to (\Mod(A), \textbf{Forget}_A)$ are contravariant monoidal equivalences between the category of semisimple Hopf algebras and $\sf{Fus1Cat}/\fdVec$.
	\end{thm}
	\begin{proof}
		The uniqueness properties of duals and antipodes imply that we need only show the existence of these structures. If $H$ is a finite semisimple Hopf algebra with antipode $S$, it is well known that $\Mod(H)$ has left and right duals. Let $\rho \colon H \to \End(V)$ be a representation. Then the representations $^*\rho$ and $\rho^*$ defined on $\End(^*V)$ and $\End(V^*)$ respectively are given by
		\begin{align*}
			\rho^* &= ~ ~(-)^*  \circ \rho \circ S\\
			^*\rho &= ~ ^*(-) \circ \rho \circ S^{-1},
		\end{align*}using the left and right dual functors of $\Vec$. The antialgebra homomorphism properties of $S$ and $(-)^*$ conspire to ensure these maps are algebra homomorphisms. Now, given a pair $(C, F)$ where $C$ has left and right dual functors , we use the left dual functor to induce a morphism in $\Slice$: 
		
		% https://q.uiver.app/?q=WzAsMyxbMCwwLCJDXnswLDFvcH0iXSxbMiwwLCJDIl0sWzEsMiwiXFxWZWMiXSxbMCwxLCJ4IFxcbWFwc3RvIHheKiJdLFsxLDIsIkYiXSxbMCwyLCIqRiIsMl0sWzUsMSwiXFxzaW0iLDIseyJzaG9ydGVuIjp7InNvdXJjZSI6MjB9fV1d
		\[\begin{tikzcd}
			{C^{0,1\text{op}}} && C \\
			\\
			& \Vec
			\arrow["{c \mapsto ^*c }", from=1-1, to=1-3]
			\arrow["F", from=1-3, to=3-2]
			\arrow[""{name=0, anchor=center, inner sep=0}, "{^*F}"', from=1-1, to=3-2]
			\arrow["\delta"', shorten <=8pt, Rightarrow, from=0, to=1-3]
		\end{tikzcd}\]
		from which we obtain a map $\End(F)  \to \End( ^*(-) \circ F)$. The 2-morphism $\delta$ is the canonical equivalence $F(^*c) \simeq ^*F(c^*)$. Whiskering with the right dual functor and using the isomorphism $(^*(-))^* \circ \cF \Rightarrow \cF$ we obtain an anti(co)algebra homomorphism $S \colon \End(F) \to \End(F)$, which has the formula
		$$S(\eta)_{c} = (\delta^{-1}\eta_{^*c}\delta)^*$$
		We expand slightly on \cite[Prop. 5.3.1]{MR3242743} for the verification that this assignment satisfies the antipode axiom. We will represent the morphism $\delta$ by $\bullet$, and make no graphical distinction between $\delta$ and $\delta^{-1}$. No confusion is possible, since only one of the morphisms $\delta, \delta^{-1}$ will type check.
		We claim
		\begin{equation} \label{1AntipodeAxiomSnake}
			(\mu \circ (1 \otimes S)\circ \Delta)(\eta) = 
			\tikzmath{
				\draw[thick] (0,-1) -- (0, 1.5)  arc (180:0:.5cm) -- (1,.5)  arc (-180:0:.5cm) -- (2, 2.666);
				% additional draw args, center, radius, additional left x-space, additional right x-space, contents
				\roundNbox{fill=white}{(.5,1)}{.4}{.5}{.5}{${\Delta(\eta)_{c \otimes {}^*c}}$};
				\filldraw[black] (1,1.5) circle (1.5pt);
				\filldraw[black] (1,.5) circle (1.5pt);
			},    
		\end{equation}
		%=(\ev_X \otimes \id_{F(X)}) \circ ((\eta_X \otimes \eta_{X^*}) \otimes \id_{F(X)} \circ \alpha _C \circ \(\id_X \circ \coev_{F(X)}
		which is clear upon evaluating $(\mu \circ (1 \otimes S))(\eta_{(1)} \otimes \eta_{(2)})$, using Sweedler notation. Additionally, we have the following equation relating $F(\ev_c)$  with $\ev_{F(c)}$, up to an instance of the isomorphism $\varphi \colon 1 \to F(1)$, which is henceforth suppressed. 
		
		\begin{equation} \label{1CatDualReplacement}
			\tikzmath{
				\draw[thick] (0, .5) -- (0, 1.5)  arc (180:0:.5cm) -- (1,.5);
				% additional draw args, center, radius, additional left x-space, additional right x-space, contents
				\filldraw[black] (1,1.5) circle (1.5pt);
			} \quad  =\quad  
			\tikzmath{
				\draw[thick] (0, .5) -- (0, 1);.5);
				\draw[thick] (.3, 1) -- (.3, 1.5)  arc (180:0:.2cm) -- (.7,1);
				\draw[thick] (1,1) -- (1,.5);
				% additional draw args, center, radius, additional left x-space, additional right x-space, contents
				%\draw (.2, 1.4) rectangle (.8, 1.8);
				\roundNbox{fill=white}{(.5,1)}{.3}{.4}{.4}{${J_{c, ^*c}}$};
				\roundNbox{fill=white}{(.5,1.85)}{.3}{.4}{.4}{${F(\ev_X)}$};  
			}
		\end{equation}
		Separated strings are tensored in the target while close strings were tensored in the domain. In particular, $J$ brings two separated strings together. 		
		The verification of the antipode axiom is the following graphical manipulation, valid by naturality of $\eta$. Here, we omit subscripts on $J, \eta, \ev$ to avoid clutter.
 		\begin{equation} \label{1AntipodeAxiomVerification}
			%\hspace{-5em}
			\tikzmath{
				\draw[thick] (0,-1) -- (0, 1.5)  arc (180:0:.5cm) -- (1,.5)  arc (-180:0:.5cm) -- (2, 2.666);
				% additional draw args, center, radius, additional left x-space, additional right x-space, contents
				\roundNbox{fill=white}{(.5,1)}{.4}{.5}{.5}{${\Delta(\eta)}$};
				\filldraw[black] (1,1.5) circle (1.5pt);
				\filldraw[black] (1,.5) circle (1.5pt);
			}  
			~~ \underset{\text{def}, \eqref{1CatDualReplacement}}{=} ~~ 
			\tikzmath{
				\draw[thick] (0,-1) -- (0, .4);
				\draw[thick] (0, -1) -- (0, .4);
				\draw[thick] (.3, .4) -- (.3, 1.5)  arc (180:0:.2cm) -- (.7,.4);
				\draw[thick] (1,.4) -- (1,.0)  arc (-180:0:.5cm) -- (2, 2.666);
				% additional draw args, center, radius, additional left x-space, additional right x-space, contents
				%\draw (.2, 1.4) rectangle (.8, 1.8);
				\roundNbox{fill=white}{(.5,1)}{.2}{.5}{.5}{${\eta}$};
				\roundNbox{fill=white}{(.5,.4)}{.2}{.5}{.5}{${J}$};
				\roundNbox{fill=white}{(.5, 1.6)}{.3}{.4}{.3}{${F(\ev)}$};
				\filldraw[black] (1,0) circle (1.5pt);
			} 
			~~ \underset{\text{nat.~} \eta}{=} ~~ 
			\tikzmath{
				\draw[thick] (0,-1) -- (0, .4);
				\draw[thick] (0, -1) -- (0, .4);
				\draw[thick] (.3, .4) -- (.3, 1.5)  arc (180:0:.2cm) -- (.7,.4);
				\draw[thick] (1,.4) -- (1,.0)  arc (-180:0:.5cm) -- (2, 2.666);
				% additional draw args, center, radius, additional left x-space, additional right x-space, contents
				%\draw (.2, 1.4) rectangle (.8, 1.8);
				\roundNbox{fill=white}{(.5,2.3)}{.25}{0}{0}{${\eta}$};
				\roundNbox{fill=white}{(.5,.4)}{.2}{.5}{.5}{${J}$};
				\roundNbox{fill=white}{(.5, 1.5)}{.3}{.4}{.3}{${F(\ev)}$};
				\filldraw[black] (1,0) circle (1.5pt);
			} 
			~~ \underset{\eqref{1CatDualReplacement}}{=} ~~ 
			\tikzmath
			{
				\draw[thick] (0,-1) -- (0, 1.2)  arc (180:0:.5cm) -- (1,.5)  arc (-180:0:.5cm) -- (2, 2.666);
				% additional draw args, center, radius, additional left x-space, additional right x-space, contents
				% \roundNbox{fill=white}{(.5,1)}{.4}{.5}{.5}{${\Delta(\eta)}$};
				\roundNbox{fill=white}{(.5,2.3)}{.25}{0}{0}{${\eta}$};
				\filldraw[black] (1,1.2) circle (1.5pt);
				\filldraw[black] (1,.5) circle (1.5pt);
			} 
			~~ = ~~  
			\tikzmath{
				\roundNbox{fill=white}{(.5,.833)}{.25}{0}{0}{${\eta}$};
				\draw[thick] (1, -1) -- (1, 2.666);
			}
		\end{equation}
		
		The other half of the antipode axiom is proven similarly. Finally, since $C$ has a right dual, then $S$ is invertible, and thus $\End(F)$ is a Hopf algebra.
	\end{proof}
	
	% }
% \begin{rem*}
	%     One advantage of this approach is that the anti-(co)algebra homomorphism property of $S$ is manifest from the definition; this avoids an otherwise unintuitive proof. Likewise, at no point was it necessary to verify the coassociativity of the comultiplication. 
	% \end{rem*}

\section{Reconstruction for 2-categories} \label{2CategoryReconstructionSection}
We categorify the approach and results of the previous section.  We refer the reader to \cite{johnson20202dimensional} for general background on 2-categories and \cite{barrett2018gray, Jones_2022} for explanations of the graphical calculus we use for functors between monoidal 2-categories.
\subsection{Monoidal Slice 2-categories and 2-coalgebra structures} 
\begin{defn}
	We define the 3-category $\TwoSlice$ as the full subcategory of $\ThreeVec/\TwoVec$ consisting of locally faithful functors. Unpacked, the 2-category $\TwoSlice$ has the following $\Hom$-sets.  
	\begin{itemize}
		\item Objects $(\cC, \cF)$ where $\cC$ is a finite semisimple 2-category and $\cF$ is a locally faithful linear 2-functor $\cC \to \TwoVec$. 
		\item A 1-morphism from $(\cC, \cF)$ to $(\cD, \cG)$ is a pair $(\cT, \tau)$ where $\cT \colon \cC \to \cD$ is a $\cT$ linear 2-functor and $\tau$ is a natural equivalence $\cF \Rightarrow \cG\cT$. 
		\item A 2-morphism from $(\cT, \tau)$ to $(\cT', \tau')$ is a pair $(\sigma, \Sigma)$ where $\sigma$ is a natural equivalence $\cT' \Rightarrow \cT$ and $\Sigma$ is an invertible modification as below:
		% https://q.uiver.app/#q=WzAsOCxbMSwyLCJcXGNWIl0sWzAsMCwiXFxjQyJdLFsyLDAsIlxcY0QiXSxbMiwxXSxbNCwxXSxbNiwwLCJcXGNEIl0sWzUsMiwiXFxjViJdLFs0LDAsIlxcY0MiXSxbMSwyLCJcXGNUIl0sWzEsMCwiXFxjRiIsMl0sWzIsMCwiXFxjRyJdLFs1LDYsIlxcY0ciXSxbNyw1LCJcXGNUJyIsMix7ImxhYmVsX3Bvc2l0aW9uIjozMCwiY3VydmUiOjJ9XSxbNyw2LCJcXGNGIiwyXSxbNyw1LCJcXGNUIiwwLHsiY3VydmUiOi0yfV0sWzMsNCwiXFxTaWdtYSIsMCx7ImxldmVsIjozfV0sWzksMiwiXFx0YXUiLDIseyJzaG9ydGVuIjp7InNvdXJjZSI6NDAsInRhcmdldCI6NDB9fV0sWzEzLDUsIlxcdGF1JyIsMix7ImxhYmVsX3Bvc2l0aW9uIjozMCwib2Zmc2V0Ijo0LCJzaG9ydGVuIjp7InNvdXJjZSI6MjAsInRhcmdldCI6NjB9fV0sWzEyLDE0LCJcXHNpZ21hIiwyLHsic2hvcnRlbiI6eyJzb3VyY2UiOjIwLCJ0YXJnZXQiOjIwfX1dXQ==
		
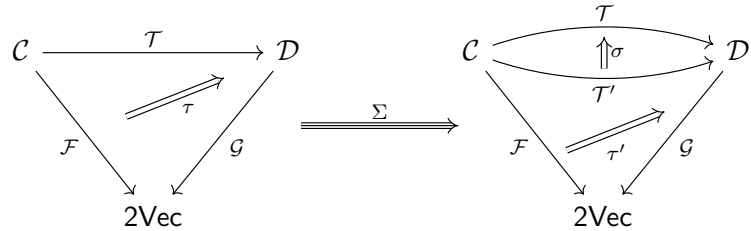
\[\begin{tikzcd}
			\cC && \cD && \cC && \cD \\
			&& {} && {} \\
			& \TwoVec &&&& \TwoVec
			\arrow["\cT", from=1-1, to=1-3]
			\arrow[""{name=0, anchor=center, inner sep=0}, "\cF"', from=1-1, to=3-2]
			\arrow["\cG", from=1-3, to=3-2]
			\arrow["\cG", from=1-7, to=3-6]
			\arrow[""{name=1, anchor=center, inner sep=0}, "{\cT'}"', curve={height=12pt}, from=1-5, to=1-7]
			\arrow[""{name=2, anchor=center, inner sep=0}, "\cF"', from=1-5, to=3-6]
			\arrow[""{name=3, anchor=center, inner sep=0}, "\cT", curve={height=-12pt}, from=1-5, to=1-7]
			\arrow["\Sigma", triple, from=2-3, to=2-5]
			\arrow["\tau"', shorten <=16pt, shorten >=16pt, Rightarrow, from=0, to=1-3]
			\arrow["{\tau'}"'{pos=0.3}, shift right=4, shorten <=8pt, shorten >=24pt, Rightarrow, from=2, to=1-7]
			\arrow["\sigma"', shorten <=3pt, shorten >=3pt, Rightarrow, from=1, to=3]
		\end{tikzcd}\]
		which is also expressible as
		% https://q.uiver.app/#q=WzAsMyxbMSwyLCJcXGNHXFxjVCJdLFswLDAsIlxcY1QiXSxbMiwwLCJcXGNHXFxjVCciXSxbMCwyLCJcXGNHXFxzaWdtYSIsMix7ImxldmVsIjoyfV0sWzEsMiwiXFxzaWdtYSIsMCx7ImxldmVsIjoyfV0sWzEsMCwiXFx0YXUiLDIseyJsZXZlbCI6Mn1dLFs1LDIsIlxcU2lnbWEiLDIseyJzaG9ydGVuIjp7InNvdXJjZSI6NDAsInRhcmdldCI6NDB9LCJsZXZlbCI6M31dXQ==
		
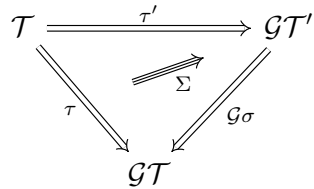
\[\begin{tikzcd}
			\cT && {\cG\cT'} \\
			\\
			& \cG\cT
			\arrow["\cG\sigma"', Rightarrow, to=3-2, swap, from=1-3]
			\arrow["\tau'", Rightarrow, from=1-1, to=1-3]
			\arrow[""{name=0, anchor=center, inner sep=0}, "\tau"', Rightarrow, from=1-1, to=3-2]
			\arrow["\Sigma"', triple, shorten <=19pt, shorten >=19pt, from=0, to=1-3]
		\end{tikzcd}\]
		\item A 3-morphism from $(\Sigma, \sigma)$ to $(\Sigma', \sigma')$ is a modification $\Theta \colon \sigma' \Rrightarrow \sigma$ such that: 
		% https://q.uiver.app/#q=WzAsNyxbMSwyLCJcXGNHXFxjVCJdLFswLDAsIlxcY1QiXSxbMiwwLCJcXGNHXFxjVCciXSxbMywxLCI9Il0sWzQsMCwiXFxjVCJdLFs1LDIsIlxcY0dcXGNUIl0sWzYsMCwiXFxjRyJdLFswLDIsIlxcY0dcXHNpZ21hIiwyLHsibGV2ZWwiOjJ9XSxbMSwyLCJcXHRhdSciLDAseyJsZXZlbCI6Mn1dLFsxLDAsIlxcdGF1IiwyLHsibGV2ZWwiOjJ9XSxbNCw1LCJcXHRhdSIsMix7ImxldmVsIjoyfV0sWzQsNiwiXFx0YXUnIiwwLHsibGV2ZWwiOjJ9XSxbNiw1LCJcXGNHXFxzaWdtYSciLDIseyJsYWJlbF9wb3NpdGlvbiI6NzAsImN1cnZlIjoyLCJsZXZlbCI6Mn1dLFs2LDUsIlxcY0dcXHNpZ21hIiwwLHsiY3VydmUiOi0yLCJsZXZlbCI6Mn1dLFs5LDIsIlxcU2lnbWEiLDIseyJzaG9ydGVuIjp7InNvdXJjZSI6NDAsInRhcmdldCI6NDB9LCJsZXZlbCI6M31dLFsxMCw2LCJcXFNpZ21hJyIsMCx7ImxhYmVsX3Bvc2l0aW9uIjozMCwib2Zmc2V0IjotMywic2hvcnRlbiI6eyJzb3VyY2UiOjMwLCJ0YXJnZXQiOjUwfSwibGV2ZWwiOjN9XSxbMTIsMTMsIiIsMix7InNob3J0ZW4iOnsic291cmNlIjoyMCwidGFyZ2V0IjoyMH0sImxldmVsIjozfV1d
		
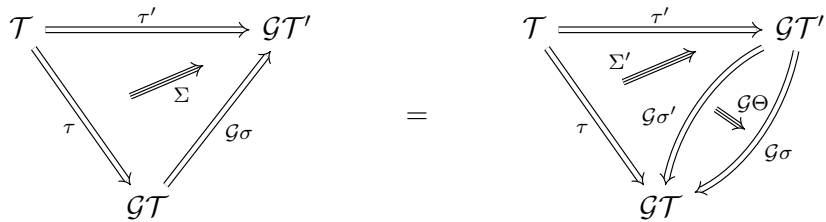
\[\begin{tikzcd}
			\cT && {\cG\cT'} && \cT && \cG\cT' \\
			&&& {=} \\
			& \cG\cT &&&& \cG\cT
			\arrow["\cG\sigma"', Rightarrow, from=3-2, to=1-3]
			\arrow["{\tau'}", Rightarrow, from=1-1, to=1-3]
			\arrow[""{name=0, anchor=center, inner sep=0}, "\tau"', Rightarrow, from=1-1, to=3-2]
			\arrow[""{name=1, anchor=center, inner sep=0}, "\tau"', Rightarrow, from=1-5, to=3-6]
			\arrow["{\tau'}", Rightarrow, from=1-5, to=1-7]
			\arrow[""{name=2, anchor=center, inner sep=0}, "{\cG\sigma'}"'{pos=0.7}, curve={height=15pt}, Rightarrow, from=1-7, to=3-6]
			\arrow[""{name=3, anchor=center, inner sep=0}, "\cG\sigma", curve={height=-15pt}, Rightarrow, from=1-7, to=3-6]
			\arrow["\Sigma"', shorten <=19pt, triple, shorten >=19pt, from=0, to=1-3]
			\arrow["{\Sigma'}"{pos=0.3}, shift left=3, triple, shorten <=14pt, shorten >=24pt, from=1, to=1-7]
			\arrow["{\cG\Theta}", triple, shorten <=5pt, shorten >=5pt, from=2, to=3]
		\end{tikzcd}\]
	\end{itemize}
	The directions of composition may at first seem strange; compare \cite{Jones_2022}, Definition 3.2. 
\end{defn}The following lemma is proven identically to Lemma \ref{1Slice1Truncated}, using local faithfulness of $\cG$ and local semisimplicity of $\cD$.
\begin{lem}
	$\TwoSlice$ is 2-truncated.
\end{lem}
\begin{lem} \label{2FunctorAxiomsForQ}With notation as above, the assignments
	\begin{itemize}
		\item $\cQ(\cC, \cF) \coloneqq \End(\cF) $
		\item $\cQ(\cT, \tau)(\eta) \coloneqq \tau^{-1}\eta\tau$
		\item $\cQ(\sigma, \Sigma) \coloneqq$  $$ \tikzmath{
			\draw[thick] (0,0) -- (0,4);
			\roundNbox{fill=white}{(0,1)}{.3}{.05}{.05}{${\tau}$};
			\roundNbox{fill=white}{(0,2)}{.3}{.05}{.05}{${\eta}$};
			\roundNbox{fill=white}{(0,3)}{.3}{.05}{.05}{${\tau^{-1}}$};
		} \quad \underset{\Sigma}{\Rightarrow} \quad 
		\tikzmath{
			\draw[thick] (0,0) -- (0,4);
			\roundNbox{fill=white}{(0,0.5)}{.3}{.05}{.05}{${\tau'}$};
			\roundNbox{fill=white}{(0,1.5)}{.3}{.05}{.05}{${\cF\sigma}$};
			\roundNbox{fill=white}{(0,2.5)}{.3}{.05}{.05}{${\eta}$};
			\roundNbox{fill=white}{(0,3.5)}{.3}{.05}{.05}{${\tau^{-1}}$};
		}
		\quad \underset{\eta_{F\sigma}}{\Rightarrow} \quad 
		\tikzmath{
			\draw[thick] (0,0) -- (0,4);
			\roundNbox{fill=white}{(0,0.5)}{.3}{.05}{.05}{${\tau'}$};
			\roundNbox{fill=white}{(0,2.5)}{.3}{.05}{.05}{${\cF\sigma}$};
			\roundNbox{fill=white}{(0,1.5)}{.3}{.05}{.05}{${\eta}$};
			\roundNbox{fill=white}{(0,3.5)}{.3}{.05}{.05}{${\tau^{-1}}$};
		}
		\quad \underset{\Sigma^{-1}}{\Rightarrow} \quad 
		\tikzmath{
			\draw[thick] (0,0) -- (0,4);
			\roundNbox{fill=white}{(0,1)}{.3}{.05}{.05}{${\tau'}$};
			\roundNbox{fill=white}{(0,2)}{.3}{.05}{.05}{${\eta}$};
			\roundNbox{fill=white}{(0,3)}{.3}{.25}{.25}{${(\tau')^{-1}}$};
		}
		$$
		\item $\cQ(\Theta) = \id_{\cQ(\sigma, \Sigma)}$ 
	\end{itemize}
	form a 2-functor $\TwoSlice^{1op} \to \TwoAlg$. 
\end{lem} 
\begin{proof}
	We show that $\cQ(\Theta)$ is well defined; the remaining verifications are routine. The commutative diagram is 
	\renewcommand{\roundNbox}[6]{
		\draw[rounded corners=5pt, very thick, #1] ($#2+(-#3,-#3)+(-#4,0)$) rectangle ($#2+(#3,#3)+(#5,0)$);
		\coordinate (ZZa) at ($#2+(-#4,0)$);
		\coordinate (ZZb) at ($#2+(#5,0)$);
		\node[label={[yshift=-0.4cm]#6}] at ($1/2*(ZZa)+1/2*(ZZb)$) {};
	}
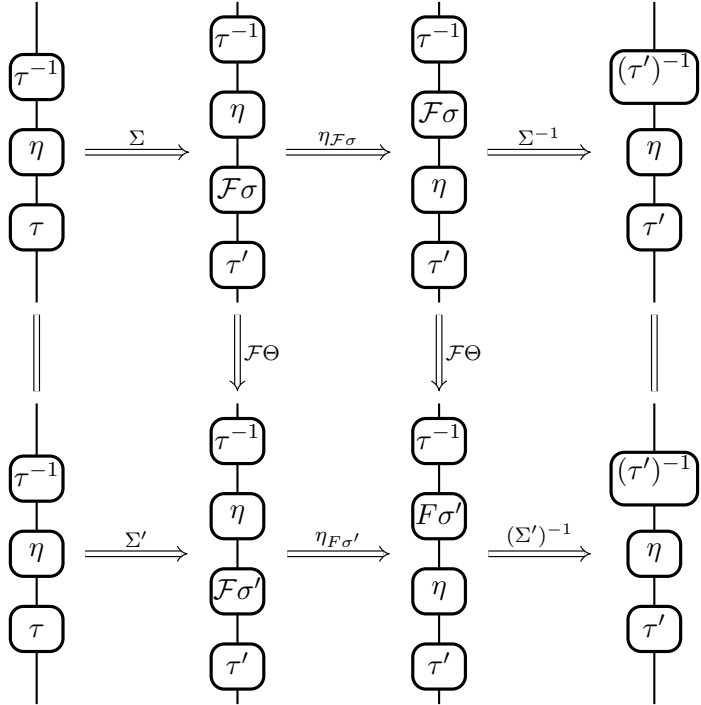
	\[\hspace{-1.25cm}\begin{tikzpicture}[baseline= (a).base] \node (a) at (0,0){
			\begin{tikzcd}[sep = large]
				\tikzmath{
					\draw[thick] (0,0) -- (0,4);
					\roundNbox{fill=white}{(0,1)}{.3}{.05}{.05}{${\tau}$};
					\roundNbox{fill=white}{(0,2)}{.3}{.05}{.05}{${\eta}$};
					\roundNbox{fill=white}{(0,3)}{.3}{.05}{.05}{${\tau^{-1}}$};
				} \ar[r, Rightarrow, "\Sigma"] \ar[d,Rightarrow, no head]
				&	
				\tikzmath{
					\draw[thick] (0,0) -- (0,4);
					\roundNbox{fill=white}{(0,0.5)}{.3}{.05}{.05}{${\tau'}$};
					\roundNbox{fill=white}{(0,1.5)}{.3}{.05}{.05}{${\cF\sigma}$};
					\roundNbox{fill=white}{(0,2.5)}{.3}{.05}{.05}{${\eta}$};
					\roundNbox{fill=white}{(0,3.5)}{.3}{.05}{.05}{${\tau^{-1}}$};
				} \ar[r, Rightarrow, "\eta_{\cF\sigma}"] \ar[d,Rightarrow, "\cF\Theta"]
				&
				\tikzmath{
					\draw[thick] (0,0) -- (0,4);
					\roundNbox{fill=white}{(0,0.5)}{.3}{.05}{.05}{${\tau'}$};
					\roundNbox{fill=white}{(0,2.5)}{.3}{.05}{.05}{${\cF\sigma}$};
					\roundNbox{fill=white}{(0,1.5)}{.3}{.05}{.05}{${\eta}$};
					\roundNbox{fill=white}{(0,3.5)}{.3}{.05}{.05}{${\tau^{-1}}$};
				}  \ar[r, Rightarrow, "\Sigma^{-1}"] \ar[d, Rightarrow, "\cF\Theta"]
				&
				\tikzmath{
					\draw[thick] (0,0) -- (0,4);
					\roundNbox{fill=white}{(0,1)}{.3}{.05}{.05}{${\tau'}$};
					\roundNbox{fill=white}{(0,2)}{.3}{.05}{.05}{${\eta}$};
					\roundNbox{fill=white}{(0,3)}{.35}{.225}{.225}{${(\tau')^{-1}}$};
				}  \ar[d, Rightarrow, no head]\\
				\tikzmath{
					\draw[thick] (0,0) -- (0,4);
					\roundNbox{fill=white}{(0,1)}{.3}{.05}{.05}{${\tau}$};
					\roundNbox{fill=white}{(0,2)}{.3}{.05}{.05}{${\eta}$};
					\roundNbox{fill=white}{(0,3)}{.3}{.05}{.05}{${\tau^{-1}}$};
				}  \ar[r, Rightarrow, "\Sigma'"] 
				&
				\tikzmath{
					\draw[thick] (0,0) -- (0,4);
					\roundNbox{fill=white}{(0,0.5)}{.3}{.05}{.05}{${\tau'}$};
					\roundNbox{fill=white}{(0,1.5)}{.3}{.05}{.05}{${\cF\sigma'}$};
					\roundNbox{fill=white}{(0,2.5)}{.3}{.05}{.05}{${\eta}$};
					\roundNbox{fill=white}{(0,3.5)}{.3}{.05}{.05}{${\tau^{-1}}$};
				}  \ar[r, Rightarrow, "\eta_{F\sigma'}"]
				&
				\tikzmath{
					\draw[thick] (0,0) -- (0,4);
					\roundNbox{fill=white}{(0,0.5)}{.3}{.05}{.05}{${\tau'}$};
					\roundNbox{fill=white}{(0,2.5)}{.3}{.05}{.05}{${F\sigma'}$};
					\roundNbox{fill=white}{(0,1.5)}{.3}{.05}{.05}{${\eta}$};
					\roundNbox{fill=white}{(0,3.5)}{.3}{.05}{.05}{${\tau^{-1}}$};
				} \ar[r, Rightarrow, "(\Sigma')^{-1}"]
				&
				\tikzmath{
					\draw[thick] (0,0) -- (0,4);
					\roundNbox{fill=white}{(0,1)}{.3}{.05}{.05}{${\tau'}$};
					\roundNbox{fill=white}{(0,2)}{.3}{.05}{.05}{${\eta}$};
					\roundNbox{fill=white}{(0,3)}{.35}{.225}{.225}{${(\tau')^{-1}}$};
				} 
		\end{tikzcd}};
	\end{tikzpicture}\]
	The outer squares commute by definition and the inner square is the modification axiom for $\cF\Theta$. We see that if there exists $\Theta \colon (\Sigma, \sigma) \Rrightarrow (\Sigma', \sigma')$ then $\cQ (\Sigma, \sigma) = \cQ  (\Sigma', \sigma')$.
\end{proof}
\begin{lem}\label{2EquivLegs}
	The functors $\cQ$ and $\Mod(-)$ are inverse equivalences between $\TwoSlice$ and $\TwoAlg$.
\end{lem}
\begin{proof}
	For every monoidal category $C$, we require a 2-natural equivalence from $C \to \End(\textbf{Forget}_C) = \End(\cF_C)$. Let $(M, \rho, m)$ be a $C$-module and $(F, s) $ a $C-$module functor.
	We define the natural map $\cY \colon C \to \End(\cF_C)$ on objects $c, c'$ and a morphism $ f\colon c \to c'$ as   
	\begin{itemize}
		\item $\cY(c)_M \coloneqq \cF(\rho(c)) \colon \cF(M) \to \cF(M)$
		\item $\cY(c)_F \coloneqq \cF(s)$
		\item $\cY(f) \coloneqq \cF(\rho(f)) \colon \cF f(c) \Rightarrow \cF f(c')$
	\end{itemize}
	This map is a 2-natural isomorphism by the linear 2-Yoneda lemma, and the fact that $\Mod(C)$ is the Cauchy completion of $\textbf{B}C$ \cite[Example 2.1.10]{MR4372801}. That is, restriction to $\mathbf{BC}$ provides an equivalence $\End(F)$ to $\End(\Hom(*, -))$. The 2-Yoneda lemma then provides $\End(\Hom(*, -)) \cong C$.  The module associativity constraint  $m$ makes $\cY$ a monoidal functor. 
	Next, for every monoidal 2-category $\cC$ we require a natural 2-equivalence in $\TwoSlice$:
	\[
	(\cC, \cF) \simeq (\Mod(\End(\cF)), \textbf{Forget}_{\End{\cF}})
	\]
	We choose the manifestly natural maps 
	\begin{itemize}
		\item $c \mapsto (\cF(c), \End(\cF)|_c)$
		\item $f \colon c' \to c' \mapsto \cF(f)$
		\item $\sigma \colon f \Rightarrow f' \mapsto \cF(\sigma)$
	\end{itemize}
	together with the equality 2-morphism. That restriction to $c$ is a monoidal functor $\End(\cF) \to \End(\cF(c))$ is clear. The assigments on 1- and 2-morphisms commute appropriately with the $\End(\cF)$ action by definition of $\End(\cF)$. This morphism is invertible in $\TwoSlice$ if the underlying functor is, by lemma \ref{EquivalenceInTwoSlice}. 
	The key fact is that from every bimodule category $_\cC \cM_\TwoVec$ we may obtain another bimodule $_\cC \cM_{\End_{\Mod_\cC}(_\cC\cM)}$; this procedure corresponds to the map above. It is enough to verify that this bimodule induces a Morita equivalence, since we have by \cite[Lemma 2.2.2]{DECOPPET2022107029} that $\Mod$ is a triequivalence and therefore
	
	\[\Mod(\End(\cF)) \simeq \Mod(\End_{\Mod(\cC)}(_\cC \cM_\TwoVec)) = \Mod(\End_{\Mod(\cC)}(_\cC \cM))\]
	
	Since $\cF$ is locally faithful, the module $\cM$ is faithful. The bimodule $\cM$ induces a Morita equivalence if and only if the bicommutant of the image of $\cC$ is $\cC$ \cite[Prop. 4.2]{MR2677836}. This is precisely \cite[Thm. 7.12.11]{MR3242743}.	
\end{proof}
\begin{rem*}
	Proposition 2.3.1 of \cite{decoppet2023drinfeld} is a further categorification of a result used here to characterize Morita equivalence, which is in turn a categorification of the double centralizer theorem from classical algebra.  
\end{rem*}
Since $\Mod(-)$ is a symmetric monoidal (\ref{ModSymmetric2Monoidal}) equivalence, its pseudoinverse is as well. 
\begin{cor} $\cQ$ has symmetric monoidal structure. 
\end{cor}
We recall from Definition 3.1 of \cite{décoppet20212deligne} the 3-universal property of the Deligne 2-tensor product once again specialized to the semisimple case.  
\begin{thm}
	Given $\cC$ and $\cD$ two finite semisimple linear 2-categories, there exists a finite semisimple linear 2-category $\cC \boxdot \cD$ and linear 2-functor $\boxdot \colon \cC \times \cD \to \cC \boxdot \cD$ such that precomposition with $\boxdot$ induces an equivalence 
	\[
	\Hom(\cC \boxdot \cD, \cE) \simeq \Hom_\textit{bilin}(\cC \times \cD, \cE)
	\]
	for all finite $\cE$. This equivalence is natural in all three variables. Unpacked, this means:
	\begin{itemize} 
		\item For every finite 2-category $\cE$ and bilinear bi-2-functor $\cF \colon \cC \times \cD $ there exists a 2-functor $\bar \cF \colon \cC \boxdot \cD \to \cE$ and 2-natural equivalence $u \colon \bar F \circ \boxdot \Rightarrow F$.
		\item For every two functors $\cG, \cH \colon \cC \boxdot \cD \to \cE$ and 2-natural transformation $\tau \colon \cG \circ \boxdot \Rightarrow \cH \circ \boxdot$, there exists a 2-natural equivalence $\tau' \colon \cG \to \cH$ and invertible modification $\Sigma \colon \tau' \circ \boxdot \Rrightarrow \tau$.
		\item Finally, for every two 2-natural transformations $\tau, \tau' \colon \cG \to \cH$ and modification $\Pi \colon \tau \circ \boxdot \rightarrow \tau' \circ \boxdot$, there exists a \textit{unique} invertible modification $\Pi' \colon \tau \to \tau$ such that $\Pi' \circ \boxdot = \Pi$.
	\end{itemize}
	Furthermore: 
	\begin{itemize}
		\item If $\cC$ and $\cD$ are monoidal, then so is $\cC \boxdot \cD$. With this monoidal structure, the 2-functor $\boxdot$ is monoidal. 
		\item If $\cF \colon \cC \times \cD \to \cE$ is monoidal, then so is the induced 2-functor $\bar F \colon \cC \boxdot \cD \to \cE$
	\end{itemize}
	Finally, for finite semisimple linear monoidal 1-categories $A, B$, then 
	\[
	\Mod(A) \boxdot \Mod(B) \simeq \Mod(A \boxtimes B)
	\]
	and this equivalence is natural in $A$ in $B$.
\end{thm}
We provide details of these constructions in Appendix \ref{2CatAppendix}.
\begin{construction}
	The Deligne 2-tensor product, together with the symmetric monoidal structure on $\TwoVec$ give $\TwoSlice$ the structure of a symmetric monoidal 2-category. 
\end{construction}
\begin{lem}
With the above structure (and the standard one on $\TwoAlg$), the functor $\cQ$ is a symmetric monoidal equivalence. 
\end{lem}
From this we have the following:
\begin{cor} \label{2CategoriesOfMonoidsEquivalent}
	 With the standard monoidal structure on modules for a coalgebra object, the functors $\cQ$ and $\Mod$ form a symmetric monoidal equivalence between the $2$-categories of algebras in $\TwoSlice$ and coalgebras in $\TwoAlg$. By Theorem 3.9 of \cite{neuchl}, the 2-category of coalgebra objects in $\TwoAlg$ is the 2-category of bialgebra objects in $\TwoVec$; i.e  \cite{1509.06811} finitely semisimple categories which are compatibly both monoidal and comonoidal. 
\end{cor}
% \begin{prop}
	% A bialgebra object in $\TwoVec$ is equivalent to a monoidal category $C$, with monoidal functors $\Delta$ and $\epsilon$ \nn{blah, here is where we give the unpacked definition of a Hopf 2-algebra in terms of a category with structure}. 
	% \end{prop}
We have the following description of the $2$-category of algebras in $\TwoSlice$:
\begin{prop} \label{2MonoidIn2Slice}
	Every monoidal 2-functor $(\cC, \cF)$ in $\TwoSlice$ is canonically an algebra. With respect to this structure every 1-morphism $(\cT, \tau)$ in $\TwoSlice$ is an algebra homomorphism if $T$ and $\tau$ are monoidal, and likewise for $2$-morphisms. Moreover, every algebra is equivalent to one of this form.
\end{prop}
\begin{proof}
	The second part of the lemma follows from Corollary \ref{2CategoriesOfMonoidsEquivalent}. We will compare the definition of algebra, algebra 1-morphism and algebra 2-morphism in a weak-2 category given in section 3.1 of \cite{decoppet2023finite} with the definitions of monoidal 2-functors, 2-transformations, and modifications from pages 90-98 of \cite{schommer-pries-thesis}. We present this comparison as a series of tables, two each for 0-,1-, and 2-morphisms. One table provides the correspondence of data, and the other the correspondence of axioms. The far right column of this latter table provides the nontrivial data in $\cC$ that allows a 3-morphism between the two modifications making up a given axiom. The most complex axioms are the first axiom for a functor/algebra object and the first axiom for a natural transformation/algebra 1-morphism. We provide explicit pasting diagram verifications of these axioms in Appendix \ref{2CatAppendix}; the others are simpler and left to the reader. 
	
 In many cases, the universal property of the $2$-Deligne tensor product is implicitly used to replace instances of $\times$ with $\boxdot$. The notation in this section is locally inherited from the two authors whose works we are comparing; in particular there may be notation conflicts between the second column and the first.
	
	We mention that the cells which are required to be invertible are as a consequence of Lemma \ref{EquivalenceInTwoSlice}.
	\begin{center}
\textbf{0-cells:}
	\end{center}
\textit{Data:}
\begin{center}
\begin{tabular}{|c|c|}
	\hline 
	Monoidal 2-functor & Corresponding algebra in $\TwoSlice$ \\
	\hline
	Underlying functor $\cF \colon \cC \to \TwoVec$ & Object $(\cC, \cF)$ \\
	\hline 
	Tensorator $\chi \colon \boxtimes \circ (\cF \times \cF) \Rightarrow \cF \circ \otimes_\cC$ & Multiplication morphism $m \coloneqq (\otimes_\cC, \chi)$ \\
	\hline
	Unitor $\iota \colon I_{\TwoVec} \Rightarrow \cF \circ I_\cC$ & Unit morphism $i \coloneqq (I_\cC, \iota)$ \\ 
	\hline
	 Hexagonator modification $\omega$ & Algebra pentagonator  $\mu^{(\cC, \cF)} \coloneqq (\alpha_\cC, \omega)$ \\
	 \hline
	 Left 2-unitor modification $\gamma$ & Left algebra 2-unitor  $\lambda^{(\cC, \cF)} \coloneqq (\ell_\cC, \gamma)$\\
	 \hline
	 Right 2-unitor modification $\delta$ & Right algebra 2-unitor $\rho^{(\cC, \cF)} \coloneqq (r_\cC, \delta)$ \\
	 \hline
\end{tabular}
\end{center}
That the 2-cells $\mu^{(\cC, \cF)}, \rho^{(\cC, \cF)}, \lambda^{(\cC, \cF)}$ are well typed is a consequence of the fact that the constraint cells for $\TwoSlice$ are induced from the monoidal structure on $\TwoVec$ and the Cartesian product on $\sf{2Cat}$, and was the motivation for the definition of the direction of the morphisms in $\TwoSlice$. \\
\textit{Axioms:}
\begin{center}
	\begin{tabular}{|c|c|c|}
		\hline
	Monoidal 2-functor axiom & Corresponding algebra axiom & Witnessing 3-morphism \\
	\hline	
	1\textsuperscript{st} &(a)& $\Pi_\cC$, the pentagonator of $\cC$. \\
	\hline
	2\textsuperscript{nd} & (b)& $\mu_\cC$, the middle 2-unitor of $\cC$.\\
	\hline
	\end{tabular}
\end{center}
 In the 3-functor axioms, the region corresponding to the 3-morphism in the first row is the region marked $H\pi$ on the bottom diagram of \cite[17]{MR1261589} (alternatively \cite[68]{MR3076451}). See Lemma \ref{2AlgebraAxiom1Verification} for a more thorough verification. The role of the pentagonator here is that of the associator in Proposition \ref{1MonoidIn1Slice}. To see the second row, invert all but the last 2-morphism in the left expression of D\'ecoppet's axiom (b). The claim then collapses to precisely the statement that the constraint cell $\mu_\TwoSlice$ is induced from $\TwoVec$. 
\begin{center}
\textbf{1-cells}
\end{center}
\textit{Data:}
\begin{center}
\begin{tabular}{|c|c|}
\hline 
Monoidal 2-functor $\cT$ and monoidal natural transformation $\tau$ & Corresponding algebra 1-morphism \\ 
\hline
Underlying functor $\cT$ and natural transformation $\tau$  &   1-morphism $(\cT, \tau)$\\
\hline
Tensorator $\chi_\cT$ of $\cT$ and pentagonator $\Pi$ of $\tau$ & 2-cell $\kappa_{(\cT, \tau)} = (\chi_T, \Pi^{-1}) $\\ 
\hline
Unitor $\iota_T$ of $\cT$ and unitor $M$ of $\tau$ & 2-morphism $(\iota_\cT, M)$\\
\hline
\end{tabular}
\end{center}
The diagram corresponding to $\kappa_{(\cT, \tau)}$ has six natural transformations; the composite of $\chi_{\cF'}$ and $\chi_\cT$ is $\chi_{\cF'\cT}$, so that $\Pi$ is well typed. A similar statement holds for the unit. \\
\textit{Axioms: }
\begin{center}
\begin{tabular}{|c|c|} \hline
	Monoidal transformation axiom & Corresponding algebra 1-morphism axiom \\
	\hline
	$MBTA1$ &  (a) \\
	\hline
	$MBTA2$ &  (b) \\
	\hline
	$MBTA3$ &  (c) 	\\
	\hline 
\end{tabular} 
\end{center}
See Lemma \ref{2Algebra1MorphismAxiom1Verification} for a verification of the first row in the above table; the others are similar.
\begin{center}
\textbf{2-cells}
\end{center}
\textit{Data:}
\begin{center}
\begin{tabular}{|c|c|}
	\hline
Monoidal transformation $\sigma$ and modification $\Sigma$ & Algebra 2-morphism \\
\hline
Underlying morphisms $\sigma$, $\Sigma$ &  2-cell $(\sigma, \Sigma)$\\
\hline
\end{tabular}
\end{center}
\textit{Axioms:}
\begin{center}
\begin{tabular}{|c|c|}
	\hline
	Monoidal modfication axiom & Corresponding Algebra 2-morphism axiom \\
	\hline
	BMBM1 & (a) \\
	\hline
	BMBM2 & (b) \\
	\hline
\end{tabular}
\end{center}
Both axioms involve only two or three 2-cells (after a nudging convention is applied to BMBM1) and are straightforward to check.
\end{proof}
We have just proved:
\begin{cor}[2-Bialgebra Reconstruction] \label{2BialgebraReconstruction}
The functors $\cQ$ and $\Mod(-)$ induce a symmetric monoidal equivalence between $\TwoSlice^{1\op}$ and the 2-category of bialgebra objects in $\TwoVec$. 
\end{cor}

\begin{remark}
Given suitable ``external'' definitions of braided/symmetric/sylleptic monoidal 2-categories, it is straightforward to add structure to both sides of this equivalence. We will revisit this topic in future work.
\end{remark}

\subsection{The Sweedler isomorphism}
 In order to make computations, we give an analogue of Sweedler notation which is appropriately natural. As expected, there is a contractible space of choices; this categorifies the many equal ways to decompose a tensor as a sum of simple tensors. Fortunately, it is sufficient for our purposes to construct only one. We have the commutative diagram:
% https://q.uiver.app/#q=WzAsNCxbMCwwLCJcXEVuZChcXGNGKSJdLFsyLDAsIlxcRW5kKFxcYm94dGltZXMgXFxjaXJjIFxcY0YgXFxib3hkb3QgXFxjRikiXSxbMCwxLCJcXEVuZChcXGNGKSBcXGJveHRpbWVzIFxcRW5kKFxcY0YpIl0sWzIsMSwiXFxFbmQoXFxjRiBcXGJveGRvdCBcXGNGKSJdLFswLDIsIlxcRGVsdGEiLDJdLFswLDEsIlxcY1EoXFxvdGltZXMsIEopIl0sWzIsMywiXFx0YXUgXFxib3h0aW1lcyBcXHRhdScgXFxtYXBzdG8gXFx0YXUgXFxib3hkb3RcXHRhdSciLDJdLFszLDEsIlxcYm94dGltZXMgXFxjaXJjIC0iLDJdLFsxLDIsIiIsMSx7InN0eWxlIjp7ImJvZHkiOnsibmFtZSI6ImRhc2hlZCJ9fX1dXQ==
\[\begin{tikzcd}
	{\End(\cF)} && {\End(\boxtimes \circ \cF \boxdot \cF)} \\
	{\End(\cF) \boxtimes \End(\cF)} && {\End(\cF \boxdot \cF)}
	\arrow["\Delta"', from=1-1, to=2-1]
	\arrow["{\cQ(\otimes, J)}", from=1-1, to=1-3]
	\arrow["{\tau \boxtimes \tau' \mapsto \tau \boxdot\tau'}"', from=2-1, to=2-3]
	\arrow["{\boxtimes \circ -}"', from=2-3, to=1-3]
	%\arrow[dashed, from=1-3, to=2-1]
\end{tikzcd}.\]
The composite of the bottom and right arrows is the tensorator of $\cQ$. Let $\set{c_i}{i \in \cI}$, be a set of representatives of the isomorphism classes of simple objects of $\End(\cF \boxtimes \cF)$. Choose any inverse $\tilde K$ of the bottom arrow and extend the object function $K(c_i)$ to a functor, isomorphic to but different from $K$, by direct sum. This provides a natural (in $\eta$) decomposition 
\[
\Delta(\eta) \simeq \bigoplus \eta_{(1)} \boxtimes \eta_{(2)}
\] 
Transporting the right object across the tensorator, we obtain a natural isomorphism which has components:
\begin{equation}\label{2SweedlerNotation}
\begin{tikzpicture}
\draw[thick] (0,-1) -- (0,1);
\draw[thick] (.75,-1) -- (.75,1);
\roundNbox{fill=white}{(0,-.333)}{.3}{.05}{.05}{${\eta_{(1)}}$};
\roundNbox{fill=white}{(.75,.333)}{.3}{.05}{.05}{${\eta_{(2)}}$};
	\draw[thick] (4,-1) -- (4,1);
	\draw[thick] (4.75,-1) -- (4.75,1);
	\roundNbox{fill=white}{(4.375,0)}{.3}{.4}{.4}{${\cQ(\otimes, J)}$};
	\node at (2.2, 0) {$\Longrightarrow$};
\end{tikzpicture}
\end{equation}
\subsection{Duals and Antipodes}
In this section we extend the results to include duality. We first recall an abbreviated definition of a 2-\textit{Hopf} algebra from \cite{neuchl}.
\begin{defn} \label{Hopf2Axiom}
	A \textit{2-bialgebra} in $\TwoVec$ is a compatibly monoidal and comonoidal finite semisimple linear category. A \textit{2-Hopf algebra} is a 2-bialgebra $C$ together with a functor $S \colon C \to C$ and two natural isomorphisms $\sigma_1 \colon \otimes \circ (S \boxtimes 1) \circ \Delta \Rightarrow \iota \circ \epsilon$ and $\sigma_2 \colon \iota \circ \epsilon \Rightarrow \otimes \circ (1 \boxtimes S) \circ \Delta$. This can be expressed in the following diagram with familiar outer shape: 
	% https://q.uiver.app/#q=WzAsNixbMCwxLCJDIl0sWzEsMCwiQyBcXGJveHRpbWVzIEMiXSxbMywxLCJDIl0sWzIsMCwiQyBcXGJveHRpbWVzIEMiXSxbMSwyLCJDIFxcYm94dGltZXMgQyJdLFsyLDIsIkMgXFxib3h0aW1lcyBDIl0sWzAsMSwiXFxEZWx0YSJdLFsxLDMsIlMgXFxib3h0aW1lcyAxIl0sWzMsMiwiXFxvdGltZXMiXSxbMCw0LCJcXERlbHRhIiwyXSxbMCwyLCJcXGlvdGEgXFxlcHNpbG9uIiwyXSxbNCw1LCIxIFxcYm94dGltZXMgUyIsMl0sWzUsMiwiXFxvdGltZXMiLDJdLFs3LDEwLCJcXHNpZ21hXzEiLDAseyJzaG9ydGVuIjp7InNvdXJjZSI6NDAsInRhcmdldCI6NDB9fV0sWzEwLDExLCJcXHNpZ21hXzIiLDAseyJzaG9ydGVuIjp7InNvdXJjZSI6NDAsInRhcmdldCI6NDB9fV1d
	\[\begin{tikzcd}
		& {C \boxtimes C} & {C \boxtimes C} \\
		C &&& C \\
		& {C \boxtimes C} & {C \boxtimes C}
		\arrow["\Delta", from=2-1, to=1-2]
		\arrow[""{name=0, anchor=center, inner sep=0}, "{S \boxtimes 1}", from=1-2, to=1-3]
		\arrow["\otimes", from=1-3, to=2-4]
		\arrow["\Delta"', from=2-1, to=3-2]
		\arrow[""{name=1, anchor=center, inner sep=0}, "{\iota \epsilon}"', from=2-1, to=2-4]
		\arrow[""{name=2, anchor=center, inner sep=0}, "{1 \boxtimes S}"', from=3-2, to=3-3]
		\arrow["\otimes"', from=3-3, to=2-4]
		\arrow["{\sigma_1}", shorten <=9pt, shorten >=9pt, Rightarrow, from=0, to=1]
		\arrow["{\sigma_2}", shorten <=9pt, shorten >=9pt, Rightarrow, from=1, to=2]
	\end{tikzcd}\]
\end{defn}
Neuchl includes another axiom stating that $\sigma_1$ and $\sigma_2$ satsify a version of the triangle identities for an adjunction, and then observes that it is not really a restriction, i.e, if there exist isomorphisms $\sigma_1$ and $\sigma_2$ as defined, then we can change at most one of them to obtain a pair satisfying the axiom. These axioms, rewritten for the not-necessarily-strict case, state that the following two pasting diagrams have identity components. 

% https://q.uiver.app/#q=WzAsMTAsWzAsMywiQyJdLFsyLDAsIkMgXFxib3h0aW1lcyBDIl0sWzQsMCwiQyBcXGJveHRpbWVzIEMiXSxbNiwzLCJDIl0sWzIsMiwiKEMgXFxib3h0aW1lcyBDKVxcYm94dGltZXMgQyJdLFs0LDIsIihDIFxcYm94dGltZXMgQykgXFxib3h0aW1lcyBDIl0sWzIsNCwiQyBcXGJveHRpbWVzIChDXFxib3h0aW1lcyBDKSJdLFs0LDQsIkMgXFxib3h0aW1lcyAoQ1xcYm94dGltZXMgQykiXSxbMiw2LCJDIFxcYm94dGltZXMgQyJdLFs0LDYsIkNcXGJveHRpbWVzIEMiXSxbMCwxLCJcXERlbHRhIl0sWzEsMiwiXFxpb3RhXFxlcHNpbG9uIFxcYm94dGltZXMgMSJdLFsyLDMsIlxcb3RpbWVzIl0sWzEsNCwiXFxEZWx0YVxcYm94dGltZXMgMSJdLFs0LDUsIigxIFxcYm94dGltZXMgUykgXFxib3h0aW1lczEpIl0sWzUsMiwiXFxvdGltZXNcXGJveHRpbWVzMSJdLFs0LDYsIlxcYWxwaGEiXSxbNSw3LCJcXGFscGhhIl0sWzYsNywiMSBcXGJveHRpbWVzIChTIFxcYm94dGltZXMgMSkiLDJdLFs4LDYsIjFcXGJveHRpbWVzXFxEZWx0YSIsMl0sWzAsOF0sWzgsOSwiMSBcXGJveHRpbWVzIFxcaW90YVxcZXBzaWxvbiIsMl0sWzcsOSwiMSBcXGJveHRpbWVzIFxcb3RpbWVzIl0sWzksMywiXFxvdGltZXMiLDJdLFsxMSwxNCwiXFxzaWdtYV8yIiwwLHsic2hvcnRlbiI6eyJzb3VyY2UiOjQwLCJ0YXJnZXQiOjQwfX1dLFsxOCwyMSwiXFxzaWdtYV8xIiwyLHsic2hvcnRlbiI6eyJzb3VyY2UiOjQwLCJ0YXJnZXQiOjQwfX1dLFsxNCwxOCwiXFxhbHBoYSIsMCx7InNob3J0ZW4iOnsic291cmNlIjo0MCwidGFyZ2V0Ijo0MH19XSxbMCwxNiwiXFxvbWVnYV9DIiwwLHsic2hvcnRlbiI6eyJzb3VyY2UiOjQwLCJ0YXJnZXQiOjQwfX1dLFsxNywzLCJcXG11X0MiLDAseyJzaG9ydGVuIjp7InNvdXJjZSI6NDAsInRhcmdldCI6NDB9fV1d
\[\begin{tikzcd}[row sep = small]
	&& {C \boxtimes C} && {C \boxtimes C} \\
	\\
	&& {(C \boxtimes C)\boxtimes C} && {(C \boxtimes C) \boxtimes C} \\
	C &&&&&& C \\
	&& {C \boxtimes (C\boxtimes C)} && {C \boxtimes (C\boxtimes C)} \\
	\\
	&& {C \boxtimes C} && {C\boxtimes C}
	\arrow["\Delta", from=4-1, to=1-3]
	\arrow[""{name=0, anchor=center, inner sep=0}, "{\iota\epsilon \boxtimes 1}", from=1-3, to=1-5]
	\arrow["\otimes", from=1-5, to=4-7]
	\arrow["{\Delta\boxtimes 1}", from=1-3, to=3-3]
	\arrow[""{name=1, anchor=center, inner sep=0}, "{(1 \boxtimes S) \boxtimes1}", swap, from=3-3, to=3-5]
	\arrow["\otimes\boxtimes1", from=3-5, to=1-5]
	\arrow[""{name=2, anchor=center, inner sep=0}, "\alpha", from=3-3, to=5-3]
	\arrow[""{name=3, anchor=center, inner sep=0}, "\alpha", swap, from=3-5, to=5-5]
	\arrow[""{name=4, anchor=center, inner sep=0}, "{1 \boxtimes (S \boxtimes 1)}"', swap, from=5-3, to=5-5]
	\arrow["1\boxtimes\Delta"', from=7-3, to=5-3]
	\arrow["\Delta", swap, from=4-1, to=7-3]
	\arrow[""{name=5, anchor=center, inner sep=0}, "{1 \boxtimes \iota\epsilon}"', from=7-3, to=7-5]
	\arrow["{1 \boxtimes \otimes}", swap, from=5-5, to=7-5]
	\arrow["\otimes"', from=7-5, to=4-7]
	\arrow["{\sigma_2 \boxtimes 1}", shorten <=10pt, shorten >=10pt, Rightarrow, from=0, to=1]
	\arrow["{1 \boxtimes \sigma_1}"', shorten <=10pt, shorten >=10pt, Rightarrow, from=4, to=5]
	\arrow["\alpha", shorten <=17pt, shorten >=17pt, Rightarrow, from=1, to=4]
	\arrow["{}", shorten <=28pt, shorten >=28pt, Rightarrow, from=4-1, to=2]
	\arrow["{}", shorten <=28pt, shorten >=28pt, Rightarrow, from=3, to=4-7]
\end{tikzcd}\]

\[\begin{tikzcd}[row sep = small]
	&& {C \boxtimes C} && {C \boxtimes C} \\
	\\
	&& {C \boxtimes (C\boxtimes C)} && {C \boxtimes (C \boxtimes C)} \\
	C &&&&&& C \\
	&& {(C \boxtimes C)\boxtimes C} && {(C \boxtimes C) \boxtimes C} \\
	\\
	&& {C \boxtimes C} && {C\boxtimes C}
	\arrow["\Delta", from=4-1, to=1-3]
	\arrow[""{name=0, anchor=center, inner sep=0}, "{S \boxtimes \iota\epsilon }", from=1-3, to=1-5]
	\arrow["\otimes", from=1-5, to=4-7]
	\arrow["{1 \boxtimes \Delta}", from=1-3, to=3-3]
	\arrow[""{name=1, anchor=center, inner sep=0}, "{S \boxtimes (1 \boxtimes S)}", swap, from=3-3, to=3-5]
	\arrow["1\boxtimes\otimes", from=3-5, to=1-5]
	\arrow[""{name=2, anchor=center, inner sep=0}, "\alpha", swap, to=3-3, from=5-3]
	\arrow[""{name=3, anchor=center, inner sep=0}, "\alpha", to=3-5, from=5-5]
	\arrow[""{name=4, anchor=center, inner sep=0}, "{(S \boxtimes 1) \boxtimes S}"', swap, from=5-3, to=5-5]
	\arrow["\Delta\boxtimes1"', from=7-3, to=5-3]
	\arrow["\Delta", swap, from=4-1, to=7-3]
	\arrow[""{name=5, anchor=center, inner sep=0}, "{1 \boxtimes \iota\epsilon}"', from=7-3, to=7-5]
	\arrow["{\otimes  \boxtimes 1}", swap, from=5-5, to=7-5]
	\arrow["\otimes"', from=7-5, to=4-7]
	\arrow["{1 \boxtimes \sigma_2}", shorten <=10pt, shorten >=10pt, Rightarrow, from=0, to=1]
	\arrow["{\sigma_1 \boxtimes 1}"', shorten <=10pt, shorten >=10pt, Rightarrow, from=4, to=5]
	\arrow["\alpha", shorten <=17pt, shorten >=17pt, Rightarrow, from=1, to=4]
	\arrow["{}", shorten <=28pt, shorten >=28pt, Rightarrow, from=4-1, to=2]
	\arrow["{}", shorten <=28pt, shorten >=28pt, Rightarrow, from=3, to=4-7]
\end{tikzcd}\]
Here, $\alpha$ is the associator on $\TwoVec$, and the pentagons are the pentagonators for the monoidal and comonoidal structures on $C$.  In order to leverage our constructions from the previous section, it will become necessary to assume the duals on our fusion 2-categories are functorial, but not monoidal. See \cite{DECOPPET2022107029} for a construction of a functorial dual starting with any dual.
\begin{thm} \label{2HopfReconstruction}
	The functors $\cQ$ and $\Mod$ are inverse equivalences between the 2-category of 2-Hopf algebras and the sub 2-category of $\TwoSlice$ consisting of fusion 2-categories.
\end{thm}
\begin{proof}
	From $(\cC, \cF)$ in $\TwoSlice$ with a left dual 2-functor, we have the morphism
% https://q.uiver.app/#q=WzAsMyxbMCwwLCJcXGNDXnswLDEsMlxcb3B7fX0iXSxbMiwwLCJcXGNDIl0sWzEsMiwiXFxUd29WZWMiXSxbMCwxLCJ4IFxcbWFwc3RvIHheKiJdLFsxLDIsIlxcY0YiXSxbMCwyLCIqXFxjRiIsMl0sWzUsMSwiXFxkZWx0YSIsMix7InNob3J0ZW4iOnsic291cmNlIjo0MCwidGFyZ2V0Ijo0MH19XV0=
\[\begin{tikzcd}
	{\cC^{0,1\op{}}} && \cC \\
	\\
	& \TwoVec
	\arrow["{c \mapsto ^*c}", from=1-1, to=1-3]
	\arrow["\cF", from=1-3, to=3-2]
	\arrow[""{name=0, anchor=center, inner sep=0}, "{^*\cF}"', from=1-1, to=3-2]
	\arrow["\delta"', shorten <=22pt, shorten >=22pt, Rightarrow, from=0, to=1-3]
\end{tikzcd}\]
	which induces a map $\End(\cF) \mapsto \End(^*(-) \circ \cF)$. Whiskering with the right dual functor and using the isomorphism $((^*(-))^* \circ \cF) \Rightarrow \cF$ \cite[Lemma 1.1.5]{weakfusion}, we have morphism $\cS \colon \End(\cF) \to \End(\cF)$, with 1-cell components agreeing with that for 1-categories:
	\begin{equation*}
		\cS(\eta)_{c} = (\delta_c^{-1}\eta_{{}^{*}c}\delta_c)^*
	\end{equation*}
	Denoting interchangers by $\phi$ and instances of $\delta$ by $\bullet$, the naturator of $\cS(\eta)$ is: 
	\begin{align}\label{2SnakeNaturator}
		\tikzmath{
			\draw[thick] (0,-1) -- (0, 2.5)  arc (180:0:.5cm) -- (1,.5)  arc (-180:0:.5cm) -- (2, 4);
			% additional draw args, center, radius, additional left x-space, additional right x-space, contents
			\roundNbox{fill=white}{(1,1.5)}{.3}{.05}{.05}{$\eta$};
			\roundNbox{fill=white}{(0,-.5)}{.3}{.05}{.05}{${\cF f}$};
			\filldraw[black] (1,2) circle (1.5pt);
			\filldraw[black] (1,1) circle (1.5pt);
		}
		~~&\underset{\phi}{\Rightarrow}~~
		\tikzmath{
			\draw[thick] (0,-1) -- (0, 2.5)  arc (180:0:.5cm) -- (1,.5)  arc (-180:0:.5cm) -- (2, 4);
			% additional draw args, center, radius, additional left x-space, additional right x-space, contents
			\roundNbox{fill=white}{(1,1)}{.3}{.05}{.05}{$\eta$};
			\roundNbox{fill=white}{(0,2)}{.3}{.05}{.05}{${\cF f}$};
			\filldraw[black] (1,1.5) circle (1.5pt);
			\filldraw[black] (1,.5) circle (1.5pt);
		}
		~~\underset{}{\Rightarrow}~~
		\tikzmath{
			\draw[thick] (0,-1) -- (0, 2.5)  arc (180:0:.5cm) -- (1,.5)  arc (-180:0:.5cm) -- (2, 4);
			% additional draw args, center, radius, additional left x-space, additional right x-space, contents
			\roundNbox{fill=white}{(1,1)}{.3}{.05}{.05}{$\eta$};
			\roundNbox{fill=white}{(1,2)}{.3}{.25}{.25}{${^*\cF f}$};
			\filldraw[black] (1,1.5) circle (1.5pt);
			\filldraw[black] (1,.5) circle (1.5pt);
		}
		~~\underset{\text{nat.~} \bullet}{\Rightarrow}~~
		\tikzmath{
			\draw[thick] (0,-1) -- (0, 2.5)  arc (180:0:.5cm) -- (1,.5)  arc (-180:0:.5cm) -- (2, 4);
			% additional draw args, center, radius, additional left x-space, additional right x-space, contents
			\roundNbox{fill=white}{(1,1)}{.3}{.05}{.05}{$\eta$};
			\roundNbox{fill=white}{(1,1.75)}{.3}{.25}{.25}{${\cF (^*f)}$};
			\filldraw[black] (1,2.5) circle (1.5pt);
			\filldraw[black] (1,.5) circle (1.5pt);
		} \nonumber \\
		~~\underset{\text{nat.~} \eta}{\Rightarrow}~~
		\tikzmath{
			\draw[thick] (0,-1) -- (0, 2.5)  arc (180:0:.5cm) -- (1,.5)  arc (-180:0:.5cm) -- (2, 4);
			% additional draw args, center, radius, additional left x-space, additional right x-space, contents
			\roundNbox{fill=white}{(1,1.75)}{.3}{.05}{.05}{$\eta$};
			\roundNbox{fill=white}{(1,1)}{.3}{.25}{.25}{${\cF (^*f)}$};
			\filldraw[black] (1,2.5) circle (1.5pt);
			\filldraw[black] (1,.5) circle (1.5pt);
		} 
		~~&\underset{\text{nat.~} \delta}{\Rightarrow}~~
		\tikzmath{
			\draw[thick] (0,-1) -- (0, 2.5)  arc (180:0:.5cm) -- (1,.5)  arc (-180:0:.5cm) -- (2, 4);
			% additional draw args, center, radius, additional left x-space, additional right x-space, contents
			\roundNbox{fill=white}{(1,2)}{.3}{.05}{.05}{$\eta$};
			\roundNbox{fill=white}{(1,1)}{.3}{.25}{.25}{${^*\cF f}$};
			\filldraw[black] (1,2.5) circle (1.5pt);
			\filldraw[black] (1,1.5) circle (1.5pt);
		} 
		~~\underset{}{\Rightarrow}~~
		\tikzmath{
			\draw[thick] (0,-1) -- (0, 2.5)  arc (180:0:.5cm) -- (1,.5)  arc (-180:0:.5cm) -- (2, 4);
			% additional draw args, center, radius, additional left x-space, additional right x-space, contents
			\roundNbox{fill=white}{(1,2)}{.3}{.05}{.05}{$\eta$};
			\roundNbox{fill=white}{(2,1)}{.3}{.05}{.05}{${\cF f}$};
			\filldraw[black] (1,2.5) circle (1.5pt);
			\filldraw[black] (1,1.5) circle (1.5pt);
		}
		~~\underset{\phi}{\Rightarrow}~~
		\tikzmath{
			\draw[thick] (0,-1) -- (0, 2.5)  arc (180:0:.5cm) -- (1,.5)  arc (-180:0:.5cm) -- (2, 4);
			% additional draw args, center, radius, additional left x-space, additional right x-space, contents
			\roundNbox{fill=white}{(1,1.5)}{.3}{.05}{.05}{$\eta$};
			\roundNbox{fill=white}{(2,3.5)}{.3}{.05}{.05}{${\cF f}$};
			\filldraw[black] (1,2) circle (1.5pt);
			\filldraw[black] (1,1) circle (1.5pt);
		} 
	\end{align}
	The unlabeled 2-cells are built from cusps and interchangers as follows: The first is an instance of the general 2-isomorphism:
	\begin{equation}\label{CuspTransferEquation}
		\tikzmath{
			\draw[thick] (0,0) -- (0, 3.5) arc (180:0:.5cm) -- (1,0);
			\roundNbox{fill=white}{(0,3)}{.3}{.05}{.05}{${g}$};
		} ~~\underset{\text{cusp}}{\Rightarrow}~~
		\tikzmath{
			\draw[thick] (-2,0) -- (-2,2) arc (180:0:.5cm) -- (-1,.5) arc (-180:0:.5cm) -- (0, 3.5) arc (180:0:.5cm) -- (1,0);
			\roundNbox{fill=white}{(0,3)}{.3}{.05}{.05}{${g}$};
		}
		~~\underset{\phi}{\Rightarrow}~~
		\tikzmath{
			\draw[thick] (-2,0) -- (-2,3.5) arc (180:0:.5cm) -- (-1,.5) arc (-180:0:.5cm) -- (0, 2) arc (180:0:.5cm) -- (1,0);
			\roundNbox{fill=white}{(0,1.5)}{.3}{.05}{.05}{${g}$};
		}
		~~=~~
		\tikzmath{
			\draw[thick] (0,0) -- (0, 3.5) arc (180:0:.5cm) -- (1,0);
			\roundNbox{fill=white}{(1,3)}{.3}{.05}{.05}{${^*g}$};
		}
		,
	\end{equation}and the second one is similar. Next, any two choices of duals for $X$, and canonical isomorphism $\bullet$ between them, we have the 2-isomorphism pictured below, with the duals depicted by the dotted and dashed strings.  
\begin{equation}\label{EvaluationReplacement}
\tikzmath{
	\draw[thick] (0,0) -- (0,.5);
	\draw[thick, dashed] (0, .5) arc (180:0:.5cm);
	\draw[thick, dotted] (1, .5) -- (1, 0);
	\filldraw[black] (1,.5) circle (1.5pt);
} 
\quad \coloneqq \quad 
\tikzmath{
	\draw[thick, dashed] (0,0)--(0, .5) arc (180:0:.5cm) -- (1, -1.5) arc (-180:0:0.5cm);
	\draw[thick, black] (2, -1.5) -- (2, -.5);
	\draw[thick, dotted] (2, -.5) arc (180:0:.5cm) -- (3, -2);
} 
\quad \underset{\text{cusp}}{\Rightarrow} \quad 
\tikzmath{
	\draw[thick] (0,0) -- (0,.5);
	\draw[thick, dotted] (0, .5) arc (180:0:.5cm) -- (1, 0);
}
\end{equation}
There is of course a similar one for the coevaluation. These morphisms commute with \eqref{CuspTransferEquation} by naturality of the interchanger.
	Then we have the natural isomorphism:
	\begin{equation} \label{2AntipodeAxiomSnake}
		\mu \circ (1 \otimes \cS)\circ \Delta(\eta)_X \cong 
		\tikzmath{
			\draw[thick] (0,-1) -- (0, 1.5)  arc (180:0:.5cm) -- (1,.5)  arc (-180:0:.5cm) -- (2, 2.666);
			% additional draw args, center, radius, additional left x-space, additional right x-space, contents
			\roundNbox{fill=white}{(1,1)}{.3}{.05}{.05}{${\eta_{(2)}}$};
			\roundNbox{fill=white}{(0,-.5)}{.3}{.05}{.05}{${\eta_{(1)}}$};
			\filldraw[black] (1,1.5) circle (1.5pt);
			\filldraw[black] (1,.5) circle (1.5pt);
		} \quad, 
	\end{equation}
	which appears different from  $\eqref{1AntipodeAxiomSnake}$ since the interchanger is required to move $\eta_{(1)}$ past the coevaluation. The naturator simply moves $f$ past $\eta_{(1)}$ and then repeats the movie $\eqref{2SnakeNaturator}$. 
	We must next finally define an invertible modification $ \Omega \colon \mu \circ (1 \otimes S)\circ \Delta(\eta)\Rrightarrow \iota \circ \epsilon(\eta)$. To do so, we use the definition of Sweedler notation, the canonical 2-morphisms \eqref{EvaluationReplacement}, and the naturality of $\eta$ as follows.
	\begin{align}\label{OmegaDefinition}
		\Omega_X \coloneqq&
		\tikzmath{
			\draw[thick] (0,-1) -- (0, 2.5)  arc (180:0:.5cm) -- (1,.5)  arc (-180:0:.5cm) -- (2, 4);
			% additional draw args, center, radius, additional left x-space, additional right x-space, contents
			\roundNbox{fill=white}{(1,1.5)}{.3}{.05}{.05}{${\eta_{(2)}}$};
			\roundNbox{fill=white}{(0,-.5)}{.3}{.05}{.05}{${\eta_{(1)}}$};
			\filldraw[black] (1,2) circle (1.5pt);
			\filldraw[black] (1,1) circle (1.5pt);
		}
		~~\underset{\phi}{\Rightarrow}~~
		\tikzmath{
			\draw[thick] (0,-1) -- (0, 2.5)  arc (180:0:.5cm) -- (1,.5)  arc (-180:0:.5cm) -- (2, 4);
			% additional draw args, center, radius, additional left x-space, additional right x-space, contents
			\roundNbox{fill=white}{(1,1.5)}{.3}{.05}{.05}{${\eta_{(2)}}$};
			\roundNbox{fill=white}{(0,1)}{.3}{.05}{.05}{${\eta_{(1)}}$};
			\filldraw[black] (1,2) circle (1.5pt);
			\filldraw[black] (1,.5) circle (1.5pt);
		}
		~~\underset{}{\Rightarrow}~~
		\tikzmath{
			\draw[thick] (0,-1) -- (0, 2.5)  arc (180:0:.5cm) -- (1,.5)  arc (-180:0:.5cm) -- (2, 4);
			% additional draw args, center, radius, additional left x-space, additional right x-space, contents
			\roundNbox{fill=white}{(.5,1.25)}{.4}{.5}{.5}{${\Delta(\eta)}$};
			\filldraw[black] (1,2) circle (1.5pt);
			\filldraw[black] (1,.5) circle (1.5pt);
		}
		~~\underset{}{\Rightarrow}~~
		\tikzmath{
			\draw[thick] (0,-1) -- (0,2);
			\draw[thick] (.2,2) -- (.2,2.7) arc (180:0:.3cm) -- (.8, 2); 
			\draw[thick] (1,2) -- (1,.5)  arc (-180:0:.5cm) -- (2, 4);
			% additional draw args, center, radius, additional left x-space, additional right x-space, contents
			\roundNbox{fill=white}{(.5,1.25)}{.4}{.5}{.5}{${\Delta(\eta)}$};
			\roundNbox{fill=white}{(.5,2.25)}{.4}{.5}{.5}{${J}$};
			%\filldraw[black] (1,2) circle (1.5pt);
			\filldraw[black] (1,.5) circle (1.5pt);
		} \\
		~~\underset{\text{def}}{\Rightarrow}~~&
		\tikzmath{
			\draw[thick] (0,-1) -- (0,1.5);
			\draw[thick] (.2,1.5) -- (.2,2.7) arc (180:0:.3cm) -- (.8, 1.5); 
			\draw[thick] (1,1.5) -- (1,.5)  arc (-180:0:.5cm) -- (2, 4);
			% additional draw args, center, radius, additional left x-space, additional right x-space, contents
			\roundNbox{fill=white}{(.5,2.25)}{.4}{.5}{.5}{${\eta}$};
			\roundNbox{fill=white}{(.5,1.25)}{.4}{.5}{.5}{${J}$};
			%\filldraw[black] (1,2) circle (1.5pt);
			\filldraw[black] (1,.5) circle (1.5pt);
		} 
		~~\underset{\text{nat.~} \eta}{\Rightarrow}~~ 
		\tikzmath{
			\draw[thick] (0,-1) -- (0,1.5);
			\draw[thick] (.2,1.5) -- (.2,2.7) arc (180:0:.3cm) -- (.8, 1.5); 
			\draw[thick] (1,1.5) -- (1,.5)  arc (-180:0:.5cm) -- (2, 4);
			% additional draw args, center, radius, additional left x-space, additional right x-space, contents
			\roundNbox{fill=white}{(.5,3.5)}{.3}{0}{0}{${\eta}$};
			\roundNbox{fill=white}{(.5,1.25)}{.4}{.5}{.5}{${J}$};
			%\filldraw[black] (1,2) circle (1.5pt);
			\filldraw[black] (1,.5) circle (1.5pt);
		}
		~~\underset{}{\Rightarrow}~~
		\tikzmath{
			\draw[thick] (0,-1) -- (0, 2.5)  arc (180:0:.5cm) -- (1,.5)  arc (-180:0:.5cm) -- (2, 4);
			% additional draw args, center, radius, additional left x-space, additional right x-space, contents
			%\roundNbox{fill=white}{(1,1.5)}{.3}{.05}{.05}{${\eta_{(2)}}$};
			%\roundNbox{fill=white}{(0,-.5)}{.3}{.05}{.05}{${\eta_{(1)}}$};
			\roundNbox{fill=white}{(.5,3.5)}{.3}{0}{0}{${\eta}$};
			\filldraw[black] (1,2) circle (1.5pt);
			\filldraw[black] (1,1) circle (1.5pt);
		}
		~~\underset{\text{cusp}}{\Rightarrow}~~
		\tikzmath{
			\draw[thick] (1, -1) -- (1, 4);
			\roundNbox{fill=white}{(.5,1.5)}{.3}{0}{0}{${\eta}$};
		}   \nonumber
	\end{align}

	We verify in Propositions \ref{2AntipodeAxiomVerification} and \ref{2AntipodeAxiomNaturality} that this morphism satisfies the modification axiom, and is natural in $X$; that is the modifications above form a natural transformation between the functors $(\mu \circ (1 \otimes S)\circ \Delta)$ and $\iota \circ \epsilon$.
	% and
	% $\mu \circ (1 \otimes \cS) (\eta \boxtimes \eta')_{f\boxtimes f'}$ is given by 
	
	% \[
	% \tikzmath{
		%     \draw[thick] (0,-.75) -- (0, .8) arc (180:0:.4cm) -- (.8,.2)  arc (-180:0:.4cm) -- (1.6, 1.5);
		%     \roundNbox{fill=white}{(0,.5)}{.3}{0}{0}{${\eta_X}$}
		%     \roundNbox{fill=white}{(.8,.5)}{.3}{0}{0}{${\eta'_{Y^*}}$}
		%     \roundNbox{fill=white}{(0, -.1)}{.3}{0}{0}{${\eta_f}$}
		% }
	% \overset{\eta_{\text{cusp}_X}}{\Longrightarrow}
	% \tikzmath{
		%     \draw[thick] (0,-.75) -- (0, .8) arc (180:0:.4cm) -- (.8,.2)  arc (-180:0:.4cm) -- (1.6, 1.5);
		%     \roundNbox{fill=white}{(0,.5)}{.3}{0}{0}{${\eta_X}$}
		%     \roundNbox{fill=white}{(.8,.5)}{.3}{0}{0}{${\eta'_{Y^*}}$}
		% }
	
	%\
	
	Likewise, from the antipode $\cS$ on a 2-Hopf algebra, given $\rho \colon C \to \End(M)$ we obtain the following  on $\Mod(C)$:
	\begin{align*}
		\rho^* &\coloneqq~ (\rho \circ \cS)^*\\
		^*\rho &\coloneqq ~^*(\rho \circ \cS^{-1})
	\end{align*}
	Since the antipode is always an algebra antihomomorphism \cite[47]{neuchl}, these composites are monoidal functors.
	The natural transformations given as part of the antipode provide the remaining duality data. Since the space of antipodes on a 2-bialgebra is contractible, see \cite[Lemma 3.16]{neuchl}, we are finished. 
\end{proof}
\begin{remark}

The Hopf axiom \eqref{Hopf2Axiom} is satisfied when the cusps are chosen to satisfy the swallowtail equations. In this framework, the fact that every dual can be made coherent \cite[Corollary 2.8]{pstrągowski2014dualizable} corresponds precisely to the fact that any natural isomorphism constructed above can be modified to satisfy \eqref{Hopf2Axiom} as stated by Neuchl. 

\end{remark}
We are now able to construct the double of a Hopf category. For any monoidal 2-category $\cC$, there exists a Drinfeld center $\cZ(\cC)$ \cite{Baez1995HigherDA}, which is finite semisimple if $\cC$ is \cite{decoppet2023drinfeld}. Additionally, $\cZ(\cC)$ comes equipped with a locally faithful, monoidal forgetful functor to $\cC$. 
\begin{construction}
The quantum double of a Hopf category $C$ is the Hopf category $\End(\tilde \cF)$, where $\tilde \cF$ is the composite fiber functor $\cZ(\Mod(C)) \to \Mod(C) \to \TwoVec$. 
\end{construction}
% \section{Applications}
% \begin{cor}
	%     Let $\cC$ be a connected fusion 2-category with fiber functor $\cF$ and unit object $I$. Then $\End(I)$ is a fusion category and $\cC$ is isomorphic to $\Mod(\Mod(\End(\cF_{I, I}))$.   
	% \end{cor}
% \begin{proof}
	%     Since $\cC$ is connected, it is of the form $\Mod(\cB) = \End(\cF)$, and $\Hom(I, I) = \cB$. Thus $\cF_{I, I}$ is a monoidal functor $\cB \to \Hom(\Vec, \Vec) = \Vec$. Since $\cB$ is fusion, $\cF_{I, I}$ is fiber. Therefore $\End(\cF) \sim\cB \sim \Mod(\End(\cF_{I, I}))$. 
	% \end{proof}
% \nn{DONE}

%%%%%%%%%%%%%%%%%%%%%%%%%%%%%%%%%%%%%%%%%%%%%%%%%%%%%%%%%%%%%%%%%%%%%

\appendix

\section{Calculations for 2-categories}\label{2CatAppendix}
In this appendix we explicitly define the symmetric monoidal 2-category structure on $\TwoSlice$ and verify that $\cQ$ is a symmetric monoidal functor, as well as the antipode axiom. 
\subsection{The monoidal 2-category 2-slice}
To start, we work with any finite semisimple 2-category $\cV$. We work in a slightly more general context so that these results will apply to the underlying 2-functors in the enriched case. 

\begin{defn} \label{2SliceDefn}
	$\TwoSlice_\cV$ is the 3-category with the following $\Hom$-sets.  
	\begin{itemize}
		\item Objects $(\cC, \cF)$ where $\cC$ is a finite semisimple 2-category and $\cF$ is a locally faithful 2-functor $\cC \to \cV$. 
		\item A 1-morphism from $(\cC, \cF)$ to $(\cD, \cG)$ is a pair $(\cT, \tau)$ where $\cT \colon \cC \to \cD$ is a pair $(\cT, \tau)$ where $\cT$ is a 2-functor and $\cT \colon \cF \Rightarrow \cD$ and $\tau$ is a natural equivalence $\cF \Rightarrow \cG\cT$. 
		\item A 2-morphism from $(\cT, \tau)$ to $(\cT', \tau')$ is a pair $(\sigma, \Sigma)$ where $\sigma$ is a natural transformation $\cT' \Rightarrow \cT$ and $\Sigma$ is an invertible modification as below:
		% https://q.uiver.app/#q=WzAsOCxbMSwyLCJcXGNWIl0sWzAsMCwiXFxjQyJdLFsyLDAsIlxcY0QiXSxbMiwxXSxbNCwxXSxbNiwwLCJcXGNEIl0sWzUsMiwiXFxjViJdLFs0LDAsIlxcY0MiXSxbMSwyLCJcXGNUIl0sWzEsMCwiXFxjRiIsMl0sWzIsMCwiXFxjRyJdLFs1LDYsIlxcY0ciXSxbNyw1LCJcXGNUJyIsMix7ImxhYmVsX3Bvc2l0aW9uIjozMCwiY3VydmUiOjJ9XSxbNyw2LCJcXGNGIiwyXSxbNyw1LCJcXGNUIiwwLHsiY3VydmUiOi0yfV0sWzMsNCwiXFxTaWdtYSIsMCx7ImxldmVsIjozfV0sWzksMiwiXFx0YXUiLDIseyJzaG9ydGVuIjp7InNvdXJjZSI6NDAsInRhcmdldCI6NDB9fV0sWzEzLDUsIlxcdGF1JyIsMix7ImxhYmVsX3Bvc2l0aW9uIjozMCwib2Zmc2V0Ijo0LCJzaG9ydGVuIjp7InNvdXJjZSI6MjAsInRhcmdldCI6NjB9fV0sWzEyLDE0LCJcXHNpZ21hIiwyLHsic2hvcnRlbiI6eyJzb3VyY2UiOjIwLCJ0YXJnZXQiOjIwfX1dXQ==
		\[\begin{tikzcd}
			\cC && \cD && \cC && \cD \\
			&& {} && {} \\
			& \cV &&&& \cV
			\arrow["\cT", from=1-1, to=1-3]
			\arrow[""{name=0, anchor=center, inner sep=0}, "\cF"', from=1-1, to=3-2]
			\arrow["\cG", from=1-3, to=3-2]
			\arrow["\cG", from=1-7, to=3-6]
			\arrow[""{name=1, anchor=center, inner sep=0}, "{\cT'}"', curve={height=12pt}, from=1-5, to=1-7]
			\arrow[""{name=2, anchor=center, inner sep=0}, "\cF"', from=1-5, to=3-6]
			\arrow[""{name=3, anchor=center, inner sep=0}, "\cT", curve={height=-12pt}, from=1-5, to=1-7]
			\arrow["\Sigma", triple, from=2-3, to=2-5]
			\arrow["\tau"', shorten <=16pt, shorten >=16pt, Rightarrow, from=0, to=1-3]
			\arrow["{\tau'}"'{pos=0.3}, shift right=4, shorten <=8pt, shorten >=24pt, Rightarrow, from=2, to=1-7]
			\arrow["\sigma"', shorten <=3pt, shorten >=3pt, Rightarrow, from=1, to=3]
		\end{tikzcd}\]
		which is also expressible as
		% https://q.uiver.app/#q=WzAsMyxbMSwyLCJcXGNHXFxjVCJdLFswLDAsIlxcY1QiXSxbMiwwLCJcXGNHXFxjVCciXSxbMCwyLCJcXGNHXFxzaWdtYSIsMix7ImxldmVsIjoyfV0sWzEsMiwiXFxzaWdtYSIsMCx7ImxldmVsIjoyfV0sWzEsMCwiXFx0YXUiLDIseyJsZXZlbCI6Mn1dLFs1LDIsIlxcU2lnbWEiLDIseyJzaG9ydGVuIjp7InNvdXJjZSI6NDAsInRhcmdldCI6NDB9LCJsZXZlbCI6M31dXQ==
		\[\begin{tikzcd}
			\cT && {\cG\cT'} \\
			\\
			& \cG\cT
			\arrow["\cG\sigma"', Rightarrow, to=3-2, swap, from=1-3]
			\arrow["\tau'", Rightarrow, from=1-1, to=1-3]
			\arrow[""{name=0, anchor=center, inner sep=0}, "\tau"', Rightarrow, from=1-1, to=3-2]
			\arrow["\Sigma"', triple, shorten <=19pt, shorten >=19pt, from=0, to=1-3]
		\end{tikzcd}\]
		\item A 3-morphism from $(\Sigma, \sigma)$ to $(\Sigma', \sigma')$ is a modification $\Theta \colon \sigma' \Rrightarrow \sigma$ such that: 
% https://q.uiver.app/#q=WzAsNyxbMSwyLCJcXGNHXFxjVCJdLFswLDAsIlxcY1QiXSxbMiwwLCJcXGNHXFxjVCciXSxbMywxLCI9Il0sWzQsMCwiXFxjVCJdLFs1LDIsIlxcY0dcXGNUIl0sWzYsMCwiXFxjRyJdLFswLDIsIlxcY0dcXHNpZ21hIiwyLHsibGV2ZWwiOjJ9XSxbMSwyLCJcXHRhdSciLDAseyJsZXZlbCI6Mn1dLFsxLDAsIlxcdGF1IiwyLHsibGV2ZWwiOjJ9XSxbNCw1LCJcXHRhdSIsMix7ImxldmVsIjoyfV0sWzQsNiwiXFx0YXUnIiwwLHsibGV2ZWwiOjJ9XSxbNiw1LCJcXGNHXFxzaWdtYSciLDIseyJsYWJlbF9wb3NpdGlvbiI6NzAsImN1cnZlIjoyLCJsZXZlbCI6Mn1dLFs2LDUsIlxcY0dcXHNpZ21hIiwwLHsiY3VydmUiOi0yLCJsZXZlbCI6Mn1dLFs5LDIsIlxcU2lnbWEiLDIseyJzaG9ydGVuIjp7InNvdXJjZSI6NDAsInRhcmdldCI6NDB9LCJsZXZlbCI6M31dLFsxMCw2LCJcXFNpZ21hJyIsMCx7ImxhYmVsX3Bvc2l0aW9uIjozMCwib2Zmc2V0IjotMywic2hvcnRlbiI6eyJzb3VyY2UiOjMwLCJ0YXJnZXQiOjUwfSwibGV2ZWwiOjN9XSxbMTIsMTMsIiIsMix7InNob3J0ZW4iOnsic291cmNlIjoyMCwidGFyZ2V0IjoyMH0sImxldmVsIjozfV1d
\[\begin{tikzcd}
	\cT && {\cG\cT'} && \cT && \cG\cT' \\
	&&& {=} \\
	& \cG\cT &&&& \cG\cT
	\arrow["\cG\sigma"', Rightarrow, from=3-2, to=1-3]
	\arrow["{\tau'}", Rightarrow, from=1-1, to=1-3]
	\arrow[""{name=0, anchor=center, inner sep=0}, "\tau"', Rightarrow, from=1-1, to=3-2]
	\arrow[""{name=1, anchor=center, inner sep=0}, "\tau"', Rightarrow, from=1-5, to=3-6]
	\arrow["{\tau'}", Rightarrow, from=1-5, to=1-7]
	\arrow[""{name=2, anchor=center, inner sep=0}, "{\cG\sigma'}"'{pos=0.7}, curve={height=15pt}, Rightarrow, from=1-7, to=3-6]
	\arrow[""{name=3, anchor=center, inner sep=0}, "\cG\sigma", curve={height=-15pt}, Rightarrow, from=1-7, to=3-6]
	\arrow["\Sigma"', shorten <=19pt, triple, shorten >=19pt, from=0, to=1-3]
	\arrow["{\Sigma'}"{pos=0.3}, shift left=3, triple, shorten <=14pt, shorten >=24pt, from=1, to=1-7]
	\arrow["{\cG\Theta}", triple, shorten <=5pt, shorten >=5pt, from=2, to=3]
\end{tikzcd}\]
	\end{itemize}
	The directions of composition may at first seem strange; compare \cite{Jones_2022}, Definition 3.2. We declare at this point that composition of 3-morphisms is composition of modifications. 
\end{defn}
Before defining the rest of the composition rules, we need the following:
\begin{lem}
	$\TwoSlice_\cV$ is 2-truncated. 
\end{lem}
\begin{proof}
	$\cG(\Theta)$ can be written pointwise as $\Sigma'\Sigma^{-1}$. Since $\cG$ was locally faithful and $\cV$ is locally semisimple, $\Theta$ is also an isomorphism, and uniquely determined. 
\end{proof}
\begin{construction}
	We now define the remaining two compositions and two natural isomorphisms required to make $\TwoSlice_\cV$ a 2-category. We will not mention the 3-morphisms, as all their assignments are forced. 
	The vertical composition of two composable 2-morphisms $(\Sigma, \sigma) \colon (\cT, \tau) \Rightarrow (\cT', \tau')$ and $(\Sigma', \sigma') \colon (\cT', \tau') \Rightarrow (\cT'', \tau'')$ is $\sigma'\sigma$ along with the whiskered modification:
	% https://q.uiver.app/#q=WzAsNixbMCwwLCJcXGNHXFxjVCJdLFsyLDIsIlxcY1QiXSxbMiwwLCJcXGNHXFxjVCciXSxbMCwyXSxbNCwwLCJcXGNHXFxjVCcnIl0sWzQsMl0sWzIsMCwiXFxjR1xcc2lnbWEiLDAseyJsZXZlbCI6Mn1dLFsxLDIsIlxcdGF1JyIsMCx7ImxldmVsIjoyfV0sWzEsMCwiXFx0YXUiLDIseyJsZXZlbCI6Mn1dLFsxLDQsIlxcdGF1JyciLDAseyJsZXZlbCI6Mn1dLFs0LDIsIlxcY0dcXHNpZ21hJyIsMCx7ImxldmVsIjoyfV0sWzMsMiwiXFxTaWdtYSIsMix7ImxhYmVsX3Bvc2l0aW9uIjo4MCwic2hvcnRlbiI6eyJzb3VyY2UiOjcwLCJ0YXJnZXQiOjEwfSwibGV2ZWwiOjN9XSxbMiw1LCJcXFNpZ21hJyIsMix7ImxhYmVsX3Bvc2l0aW9uIjoyMCwic2hvcnRlbiI6eyJzb3VyY2UiOjEwLCJ0YXJnZXQiOjcwfSwibGV2ZWwiOjN9XV0=
	\[\begin{tikzcd}
		\cG\cT && {\cG\cT'} && {\cG\cT''} \\
		\\
		{} && \cT && {}
		\arrow["\cG\sigma", Rightarrow, to=1-1, from=1-3]
		\arrow["{\tau'}", Rightarrow, from=3-3, to=1-3]
		\arrow["\tau"', swap, Rightarrow, from=3-3, to=1-1]
		\arrow["{\tau''}", swap, Rightarrow, from=3-3, to=1-5]
		\arrow["{\cG\sigma'}", Rightarrow, from=1-5, to=1-3]
		\arrow["\Sigma"'{pos=0.8}, triple, shorten <=55pt, shorten >=6pt, from=3-1, to=1-3]
		\arrow["{\Sigma'}"'{pos=0.2}, triple, shorten <=6pt, shorten >=55pt, from=1-3, to=3-5]
	\end{tikzcd}\]
	We next define the horizontal composition functors on 1-morphisms by
	\begin{equation*}
		\begin{tikzcd}
			\cC && \cD \\
			\\
			& \cV
			\arrow["\cT", from=1-1, to=1-3]
			\arrow[""{name=0, anchor=center, inner sep=0}, "\cF"', from=1-1, to=3-2]
			\arrow["\cG", from=1-3, to=3-2]
			\arrow["\tau"', shorten <=16pt, shorten >=16pt, Rightarrow, from=0, to=1-3]
		\end{tikzcd} \circ 
		\begin{tikzcd}
			\cD && \cE \\
			\\
			& \cV
			\arrow["\cK", from=1-1, to=1-3]
			\arrow[""{name=0, anchor=center, inner sep=0}, "\cG"', from=1-1, to=3-2]
			\arrow["\cH", from=1-3, to=3-2]
			\arrow["\kappa"', shorten <=16pt, shorten >=16pt, Rightarrow, from=0, to=1-3]
		\end{tikzcd}
		\coloneqq
		\begin{tikzcd}
			\cC && \cD && \cE \\
			\\
			{} && \cV && {}
			\arrow["\cT", from=1-1, to=1-3]
			\arrow["\cF"', from=1-1, to=3-3]
			\arrow["\cG", from=1-3, to=3-3]
			\arrow["\tau"{pos=0.7}, shorten <=49pt, shorten >=6pt, Rightarrow, from=3-1, to=1-3]
			\arrow["\cK", from=1-3, to=1-5]
			\arrow["\cH", from=1-5, to=3-3]
			\arrow["\kappa"{pos=0.3}, shorten <=6pt, shorten >=49pt, Rightarrow, from=1-3, to=3-5]
		\end{tikzcd}
	\end{equation*}
	and on 2-morphisms by 
	\begin{equation}\label{2Slice2MorHorizontalComposition}
		\hspace{-.5cm}
		\begin{tikzcd}
			\cT && {\cG\cT'} \\
			\\
			& \cG\cT
			\arrow["\cG\sigma"', Rightarrow, to=3-2, swap, from=1-3]
			\arrow["\tau'", Rightarrow, from=1-1, to=1-3]
			\arrow[""{name=0, anchor=center, inner sep=0}, "\tau"', Rightarrow, from=1-1, to=3-2]
			\arrow["\Sigma"', triple, shorten <=19pt, shorten >=19pt, from=0, to=1-3]
		\end{tikzcd}
		\circ 
		\begin{tikzcd}
			\cK && {\cH\cK'} \\
			\\
			& \cH\cK
			\arrow["\cH\lambda"', Rightarrow, to=3-2, swap, from=1-3]
			\arrow["\kappa'", Rightarrow, from=1-1, to=1-3]
			\arrow[""{name=0, anchor=center, inner sep=0}, "\kappa"', Rightarrow, from=1-1, to=3-2]
			\arrow["\Lambda"', triple, shorten <=19pt, shorten >=19pt, from=0, to=1-3]
		\end{tikzcd} \coloneqq % https://q.uiver.app/#q=WzAsOCxbMCwwLCJcXGNGIl0sWzIsMiwiXFxjR1xcY1QiXSxbNCw0LCJcXGNIXFxjS1xcY1QiXSxbMiwwLCJcXGNHXFxjVCciXSxbNCwyLCJcXGNIXFxjSydcXGNUIl0sWzQsMCwiXFxjSFxcY0snXFxjVCciXSxbMCwyXSxbMiw0XSxbMCwxLCJcXHRhdSIsMix7ImxldmVsIjoyfV0sWzEsMiwiXFxrYXBwYVxcY1QiLDIseyJsZXZlbCI6Mn1dLFswLDMsIlxcdGF1JyIsMCx7ImxldmVsIjoyfV0sWzMsMSwiXFxjR1xcc2lnbWEiLDAseyJsZXZlbCI6Mn1dLFsxLDQsIlxca2FwcGEnXFxjVCIsMCx7ImxldmVsIjoyfV0sWzQsMiwiXFxjSFxcbGFtYmRhXFxjVCIsMCx7ImxldmVsIjoyfV0sWzUsNCwiXFxjSFxcY0tcXHNpZ21hIiwwLHsibGV2ZWwiOjJ9XSxbMyw1LCJcXGthcHBhJ1xcY1QnIiwwLHsibGV2ZWwiOjJ9XSxbMSw1LCJcXGthcHBhJ197XFxzaWdtYX0iLDAseyJzaG9ydGVuIjp7InNvdXJjZSI6NDAsInRhcmdldCI6NDB9LCJsZXZlbCI6M31dLFs2LDMsIlxcU2lnbWEiLDAseyJsYWJlbF9wb3NpdGlvbiI6NzAsInNob3J0ZW4iOnsic291cmNlIjo3MCwidGFyZ2V0IjoxMH19XSxbNyw0LCJcXExhbWJkYVxcY1QiLDAseyJsYWJlbF9wb3NpdGlvbiI6NzAsInNob3J0ZW4iOnsic291cmNlIjo3MCwidGFyZ2V0IjoxMH0sImxldmVsIjozfV1d
		\hspace{-.75cm}\begin{tikzcd}
			\cF && {\cG\cT'} && {\cH\cK'\cT'} \\
			\\
			{} && \cG\cT && {\cH\cK'\cT} \\
			\\
			&& {} && \cH\cK\cT
			\arrow["\tau"', Rightarrow, from=1-1, to=3-3]
			\arrow["\kappa\cT"', Rightarrow, from=3-3, to=5-5]
			\arrow["{\tau'}", Rightarrow, from=1-1, to=1-3]
			\arrow["\cG\sigma", Rightarrow, from=1-3, to=3-3]
			\arrow["{\kappa'\cT}", Rightarrow, from=3-3, to=3-5]
			\arrow["\cH\lambda\cT", Rightarrow, from=3-5, to=5-5]
			\arrow["\cH\cK\sigma", Rightarrow, from=1-5, to=3-5]
			\arrow["{\kappa'\cT'}", Rightarrow, from=1-3, to=1-5]
			\arrow["{\kappa'_{\sigma}}", triple, shorten <=26pt, shorten >=26pt, from=3-3, to=1-5]
			\arrow["\Sigma"{pos=0.7}, triple ,shorten <=51pt, shorten >=6pt, from=3-1, to=1-3]
			\arrow["\Lambda\cT"{pos=0.7}, triple, shorten <=51pt, shorten >=7pt, from=5-3, to=3-5]
		\end{tikzcd}
	\end{equation} The modification $\kappa'_\sigma$ has component at $c$ the naturality 2-cell $\kappa'_{\sigma_c}$. This equation has the consequence that when horizontally composing two morphisms in $\TwoSlice_\cV$ and converting the resulting whiskered composite to a pasting diagram, an extra region appears! The pasting theorem for bicategories implies this region is generally safe to omit, but when choosing a bracketing for a diagram, it may appear. 
	The unit 1,2 and 3- morphisms are those with identity components. 
	
	The associator and unitors for 1-composition are inherited pointwise from $\cV$, and consequently satisfy the requisite axioms. 
\end{construction}
We will frequently use the following lemma, sometimes without mention. It is proven similarly to lemma \ref{EquivalenceInSlice}, by choosing biadjoint biequivalence data for $\cT$ and adjoint equivalence data for $\tau$ and $\sigma$.
\begin{lem}\label{EquivalenceInTwoSlice}
A morphism $(\cT, \tau)$ is an equivalence in $\Slice$ iff $\cT$ is an equivalence. A 2-morphism $(\sigma, \Sigma)$ is an isomorphism iff $\sigma$ is an equivalence.
\end{lem}
\begin{construction} We next construct the monoidal structure on $\TwoSlice_\cV$ using the Deligne 2-tensor product $\boxdot$. This process is exactly akin to constructing the monoidal 1-category structure on $\Vec/A$ where $A$ is an algebra or the 2-category structure on $\TwoVec/C$ when $C$ is monoidal. Let $(\cC, \cF)$ and $(\cD, \cG)$ be objects in $\TwoSlice$. The functor $\cF \times \cG$ is bilinear and so we obtain a functor $\cF \boxdot \cG \colon \cC \boxdot \cD \to \cV \boxdot \cV$. Composing with the induced morphism tensor product, we get an assignment on objects 
$$((\cC, \cF) \boxdot (\cD, \cG)) \coloneqq (\cC \boxdot \cD, \otimes_\cV (\cF \boxdot\cG)) $$
Continuing, given a pair of morphisms: 
	\begin{equation*}
	\begin{tikzcd}
		\cC && \cC' \\
		\\
		& \cV
		\arrow["\cT", from=1-1, to=1-3]
		\arrow[""{name=0, anchor=center, inner sep=0}, "\cF"', from=1-1, to=3-2]
		\arrow["{\cF'}", from=1-3, to=3-2]
		\arrow["\tau"', shorten <=16pt, shorten >=16pt, Rightarrow, from=0, to=1-3]
	\end{tikzcd}\quad,\quad 
	\begin{tikzcd}
		\cD && \cD' \\
		\\
		& \cV
		\arrow["\cK", from=1-1, to=1-3]
		\arrow[""{name=0, anchor=center, inner sep=0}, "\cG"', from=1-1, to=3-2]
		\arrow["{\cG'}", from=1-3, to=3-2]
		\arrow["\kappa"', shorten <=16pt, shorten >=16pt, Rightarrow, from=0, to=1-3]
	\end{tikzcd}
\end{equation*}
we get a functor $\cT \boxtimes \cK$ by universal property.
Since $\tau \times \kappa$ is a natural transformation $\cF \times \cG \Rightarrow \cF'\cT \times \cG'\cK = (\cF' \times \cG')(\cT \times \cK)$ we induce $\tau \boxdot \kappa$ filling the obvious triangle, and we may repeat this argument essentially unchanged to induce both the vertical and horizontal functoriality constraints of $\boxdot$. In order for $\boxdot$ to be well defined, we must have that the product of two faithful functors is again faithful. This follows from Theorem 4.6 of \cite{decoppet2023drinfeld} as we are working over an algebraically closed field. We turn our attention now to the morphisms that make up the monoidal structure.

The associator is induced as follows: There is a canonical equivalence of 2- categories \cite[Lemma 5.1]{decoppet2023finite},
$$(\cC \boxdot \cC') \boxdot \cC'' \sim_{\cC, \cC', \cC''} \cC \boxdot (\cC' \boxdot \cC'')$$

and so we can define the associator 1-morphism as: 
% https://q.uiver.app/#q=WzAsOSxbMCwwLCJcXGNDKFxcY0MnXFxjQycnKSJdLFsyLDQsIlxcY1YiXSxbMSwzLCJcXGNWXFxjViJdLFsxLDIsIihcXGNWXFxjVilcXGNWIl0sWzMsMiwiKFxcY1ZcXGNWKVxcY1YiXSxbMywzLCJcXGNWXFxjViJdLFs0LDAsIlxcY0MoXFxjQydcXGNDJycpIl0sWzEsMV0sWzMsMV0sWzMsNCwiXFxzaW1fe1xcY1YsIFxcY1YsIFxcY1Z9Il0sWzMsMiwiXFxib3h0aW1lczEiXSxbMiwxLCJcXGJveHRpbWVzIl0sWzQsNSwiMVxcYm94dGltZXMiLDJdLFs1LDEsIlxcYm94dGltZXMiLDJdLFsyLDUsIlxcYWxwaGFfXFxjViIsMCx7Im9mZnNldCI6LTMsInNob3J0ZW4iOnsic291cmNlIjo0MCwidGFyZ2V0Ijo0MH0sImxldmVsIjoyfV0sWzAsMywiXFxjRihcXGNGJ1xcY0YnJykiXSxbMCw2LCJcXHNpbV97XFxjQywgXFxjQycsXFxjQycnfSIsMl0sWzAsMSwiIiwxLHsiY3VydmUiOjV9XSxbNiw0LCJcXGNGKFxcY0YnXFxjRicnKSIsMl0sWzcsOCwiXFxzaW1fe1xcY0YsIFxcY0YnLFxcY0YnJ30iLDAseyJzaG9ydGVuIjp7InNvdXJjZSI6NDAsInRhcmdldCI6NDB9LCJsZXZlbCI6Mn1dLFs2LDEsIiIsMCx7ImN1cnZlIjotNX1dXQ==
\[\begin{tikzcd}
	{\cC(\cC'\cC'')} &&&& {\cC(\cC'\cC'')} \\
	& {} && {} \\
	& {(\cV\cV)\cV} && {(\cV\cV)\cV} \\
	& \cV\cV && \cV\cV \\
	&& \cV
	\arrow["{\sim_{\cV, \cV, \cV}}", from=3-2, to=3-4]
	\arrow["\boxtimes1", from=3-2, to=4-2]
	\arrow["\boxtimes", from=4-2, to=5-3]
	\arrow["1\boxtimes"', from=3-4, to=4-4]
	\arrow["\boxtimes"', from=4-4, to=5-3]
	\arrow["{\alpha_\cV}", shift left=3, shorten <=25pt, shorten >=25pt, Rightarrow, from=4-2, to=4-4]
	\arrow["{\cF(\cF'\cF'')}", from=1-1, to=3-2]
	\arrow["{\sim_{\cC, \cC',\cC''}}"', from=1-1, to=1-5]
	\arrow[curve={height=40pt}, from=1-1, to=5-3]
	\arrow["{\cF(\cF'\cF'')}"', from=1-5, to=3-4]
	\arrow["{\sim_{\cF, \cF',\cF''}}", shorten <=31pt, shorten >=31pt, Rightarrow, from=2-2, to=2-4]
	\arrow[curve={height=-40pt}, from=1-5, to=5-3]
\end{tikzcd}\]

We have omitted the $\boxdot$ symbol for brevity. The unlabeled arrows are the functors defined by the tensor product; they are equal to the composite with which they bound a region after precomposing with the universal functor into the product; but the induced transformation may be nonidentity. The other cells in the diagram are all induced by universal property and the fact that the 2-deligne tensor product is natural. While the morphisms $\sim_{\cC, \cC', \cC''}$ are a good candidate to begin defining a monoidal structure on $\ThreeVec$, we make no claims about the morphisms in the above diagram beyond existence and naturality (as morphisms in $\TwoSlice$, so up to necessarily invertible modification). 

The unit object is $\id_\cV$ with unitality transformations and modifications induced by whiskering with those of $\cV$ in a similar fashion. The unit and pentagonator modifications are also induced this way, with their unitality naturality properties inherited from those of $\cV$.
 
All axioms are checked simultaneously the same way. Let $M_1$ and $M_2$ be the two modifications corresponding to any axiom, with source, target and regions all corresponding to coherence morphisms. By construction, the 2-morphism $(M_1M_2^{-1})_{(\cC, \cF)}$ has underlying modification induced from whiskering by the naturality 2-morphisms of $\boxdot$, then applying the appropriate coherence data in $\cV$ pointwise. This last operation is an identity operation since $\cV$ is a monoidal 2-category. Therefore there is a 3-morphism between $M_1$ and $M_2$ and we are done.
\end{construction}
\begin{rem}
The above argument uses the fact that $\TwoSlice_\cV$ is truncated in an essential way and therefore does not use the full power of the 3-universal property of $\boxdot$, which is presumably needed to show that $\ThreeVec$ is a (symmetric) monoidal 3-category. 
\end{rem}
The same argument shows the following: 
\begin{prop} If $\cV$ is braided/sylleptic/symmetric monoidal, then so is $\TwoSlice_\cV$. 
\end{prop}
\begin{proof}
Induce the desired data by composing with that of $\cV$, and check the axioms by the argument as before. 
\end{proof}
\subsection{The functor $\cQ$}
We now assume $\cV = \TwoVec$.  
\begin{lem} \label{ModSymmetric2Monoidal}
	There is a 2-functor $\Mod(-) \colon \TwoAlg \to \TwoSlice^{1op}$. It is monoidal/braided/sylleptic/symmetric monoidal if $\cV$ is, and therefore canonically takes comonoid objects to monoid objects.
\end{lem}
\begin{proof}
The construction of the underlying functor is routine. We provide the tensorator here. We need a natural map $\Mod(C) \boxdot \Mod(D) \to \Mod(C \boxtimes D)$, commuting appropriately with the forgetful functor, up to invertible natural transformation.

We simply make explicit the conclusion of Corollaries 3.8, 3.9 of \cite{décoppet20212deligne}, along with the following discussion. We know that $\Mod(C) \boxdot \Mod(D) \simeq \Mod(C \boxtimes D)$ already, just not naturally. However as $\Mod(C) \boxdot \Mod(D)$ is locally separable (since it is equivalent to a separable 2-category), a functor $\Mod(C) \boxdot \Mod(D)$ can be induced by universal property from $\Mod(C) \times \Mod(D)$, and this in turn by a pair of functors from $\Mod(C)$ and $\Mod(D)$. These in turn are specified by their values on $C$ and $D$ (since $\Mod(C)$ and $\Mod(D)$ are Cauchy completions). We use the functors $- \boxtimes D$ and $C \boxtimes -$, respectively. The induced map $\chi_{C, D} \colon \Mod(C) \boxdot \Mod(D) \to \Mod(C \boxtimes D)$ is the tensorator, and is manifestly natural. The unitor can be induced similarly, and these functors equipped with the data of adjoint equivalences (in $\TwoSlice_\cV$).

The 2-universal property of the Deligne 2-tensor product induces the requisite natural transformation between forgetful functors. All the modifications required for a braided monoidal functor are induced by universal property (potentially using the braiding on $\cV$), and the modification equations hold pointwise in $\cV$.  
\end{proof}
\begin{lem} \label{2AlgebraAxiom1Verification}
	The morphisms constructed in \ref{2MonoidIn2Slice} satisfy the first algebra object axiom(\cite[20]{decoppet2023finite}), up to a necessarily unique 3-morphism, as a direct consequence of the first monoidal 2-functor axiom (\cite[17]{MR1261589}). 
\end{lem}
\begin{proof}
	We use a minimal amount of parentheses while remaining unambiguous about the order of tensor products, mostly omitting them around arguments of functions; i.e $\cF a(bc) \cF d$ means $(\cF(a) \boxtimes (\cF(b) \boxtimes \cF(c))) \boxtimes \cF(d)$. We expand the first expression of axiom (a) into a pasting diagram in $\cV$ and obtain: 
	
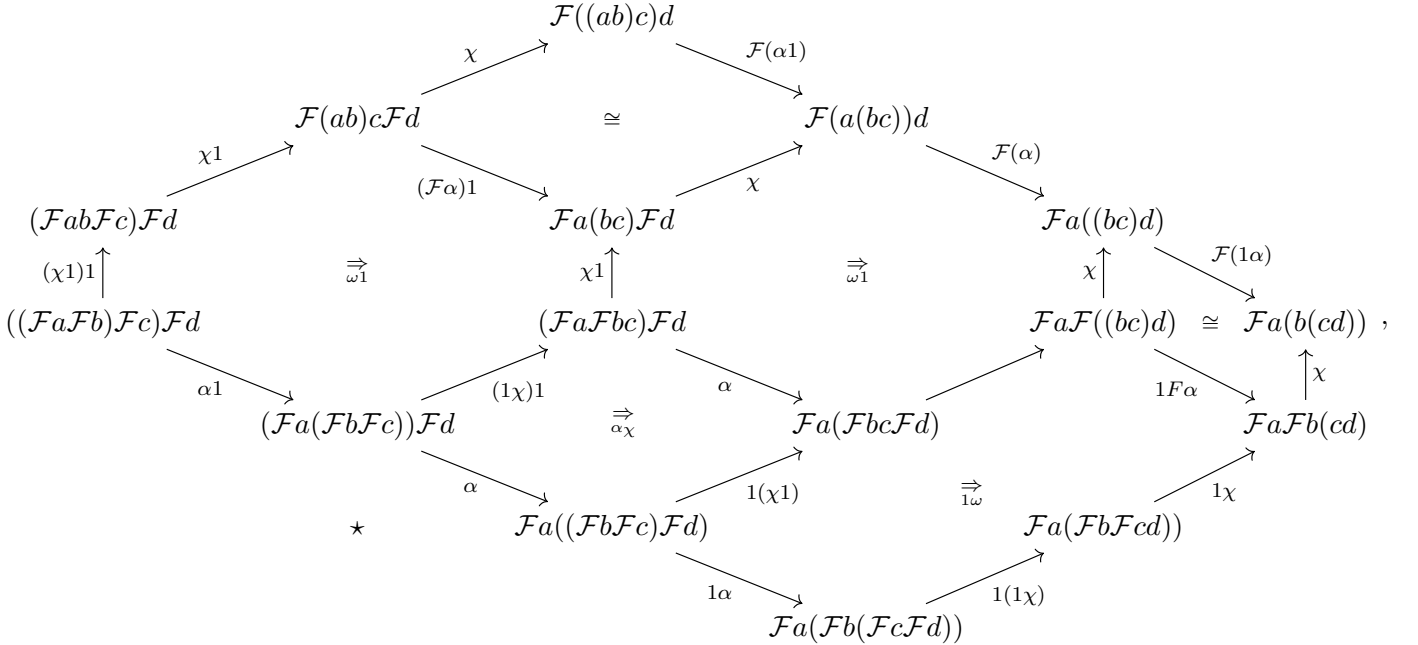
\[\hspace{-1cm}\begin{tikzcd}[column sep = small]
		&& {\cF((ab)c)d}\\
		& {\cF (ab)c\cF d} && {\cF(a(bc))d} & \\
		{(\cF ab \cF c) \cF d} && {\cF a(bc)\cF d} && {\cF a((bc)d)} \\
		{((\cF a\cF b) \cF c) \cF d} && {(\cF a \cF bc)\cF d} && {\cF a \cF((bc)d)} & {\cF a (b (c d))} \\
		& {(\cF a(\cF b \cF c))\cF d} && {\cF a (\cF bc\cF d)} && {\cF a \cF b (c d)} \\
		&\star& {\cF a ((\cF b \cF c) \cF d)} && {\cF a (\cF b \cF c d))} \\
		&&& {\cF a (\cF b (\cF c \cF d))}
		\arrow[""{name=0, anchor=center, inner sep=0}, "\chi1", from=4-3, to=3-3]
		\arrow["{(\cF\alpha)1}"', from=2-2, to=3-3]
		\arrow["\chi"', from=3-3, to=2-4]
		\arrow["{\cF(\alpha1)}", from=1-3, to=2-4]
		\arrow["\chi", from=2-2, to=1-3]
		\arrow["\alpha"', from=4-3, to=5-4]
		\arrow[from=5-4, to=4-5]
		\arrow[""{name=1, anchor=center, inner sep=0}, "\chi", from=4-5, to=3-5]
		\arrow["{(1\chi)1}"', from=5-2, to=4-3]
		\arrow["\alpha"', from=5-2, to=6-3]
		\arrow["\chi1", from=3-1, to=2-2]
		\arrow[""{name=2, anchor=center, inner sep=0}, "{(\chi 1)1}", from=4-1, to=3-1]
		\arrow["\alpha1"', from=4-1, to=5-2]
		\arrow["{\cF(1\alpha)}", from=3-5, to=4-6]
		\arrow["1F\alpha"', from=4-5, to=5-6]
		\arrow["\chi"', from=5-6, to=4-6]
		\arrow["1\alpha"', from=6-3, to=7-4]
		\arrow["{1(1\chi)}"', from=7-4, to=6-5]
		\arrow[""{name=3, anchor=center, inner sep=0}, "1\chi"', from=6-5, to=5-6]
		\arrow["{\cF(\alpha)}", from=2-4, to=3-5]
		\arrow["\cong"{description}, draw=none, from=2-2, to=2-4]
		\arrow[""{name=4, anchor=center, inner sep=0}, "{1(\chi1)}"', from=6-3, to=5-4]
		\arrow["{\underset{\alpha_\chi}{\Rightarrow}}"{description}, draw=none, from=5-2, to=5-4]
		\arrow["\cong"{description}, draw=none, from=4-5, to=4-6]
		\arrow["{\underset{\omega1}{\Rightarrow}}"{description}, draw=none, from=2, to=0]
		\arrow["{\underset{\omega1}{\Rightarrow}}"{description}, draw=none, from=0, to=1]
		\arrow["{\underset{1\omega}{\Rightarrow}}"', draw=none, from=4, to=3]
	\end{tikzcd},\]
	where we have designated some regions with stars. The areas marked with $\cong$ appear as coherences from the equation for horizontal composition. \eqref{2Slice2MorHorizontalComposition}.
	Repeating for the second expression, we acquire
	\[\begin{tikzcd}[column sep=small]
		&&& {\star \star} \\
		&& {\cF((ab)c)d} & {\cF(ab)(cd)} & {\cF a (b (c d))} \\
		& {\cF (ab)c\cF d} & {\cF ab\cF cd} && {\cF a \cF b (c d)} \\
		{(\cF ab \cF c) \cF d} & {\cF a b(\cF c \cF  d)} & {(\cF a \cF b)\cF cd} & {\cF a (\cF b \cF c d))} \\
		{((\cF a\cF b) \cF c) \cF d} & {(\cF a \cF b)(\cF c \cF  d)} & {\cF a (\cF b (\cF c \cF d))} \\
		{(\cF a(\cF b \cF c))\cF d} & {\cF a ((\cF b \cF c) \cF d)}
		\arrow[""{name=0, anchor=center, inner sep=0}, "\chi", from=3-2, to=2-3]
		\arrow["\alpha"', from=6-1, to=6-2]
		\arrow["\chi1", from=4-1, to=3-2]
		\arrow[""{name=1, anchor=center, inner sep=0}, "{(\chi 1)1}", from=5-1, to=4-1]
		\arrow[""{name=2, anchor=center, inner sep=0}, "\alpha1"', from=5-1, to=6-1]
		\arrow[""{name=3, anchor=center, inner sep=0}, "\chi"', from=3-5, to=2-5]
		\arrow[""{name=4, anchor=center, inner sep=0}, "1\alpha"', from=6-2, to=5-3]
		\arrow["1\chi"', from=4-4, to=3-5]
		\arrow[""{name=5, anchor=center, inner sep=0}, "{\chi(11)}"', from=5-2, to=4-2]
		\arrow[""{name=6, anchor=center, inner sep=0}, "11\chi"', from=5-2, to=4-3]
		\arrow["\alpha", from=5-1, to=5-2]
		\arrow["\alpha"', from=5-2, to=5-3]
		\arrow[""{name=7, anchor=center, inner sep=0}, "\chi1"', from=4-3, to=3-3]
		\arrow["1\chi", from=4-2, to=3-3]
		\arrow["\alpha", from=4-1, to=4-2]
		\arrow[""{name=8, anchor=center, inner sep=0}, swap, "\chi"', from=3-3, to=2-4]
		\arrow["{\cF(\alpha)}"', swap,from=2-4, to=2-5]
		\arrow["\alpha"', from=4-3, to=4-4]
		\arrow["{\underset{\phi}{\Rightarrow}}"{description}, draw=none, from=4-2, to=4-3]
		\arrow["{\cF(\alpha)}", from=2-3, to=2-4]
		\arrow[""{name=9, anchor=center, inner sep=0}, "{1(1\chi)}"', from=5-3, to=4-4]
		\arrow["{\underset{\omega}{\Rightarrow}}"{description}, draw=none, from=0, to=8]
		\arrow["{\underset{\omega}{\Rightarrow}}"{description}, draw=none, from=7, to=3]
		\arrow["{\underset{\alpha_\chi}{\Rightarrow}}"{description}, draw=none, from=1, to=5]
		\arrow["{\underset{\Pi_\TwoVec}{\Rightarrow}}"{description}, draw=none, from=2, to=4]
		\arrow["{\underset{\alpha_\chi}{\Rightarrow}}"{description}, draw=none, from=6, to=9]
	\end{tikzcd}\]
	
	These 2-morphisms are not even parallel! By the definition of 3-morphisms in $\TwoSlice$, we may insert the image of the pentagonator of $\cC$ into the region marked $\star\star$ in the second figure.  After moving the pentagonator of $\cV$ to the region marked $\star$ in the first diagram; we have the two expressions in the monoidal 2-functor axiom. 
\end{proof}
\begin{lem} \label{2Algebra1MorphismAxiom1Verification}
The morphisms constructed in \ref{2MonoidIn2Slice} satsify the first axiom for a 1-morphism of algebra objects(\cite[21]{decoppet2023finite}), up to a necessarily unique 3-morphism, as a direct consequence of the first monoidal natural transformation axiom (\cite[95]{schommer-pries-thesis}). 
\end{lem}
\begin{proof}
We use similar conventions as in the previous proof. The first pasting diagram is:
\[\begin{tikzcd}
	\star && {(\cF'\cT a \cF'\cT b) \cF c} \\
	&&& {(\cF'\cT a \cF'\cT b) \cF'\cT c} \\
	\\
	{(\cF'\cT a \cF b)\cF c} && {\cF'\cT ab\ \cF c} && {\cF'\cT ab \cF'\cT c} \\
	& {\cF ab\cF c } && {\cF'\cT a\cF'\cT bc} \\
	{(\cF a\cF b)\cF c} && {\cF (ab)c} \\
	&&&& {\cF'\cT (ab)c} \\
	{\cF a (\cF b \cF c)} && {\cF a(bc)} \\
	& {\cF a\cF bc} &&& {\cF'\cT a(bc)}
	\arrow[""{name=0, anchor=center, inner sep=0}, "\alpha", from=6-1, to=8-1]
	\arrow["1\chi", from=8-1, to=9-2]
	\arrow["\chi", from=9-2, to=8-3]
	\arrow["\tau", from=8-3, to=9-5]
	\arrow[""{name=1, anchor=center, inner sep=0}, "\cF\alpha"', from=6-3, to=8-3]
	\arrow[""{name=2, anchor=center, inner sep=0}, "{\cF'\cT \alpha}", from=7-5, to=9-5]
	\arrow["\tau"', from=6-3, to=7-5]
	\arrow["\chi1"', from=6-1, to=5-2]
	\arrow["\chi"', from=5-2, to=6-3]
	\arrow["{(\tau1)1}"', from=6-1, to=4-1]
	\arrow["{(1\tau)1}"', from=4-1, to=1-3]
	\arrow["\tau1"', from=5-2, to=4-3]
	\arrow["{\chi_{\cF'\cT}}", from=5-4, to=7-5]
	\arrow["1\tau", from=4-3, to=5-4]
	\arrow[""{name=3, anchor=center, inner sep=0}, "{\chi_{\cF'\cT}1}", from=1-3, to=4-3]
	\arrow["{(11)\tau}"', from=1-3, to=2-4]
	\arrow[""{name=4, anchor=center, inner sep=0}, "{\chi_{\cF'\cT}1}", from=2-4, to=4-5]
	\arrow["1\tau"', from=4-3, to=4-5]
	\arrow["{\underset{\Pi^{-1}1}{\Rightarrow}}"{description}, draw=none, from=4-1, to=4-3]
	\arrow["{\underset{\Pi^{-1}}{\Rightarrow}}"{description}, draw=none, from=5-2, to=5-4]
	\arrow["{\chi_{\cF'\cT}}", from=4-5, to=7-5]
	\arrow["{\underset{\phi_{\tau,\chi}}{\Rightarrow}}"{description}, draw=none, from=5-4, to=4-5]
	\arrow["{\underset{\omega}{\Rightarrow}}"{description}, draw=none, from=0, to=1]
	\arrow["{\underset{\tau_\chi}{\Rightarrow}}"{description}, draw=none, from=3, to=4]
	\arrow["{\underset{\tau_\alpha}{\Rightarrow}}"{description}, draw=none, from=1, to=2]
\end{tikzcd}\]
where both the regions involving naturality of $\tau$ come from \eqref{2Slice2MorHorizontalComposition}. The other pasting diagram is 

\[\vspace{-.5cm}\begin{tikzcd}
	{(\cF a\cF b)\cF c} \\
	{\cF a(\cF b\cF c)} && {(\cF'\cT a \cF b)\cF c} & {(\cF'\cT a \cF'\cT b)\cF c} & \star \\
	{\cF a \cF bc} && {\cF'\cT a (\cF b\cF c)} & {\cF'\cT a (\cF'\cT b\cF c)} & {(\cF'\cT a \cF'\cT b)\cF'\cT c} \\
	&& {\cF'\cT a \cF bc} & {\cF'\cT a (\cF'\cT b\cF'\cT c)} \\
	{\cF a(bc)} &&&& {(\cF'\cT ab)\cF'\cT c} \\
	&& {\cF'\cT a \cF'\cT bc} \\
	{\cF'\cT a(bc)} &&& {\cF'\cT (ab)c}
	\arrow["\alpha", from=1-1, to=2-1]
	\arrow["1\chi", from=2-1, to=3-1]
	\arrow["\chi", from=3-1, to=5-1]
	\arrow["\tau", from=5-1, to=7-1]
	\arrow["\tau1", from=3-1, to=4-3]
	\arrow[""{name=0, anchor=center, inner sep=0}, "1\tau", from=4-3, to=6-3]
	\arrow["{\chi_{\cF'\cT}}"', from=6-3, to=7-1]
	\arrow["1\chi", from=3-3, to=4-3]
	\arrow["{\tau(11)}", from=2-1, to=3-3]
	\arrow["{(\tau1)1}", from=1-1, to=2-3]
	\arrow[""{name=1, anchor=center, inner sep=0}, "\alpha", from=2-3, to=3-3]
	\arrow["{(1\tau)1}", from=2-3, to=2-4]
	\arrow["{1(\tau1)}", from=3-3, to=3-4]
	\arrow["{(11)\tau}", from=2-4, to=3-5]
	\arrow[""{name=2, anchor=center, inner sep=0}, "\alpha", from=2-4, to=3-4]
	\arrow["{1(1\tau)}", from=3-4, to=4-4]
	\arrow[""{name=3, anchor=center, inner sep=0}, "{\chi_{F'T}}", from=4-4, to=6-3]
	\arrow["\alpha", from=3-5, to=4-4]
	\arrow["{\chi_{\cF'\cT}1}", from=3-5, to=5-5]
	\arrow["{\chi_{\cF'\cT}}"{description}, from=5-5, to=7-4]
	\arrow["{\cF'\cT(\alpha)}"', from=7-4, to=7-1]
	\arrow["{\underset{\phi_{\chi, \tau}}{\Rightarrow}}"{description}, draw=none, from=3-1, to=3-3]
	\arrow[draw=none, from=2-1, to=2-3]
	\arrow["{\underset{\alpha_{\tau11}}{\Rightarrow}}"{description}, draw=none, from=2-1, to=2-3]
	\arrow["{\underset{\alpha_{11\tau}}{\Rightarrow}}"{description}, draw=none, from=3-4, to=3-5]
	\arrow["{\underset{1\Pi^{-1}}{\Rightarrow}}"{description}, draw=none, from=4-3, to=4-4]
	\arrow["{\underset{\Pi^{-1}}{\Rightarrow}}"{description}, draw=none, from=5-1, to=0]
	\arrow["{\underset{\alpha_{1\tau1}}{\Rightarrow}}"{description}, draw=none, from=1, to=2]
	\arrow["{\underset{\omega_{\cF'\cT}}{\Rightarrow}}"{description}, draw=none, from=3, to=5-5]
\end{tikzcd}\]

The inverse of the 2-morphisms associated to these pasting diagrams are precisely the two expressions in the monoidal transformation axiom, up to the insertion of an identical region involving the mate of the associator in both diagrams, at the locations marked with $\star$.  The inverse appears here as the monoidal transformation axiom requires a 2-morphism from a composite of two 1-morphisms to a composite of three, but the 1-cell axiom is the other way around.
\end{proof}

\subsection{The Modification Axiom}
Here we concern ourselves with the two axioms that must be proven about the morphisms $\Omega(\eta)$. In this section we have $f \colon b \to c \in \cC$, $\eta, \eta' \in \End(\cF)$ and $\Theta \colon \eta \to \eta'$.  
We will use the following labels for invertible 2-morphisms to save space. 
\definecolor{NewYellow}{HTML}{800020}
\[\begin{tabular}{|c|c|}
	\hline 
Symbol & Meaning \\
	\hline
	$\diamondsuit$ & The 2-morphism \eqref{CuspTransferEquation} moving $f$ around a cap (or cup).  \\
	\hline 
	$\phi$ & Interchangers  \\
	\hline 
	$\clubsuit$ & The 2-morphism \eqref{EvaluationReplacement} $ \ev_{\cF(-)}(1\delta) \Rightarrow \cF(\ev_-)$ \\
	\hline 
	$\spadesuit$ & The Sweedler isomorphism \eqref{2SweedlerNotation}   \\
	\hline 
%	{$\Rightarrow$} & The indicated modification \\
%	\hline
\end{tabular}\]
%  & \color{orange}{$\Rightarrow$} &\color{YellowOrange}{$\Rightarrow$} &\color{darkgreen}{$\Rightarrow$} &\color{blue}{$\Rightarrow$} &\color{purple}{$\Rightarrow$}  \\
%   \hline 
% . & Interchangers & The naturator of $\delta$, or its inverse. &  
No distinction is made between a 2-morphism and its inverse. In all cases, the precise location and direction a 2-morphism was applied will be clear from the source and target, and in most instances there is only one option. In addition, no confusion should arise from symbols which have multiple associated types of 2-morphisms, as only one will be possible at a time. 
In some cases multiple arrows will be used at once, and vertically or horizontally stacked. The symbols: 
\[\Centerstack{ $\phi$ \\ $\Longrightarrow$ \\ + \\ $\Longrightarrow$ \\ $\spadesuit$} \quad \quad \text{and} \quad \quad \phi{\Big\Downarrow}   \color{black}{~~+~~} \Big\Downarrow \spadesuit\] 
both mean that an interchanger was applied first, then the Sweedler isomorphism.
\begin{prop}\label{2AntipodeAxiomVerification}
	The morphism \eqref{OmegaDefinition} satisfies the modification axiom. 
\end{prop}
\renewcommand{\roundNbox}[6]{
	\draw[rounded corners=5pt, very thick, #1] ($#2+(-#3,-#3)+(-#4,0)$) rectangle ($#2+(#3,#3)+(#5,0)$);
	\coordinate (ZZa) at ($#2+(-#4,0)$);
	\coordinate (ZZb) at ($#2+(#5,0)$);
	\node[label={[yshift=-0.5cm]#6}] at ($1/2*(ZZa)+1/2*(ZZb)$) {};
}
\begin{proof} We simply make rigorous the statement that ``each step of $\Omega$ is natural''.
	We write $\cQ(\otimes, J)(\eta) = \cQ_J(\eta)$ and likewise $f\otimes 1 = f1$. Observe the following diagram: 
	
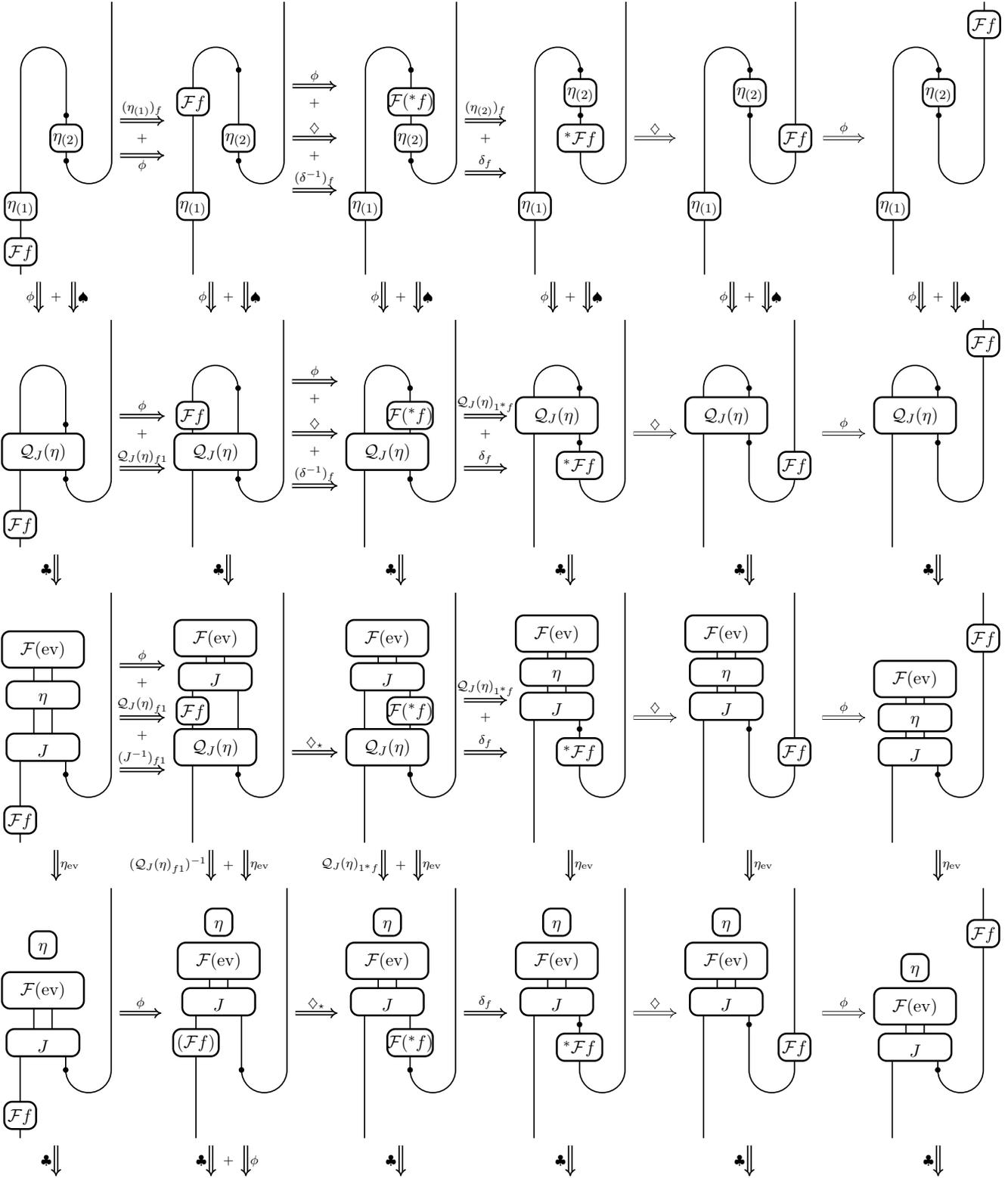
\[\hspace{-1.25cm}\begin{tikzpicture}[baseline= (a).base] \node[scale=.8] (a) at (0,0){
			\begin{tikzcd}
				%%LINE 1
				%%%%%%%%%%%%%%%%%%%%%%%%%%%%%%%%%%%%%%%%%%%%%%%%%%%%%%%%%%%%%%%%%%%%%%%%%%%%%%%%%%%%%%%%%%%%%%%%%%
				\tikzmath{
					\draw[thick] (0,-2) -- (0, 2.5)  arc (180:0:.5cm) -- (1,.5)  arc (-180:0:.5cm) -- (2, 4);
					% additional draw args, center, radius, additional left x-space, additional right x-space, contents
					\roundNbox{fill=white}{(1,1)}{.3}{.05}{.05}{${\eta_{(2)}}$};
					\roundNbox{fill=white}{(0,-.5)}{.3}{.05}{.05}{${\eta_{(1)}}$};
					\roundNbox{fill=white}{(0,-1.5)}{.3}{.05}{.05}{${\cF f}$};
					\filldraw[black] (1,1.5) circle (1.5pt);
					\filldraw[black] (1,.5) circle (1.5pt);
				}
				\ar[r ,thick, Rightarrow, shift left = 1em, "(\eta_{(1)})_f"]
			  	\ar[r, "{+}"{description}, draw=none]
				\ar[r,thick, Rightarrow, shift right = 1em, swap, "\phi"]
				\ar[d,thick, Rightarrow, shift left = 1em, "\spadesuit"]
				\ar[d, "{+}"{description}, draw=none]
				\ar[d,thick, Rightarrow, shift right = 1em, swap, "\phi"]
				& 
					\tikzmath{
					\draw[thick] (0,-2) -- (0, 2.5)  arc (180:0:.5cm) -- (1,.5)  arc (-180:0:.5cm) -- (2, 4);
					% additional draw args, center, radius, additional left x-space, additional right x-space, contents
					\roundNbox{fill=white}{(1,1)}{.3}{.05}{.05}{${\eta_{(2)}}$};
					\roundNbox{fill=white}{(0,-.5)}{.3}{.05}{.05}{${\eta_{(1)}}$};
					\roundNbox{fill=white}{(0,1.8)}{.3}{.05}{.05}{${\cF f}$};
					\filldraw[black] (1,2.5) circle (1.5pt);
					\filldraw[black] (1,.5) circle (1.5pt);
				}
\ar[r, thick, Rightarrow, shift left = 3em, "\phi" ]
\ar[r, thick, Rightarrow,  "\diamondsuit"]
\ar[r,thick, Rightarrow, shift right = 3em, "(\delta^{-1})_f" ]
\ar[r, "{+}"{description}, shift left = 2em, draw=none]
\ar[r, "{+}"{description},  shift right = 1em, draw=none]
				\ar[d,thick, Rightarrow, shift left = 1em, "\spadesuit"]
\ar[d, "{+}"{description}, draw=none]
\ar[d,thick, Rightarrow, shift right = 1em, swap, "\phi"]
				& 
				\tikzmath{
					\draw[thick] (0,-2) -- (0, 2.5)  arc (180:0:.5cm) -- (1,.5)  arc (-180:0:.5cm) -- (2, 4);
					% additional draw args, center, radius, additional left x-space, additional right x-space, contents
					\roundNbox{fill=white}{(1,1)}{.3}{.05}{.05}{${\eta_{(2)}}$};
					\roundNbox{fill=white}{(0,-.5)}{.3}{.05}{.05}{${\eta_{(1)}}$};
					\roundNbox{fill=white}{(1,1.8)}{.3}{.2}{.2}{${\cF (^*f)}$};
					\filldraw[black] (1,2.5) circle (1.5pt);
					\filldraw[black] (1,.5) circle (1.5pt);
				}
\ar[r, thick, Rightarrow, shift left = 1em,	"(\eta_{(2)})_f"]
\ar[r, thick, Rightarrow, shift right = 2em, "\delta_f"]
\ar[r, "{+}"{description}, draw=none]
				\ar[d,thick, Rightarrow, shift left = 1em, "\spadesuit"]
\ar[d, "{+}"{description}, draw=none]
\ar[d,thick, Rightarrow, shift right = 1em, swap, "\phi"]
			& 
			\tikzmath{
				\draw[thick] (0,-2) -- (0, 2.5)  arc (180:0:.5cm) -- (1,.5)  arc (-180:0:.5cm) -- (2, 4);
				% additional draw args, center, radius, additional left x-space, additional right x-space, contents
				\roundNbox{fill=white}{(1,2)}{.3}{.05}{.05}{${\eta_{(2)}}$};
				\roundNbox{fill=white}{(0,-.5)}{.3}{.05}{.05}{${\eta_{(1)}}$};
				\roundNbox{fill=white}{(1,1)}{.3}{.2}{.2}{${^*\cF f}$};
				\filldraw[black] (1,2.5) circle (1.5pt);
				\filldraw[black] (1,1.5) circle (1.5pt);
			}
				\ar[r, "\diamondsuit", Rightarrow]
				\ar[d,thick, Rightarrow, shift left = 1em, "\spadesuit"]
\ar[d, "{+}"{description}, draw=none]
\ar[d,thick, Rightarrow, shift right = 1em, swap, "\phi"]
		& 
		\tikzmath{
			\draw[thick] (0,-2) -- (0, 2.5)  arc (180:0:.5cm) -- (1,.5)  arc (-180:0:.5cm) -- (2, 4);
			% additional draw args, center, radius, additional left x-space, additional right x-space, contents
			\roundNbox{fill=white}{(1,2)}{.3}{.05}{.05}{${\eta_{(2)}}$};
			\roundNbox{fill=white}{(0,-.5)}{.3}{.05}{.05}{${\eta_{(1)}}$};
			\roundNbox{fill=white}{(2,1)}{.3}{.05}{.05}{${\cF f}$};
			\filldraw[black] (1,2.5) circle (1.5pt);
			\filldraw[black] (1,1.5) circle (1.5pt);
		}
\ar[r, "\phi", Rightarrow]
\ar[d,thick, Rightarrow, shift left = 1em, "\spadesuit"]
\ar[d, "{+}"{description}, draw=none]
\ar[d,thick, Rightarrow, shift right = 1em, swap, "\phi"]
	& 
	\tikzmath{
		\draw[thick] (0,-2) -- (0, 2.5)  arc (180:0:.5cm) -- (1,.5)  arc (-180:0:.5cm) -- (2, 4);
		% additional draw args, center, radius, additional left x-space, additional right x-space, contents
		\roundNbox{fill=white}{(1,2)}{.3}{.05}{.05}{${\eta_{(2)}}$};
		\roundNbox{fill=white}{(0,-.5)}{.3}{.05}{.05}{${\eta_{(1)}}$};
		\roundNbox{fill=white}{(2,3.5)}{.3}{.05}{.05}{${\cF f}$};
		\filldraw[black] (1,2.5) circle (1.5pt);
		\filldraw[black] (1,1.5) circle (1.5pt);
	} 
				\ar[d,thick, Rightarrow, shift left = 1em, "\spadesuit"]
\ar[d, "{+}"{description}, draw=none]
\ar[d,thick, Rightarrow, shift right = 1em, swap, "\phi"]
\\
	%%LINE 2
	%%%%%%%%%%%%%%%%%%%%%%%%%%%%%%%%%%%%%%%%%%%%%%%%%%%%%%%%%%%%%%%%%%%%%%%%%%%%%%%%%%%%%%%%%%%%%%%%% 
\tikzmath{
	\draw[thick] (0,-1) -- (0, 2.5)  arc (180:0:.5cm) -- (1,.5)  arc (-180:0:.5cm) -- (2, 4);
	% additional draw args, center, radius, additional left x-space, additional right x-space, contents
	\roundNbox{fill=white}{(.5,1.1)}{.4}{.5}{.5}{${\cQ_J(\eta)}$};
	\roundNbox{fill=white}{(0,-.5)}{.3}{.05}{.05}{${\cF f}$};
	\filldraw[black] (1,1.7) circle (1.5pt);
	\filldraw[black] (1,.5) circle (1.5pt);
}
\ar[r, "\phi" ,thick, Rightarrow, shift left = 1em  ]
\ar[r, thick, Rightarrow,  shift right =2em, "{\cQ_J(\eta)_{f  1}}"]
\ar[r, "{+}"{description}, draw=none]
\ar[d, "\clubsuit", swap ,thick, Rightarrow]
& 
\tikzmath{
	\draw[thick] (0,-1) -- (0, 2.5)  arc (180:0:.5cm) -- (1,.5)  arc (-180:0:.5cm) -- (2, 4);
	% additional draw args, center, radius, additional left x-space, additional right x-space, contents
	\roundNbox{fill=white}{(.5,1.1)}{.4}{.5}{.5}{${{\cQ_J(\eta)}}$};
	\roundNbox{fill=white}{(0,1.9)}{.3}{.05}{.05}{${\cF f}$};
	\filldraw[black] (1,2.5) circle (1.5pt);
	\filldraw[black] (1,.5) circle (1.5pt);
}
\ar[r, "\phi" ,thick, Rightarrow, shift left = 3em,  ]
\ar[r, thick, Rightarrow, "\diamondsuit"]
\ar[r, thick, Rightarrow, shift right = 3em, "(\delta^{-1})_f" ]
\ar[r, "{+}"{description}, shift left = 2em, draw=none]
\ar[r, "{+}"{description},  shift right = 1em, draw=none]
\ar[d, "\clubsuit", swap ,thick, Rightarrow]
& 
\tikzmath{
	\draw[thick] (0,-1) -- (0, 2.5)  arc (180:0:.5cm) -- (1,.5)  arc (-180:0:.5cm) -- (2, 4);
	% additional draw args, center, radius, additional left x-space, additional right x-space, contents
	\roundNbox{fill=white}{(.5,1.1)}{.4}{.5}{.5}{${{\cQ_J(\eta)}}$};
				\roundNbox{fill=white}{(1,1.9)}{.3}{.2}{.2}{${\cF (^*f)}$};
\filldraw[black] (1,2.5) circle (1.5pt);
\filldraw[black] (1,.5) circle (1.5pt);
}
\ar[r, thick, Rightarrow, shift left = 1em,	"{\cQ_J(\eta)}_{1  ^*f}"]
\ar[r, thick, Rightarrow, shift right = 2em, "\delta_f"]
\ar[r, "{+}"{description}, draw=none]
\ar[d, "\clubsuit", swap ,thick, Rightarrow]
& 
\tikzmath{
	\draw[thick] (0,-1) -- (0, 2.5)  arc (180:0:.5cm) -- (1,.5)  arc (-180:0:.5cm) -- (2, 4);
	% additional draw args, center, radius, additional left x-space, additional right x-space, contents
	\roundNbox{fill=white}{(.5,1.9)}{.4}{.5}{.5}{${{\cQ_J(\eta)}}$};
	\roundNbox{fill=white}{(1,.8)}{.3}{.2}{.2}{${^*\cF f}$};
	\filldraw[black] (1,2.5) circle (1.5pt);
	\filldraw[black] (1,1.3) circle (1.5pt);
}
\ar[r, "\diamondsuit", Rightarrow]
\ar[d, "\clubsuit", swap ,thick, Rightarrow]
& 
\tikzmath{
	\draw[thick] (0,-1) -- (0, 2.5)  arc (180:0:.5cm) -- (1,.5)  arc (-180:0:.5cm) -- (2, 4);
	% additional draw args, center, radius, additional left x-space, additional right x-space, contents
	\roundNbox{fill=white}{(.5,1.9)}{.4}{.5}{.5}{${{\cQ_J(\eta)}}$};
\filldraw[black] (1,2.5) circle (1.5pt);
\filldraw[black] (1,1.3) circle (1.5pt);
	\roundNbox{fill=white}{(2,.8)}{.3}{.05}{.05}{${\cF f}$};
}
\ar[r, "\phi", Rightarrow]
\ar[d, "\clubsuit", swap ,thick, Rightarrow]
& 
\tikzmath{
	\draw[thick] (0,-1) -- (0, 2.5)  arc (180:0:.5cm) -- (1,.5)  arc (-180:0:.5cm) -- (2, 4);
	% additional draw args, center, radius, additional left x-space, additional right x-space, contents
	\roundNbox{fill=white}{(.5,1.9)}{.4}{.5}{.5}{${{\cQ_J(\eta)}}$};
\filldraw[black] (1,2.5) circle (1.5pt);
\filldraw[black] (1,1.3) circle (1.5pt);
	\roundNbox{fill=white}{(2,3.5)}{.3}{.05}{.05}{${\cF f}$};
} 
\ar[d, "\clubsuit", swap ,thick, Rightarrow]
\\
%%LINE 3
%%%%%%%%%%%%%%%%%%%%%%%%%%%%%%%%%%%%%%%%%%%%%%%%%%%%%%%%%%%%%%%%%%%%%%%%%%%%%%%%%%%%%%%%%%%%%%%%% 
\tikzmath{
	\draw[thick] (0,-1) -- (0, 1);
	\draw[thick] (1, 1)-- (1,.5)  arc (-180:0:.5cm) -- (2, 4.5);
	\draw[thick] (.3, 1) -- (.3,3);
	\draw[thick] (.7, 1) -- (.7,3);
	\roundNbox{fill=white}{(.5,2.25)}{.3}{.5}{.5}{${\eta}$};
	\roundNbox{fill=white}{(.5,3.25)}{.4}{.5}{.5}{${\cF(\ev)}$};
	\roundNbox{fill=white}{(.5,1.1)}{.3}{.5}{.5}{${J}$};
	% additional draw args, center, radius, additional left x-space, additional right x-space, contents
	\roundNbox{fill=white}{(0,-.5)}{.3}{.05}{.05}{${\cF f}$};
	\filldraw[black] (1,.5) circle (1.5pt);
}
\ar[r, "\phi" ,thick, Rightarrow, shift left = 3em,  ]
\ar[r, thick, Rightarrow, "{\cQ_J(\eta)}_{f  1}"]
\ar[r, thick, Rightarrow, shift right = 3em, "(J^{-1})_{f1}" ]
\ar[r, "{+}"{description}, shift left = 2em, draw=none]
\ar[r, "{+}"{description},  shift right = 1em, draw=none]
\ar[d, thick, Rightarrow, "\eta_{\ev}"]
& 
\tikzmath{
\draw[thick] (0,-1) -- (0, 2.5);
\draw[thick] (1, 2.5)-- (1,.5)  arc (-180:0:.5cm) -- (2, 4.5);
\draw[thick] (.3, 2.7) -- (.3,3.5);
\draw[thick] (.7, 2.7) -- (.7,3.5);
\roundNbox{fill=white}{(.5,3.5)}{.4}{.5}{.5}{${\cF(\ev)}$};
\roundNbox{fill=white}{(.5,2.65)}{.3}{.5}{.5}{${J}$};
	% additional draw args, center, radius, additional left x-space, additional right x-space, contents
	\roundNbox{fill=white}{(.5,1.1)}{.4}{.5}{.5}{${{\cQ_J(\eta)}}$};
	\roundNbox{fill=white}{(0,1.9)}{.3}{.05}{.05}{${\cF f}$};
	\filldraw[black] (1,.5) circle (1.5pt);
}
%\ar[r, "\phi" ,thick, Rightarrow, shift left = 1em  ]
\ar[r,thick, Rightarrow,  shift right =2em, "\diamondsuit_\star"]
%\ar[r, "{+}"{description}, draw=none]
\ar[d, thick, Rightarrow, swap, shift right = 1em,  "({\cQ_J(\eta)}_{f  1})^{-1}"]
\ar[d, "{+}"{description}, draw=none]
\ar[d, thick, Rightarrow, shift left = 1em,  "\eta_{\ev}"]
& 
\tikzmath{
\draw[thick] (0,-1) -- (0, 2.5);
\draw[thick] (1, 2.5)-- (1,.5)  arc (-180:0:.5cm) -- (2, 4.5);
\draw[thick] (.3, 2.7) -- (.3,3.5);
\draw[thick] (.7, 2.7) -- (.7,3.5);
\roundNbox{fill=white}{(.5,3.5)}{.4}{.5}{.5}{${\cF(\ev)}$};
\roundNbox{fill=white}{(.5,2.65)}{.3}{.5}{.5}{${J}$};
	% additional draw args, center, radius, additional left x-space, additional right x-space, contents
	\roundNbox{fill=white}{(.5,1.1)}{.4}{.5}{.5}{${{\cQ_J(\eta)}}$};
	\roundNbox{fill=white}{(1,1.9)}{.3}{.2}{.2}{${\cF (^*f)}$};
	\filldraw[black] (1,.5) circle (1.5pt);
}
\ar[r, thick, Rightarrow, shift left = 1em,	"{\cQ_J(\eta)}_{1  ^*f}"]
\ar[r, thick, Rightarrow, shift right = 2em, "\delta_f"]
\ar[r, "{+}"{description}, draw=none]
\ar[d, thick, Rightarrow, swap, shift right = 1em,  "{\cQ_J(\eta)}_{1  ^*f}"]
\ar[d, "{+}"{description}, draw=none]
\ar[d, thick, Rightarrow, shift left = 1em,  "\eta_{\ev}"]
& 
\tikzmath{
	\draw[thick] (0,-1) -- (0, 2);
	\draw[thick] (1, 2)-- (1,.5)  arc (-180:0:.5cm) -- (2, 4.5);
	\draw[thick] (.3, 1.75) -- (.3,3.5);
	\draw[thick] (.7, 1.75) -- (.7,3.5);
	\roundNbox{fill=white}{(.5,2.75)}{.3}{.5}{.5}{${\eta}$};
	\roundNbox{fill=white}{(.5,3.6)}{.4}{.5}{.5}{${\cF(\ev)}$};
	\roundNbox{fill=white}{(.5,2)}{.3}{.5}{.5}{${J}$};
	\filldraw[black] (1,1.5) circle (1.5pt);
	\roundNbox{fill=white}{(1,1)}{.3}{.2}{.2}{${^*\cF f}$};
} 
\ar[r, "\diamondsuit", Rightarrow]
\ar[d, thick, Rightarrow, "\eta_{\ev}"]
& 
\tikzmath{
	\draw[thick] (0,-1) -- (0, 2);
	\draw[thick] (1, 2)-- (1,.5)  arc (-180:0:.5cm) -- (2, 4.5);
	\draw[thick] (.3, 1.75) -- (.3,3.5);
	\draw[thick] (.7, 1.75) -- (.7,3.5);
	\roundNbox{fill=white}{(.5,2.75)}{.3}{.5}{.5}{${\eta}$};
	\roundNbox{fill=white}{(.5,3.6)}{.4}{.5}{.5}{${\cF(\ev)}$};
	\roundNbox{fill=white}{(.5,2)}{.3}{.5}{.5}{${J}$};
	\filldraw[black] (1,1.5) circle (1.5pt);
	\roundNbox{fill=white}{(2,1)}{.3}{.05}{.05}{${\cF f}$};
} 
\ar[r, "\phi", Rightarrow]
\ar[d, thick, Rightarrow, "\eta_{\ev}"]
& 
\tikzmath{
	\draw[thick] (0,-1) -- (0, 1);
	\draw[thick] (1, 1)-- (1,.5)  arc (-180:0:.5cm) -- (2, 4.5);
	\draw[thick] (.3, .75) -- (.3,2.5);
	\draw[thick] (.7, .75) -- (.7,2.5);
	\roundNbox{fill=white}{(.5,1.75)}{.3}{.5}{.5}{${\eta}$};
	\roundNbox{fill=white}{(.5,2.6)}{.4}{.5}{.5}{${\cF(\ev)}$};
	\roundNbox{fill=white}{(.5,1)}{.3}{.5}{.5}{${J}$};
	\filldraw[black] (1,.5) circle (1.5pt);
	\roundNbox{fill=white}{(2,3.5)}{.3}{.05}{.05}{${\cF f}$};
} 
\ar[d, thick, Rightarrow, "\eta_{\ev}"]
\\
%%LINE 4
%%%%%%%%%%%%%%%%%%%%%%%%%%%%%%%%%%%%%%%%%%%%%%%%%%%%%%%%%%%%%%%%%%%%%%%%%%%%%%%%%%%%%%%%%%%%%%%%% 
\tikzmath{
	\draw[thick] (0,-1) -- (0, 1);
	\draw[thick] (1, 1)-- (1,.5)  arc (-180:0:.5cm) -- (2, 4.5);
	\draw[thick] (.3, 1) -- (.3,2);
	\draw[thick] (.7, 1) -- (.7,2);
	\roundNbox{fill=white}{(.5,3.25)}{.3}{0}{0}{${\eta}$};
	\roundNbox{fill=white}{(.5,2.25)}{.4}{.5}{.5}{${\cF(\ev)}$};
	\roundNbox{fill=white}{(.5,1.1)}{.3}{.5}{.5}{${J}$};
	% additional draw args, center, radius, additional left x-space, additional right x-space, contents
	\roundNbox{fill=white}{(0,-.5)}{.3}{.05}{.05}{${\cF f}$};
	\filldraw[black] (1,.5) circle (1.5pt);
}
\ar[r, "\phi" ,thick, Rightarrow]
\ar[d, "\clubsuit", swap ,thick, Rightarrow]
& 
\tikzmath{
	\draw[thick] (0,-1) -- (0, 2);
	\draw[thick] (1, 2)-- (1,.5)  arc (-180:0:.5cm) -- (2, 4.5);
	\draw[thick] (.3, 1.75) -- (.3,2.8);
	\draw[thick] (.7, 1.75) -- (.7,2.8);
	\roundNbox{fill=white}{(.5,3.75)}{.3}{0}{0}{${\eta}$};
	\roundNbox{fill=white}{(.5,2.9)}{.4}{.5}{.5}{${\cF(\ev)}$};
	\roundNbox{fill=white}{(.5,2)}{.3}{.5}{.5}{${J}$};
	\filldraw[black] (1,.5) circle (1.5pt);
	\roundNbox{fill=white}{(0,1.1)}{.3}{.2}{.2}{${(\cF f)}$};
} 
\ar[r, thick, Rightarrow,  "\diamondsuit_\star"]
\ar[d, "\phi" ,thick, Rightarrow, shift left = 1em]
\ar[d, "{+}"{description}, draw=none]
\ar[d, "\clubsuit", swap,thick, Rightarrow, shift right = 1em]
& 
\tikzmath{
	\draw[thick] (0,-1) -- (0, 2);
	\draw[thick] (1, 2)-- (1,.5)  arc (-180:0:.5cm) -- (2, 4.5);
	\draw[thick] (.3, 1.75) -- (.3,2.8);
	\draw[thick] (.7, 1.75) -- (.7,2.8);
	\roundNbox{fill=white}{(.5,3.75)}{.3}{0}{0}{${\eta}$};
	\roundNbox{fill=white}{(.5,2.9)}{.4}{.5}{.5}{${\cF(\ev)}$};
	\roundNbox{fill=white}{(.5,2)}{.3}{.5}{.5}{${J}$};
	\filldraw[black] (1,.5) circle (1.5pt);
	\roundNbox{fill=white}{(1,1.1)}{.3}{.2}{.2}{${\cF (^*f)}$};
} 
\ar[r, thick, Rightarrow, "\delta_f"]
\ar[d, "\clubsuit", swap ,thick, Rightarrow]
& 
\tikzmath{
	\draw[thick] (0,-1) -- (0, 2);
	\draw[thick] (1, 2)-- (1,.5)  arc (-180:0:.5cm) -- (2, 4.5);
	\draw[thick] (.3, 1.75) -- (.3,2.8);
\draw[thick] (.7, 1.75) -- (.7,2.8);
\roundNbox{fill=white}{(.5,3.75)}{.3}{0}{0}{${\eta}$};
\roundNbox{fill=white}{(.5,2.9)}{.4}{.5}{.5}{${\cF(\ev)}$};
	\roundNbox{fill=white}{(.5,2)}{.3}{.5}{.5}{${J}$};
	\filldraw[black] (1,1.5) circle (1.5pt);
	\roundNbox{fill=white}{(1,1)}{.3}{.2}{.2}{${^*\cF f}$};
} 
\ar[r, "\diamondsuit", Rightarrow]
\ar[d, "\clubsuit", swap ,thick, Rightarrow]
& 
\tikzmath{
	\draw[thick] (0,-1) -- (0, 2);
	\draw[thick] (1, 2)-- (1,.5)  arc (-180:0:.5cm) -- (2, 4.5);
	\draw[thick] (.3, 1.75) -- (.3,2.8);
	\draw[thick] (.7, 1.75) -- (.7,2.8);
	\roundNbox{fill=white}{(.5,3.75)}{.3}{0}{0}{${\eta}$};
	\roundNbox{fill=white}{(.5,2.9)}{.4}{.5}{.5}{${\cF(\ev)}$};
	\roundNbox{fill=white}{(.5,2)}{.3}{.5}{.5}{${J}$};
	\filldraw[black] (1,1.5) circle (1.5pt);
	\roundNbox{fill=white}{(2,1)}{.3}{.05}{.05}{${\cF f}$};
} 
\ar[r, "\phi", Rightarrow]
\ar[d, "\clubsuit", swap ,thick, Rightarrow]
& 
\tikzmath{
	\draw[thick] (0,-1) -- (0, 1);
\draw[thick] (1, 1)-- (1,.5)  arc (-180:0:.5cm) -- (2, 4.5);
\draw[thick] (.3, .75) -- (.3,1.75);
\draw[thick] (.7, .75) -- (.7,1.75);
\roundNbox{fill=white}{(.5,2.75)}{.3}{0}{0}{${\eta}$};
\roundNbox{fill=white}{(.5,1.9)}{.4}{.5}{.5}{${\cF(\ev)}$};
	\roundNbox{fill=white}{(.5,1)}{.3}{.5}{.5}{${J}$};
	\filldraw[black] (1,.5) circle (1.5pt);
	\roundNbox{fill=white}{(2,3.5)}{.3}{.05}{.05}{${\cF f}$};
} 
\ar[d, "\clubsuit", swap ,thick, Rightarrow]
\\
~&~&~&~&~&~
\end{tikzcd}}; 
\end{tikzpicture}\]
	
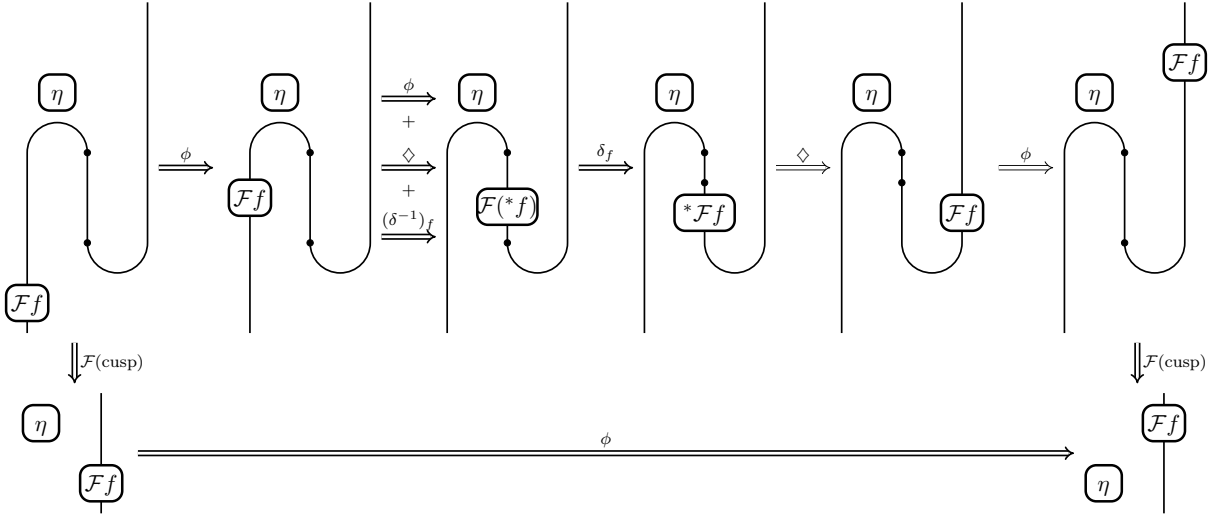
\[\hspace{-1.25cm}\begin{tikzpicture}[baseline= (a).base] \node[scale=.8] (a) at (0,0){
		\begin{tikzcd}
			%%LINE 5
			%%%%%%%%%%%%%%%%%%%%%%%%%%%%%%%%%%%%%%%%%%%%%%%%%%%%%%%%%%%%%%%%%%%%%%%%%%%%%%%%%%%%%%%%%%%%%%%%% 
			\tikzmath{
				\draw[thick] (0,-1) -- (0, 2) arc (+180:0:.5cm) -- (1,.5)  arc (-180:0:.5cm) -- (2, 4.5);
				\roundNbox{fill=white}{(.5,3)}{.3}{0}{0}{${\eta}$};
				% additional draw args, center, radius, additional left x-space, additional right x-space, contents
				\roundNbox{fill=white}{(0,-.5)}{.3}{.05}{.05}{${\cF f}$};
				\filldraw[black] (1,2) circle (1.5pt);
				\filldraw[black] (1,.5) circle (1.5pt);
			}
			\ar[r, "\phi" ,thick, Rightarrow]
			\ar[d,thick, Rightarrow, "\cF(\text{cusp})"]
			& 
					\tikzmath{
			\draw[thick] (0,-1) -- (0, 2) arc (+180:0:.5cm) -- (1,.5)  arc (-180:0:.5cm) -- (2, 4.5);
			\roundNbox{fill=white}{(.5,3)}{.3}{0}{0}{${\eta}$};
			% additional draw args, center, radius, additional left x-space, additional right x-space, contents
			\roundNbox{fill=white}{(0,1.25)}{.3}{.05}{.05}{${\cF f}$};
			\filldraw[black] (1,2) circle (1.5pt);
			\filldraw[black] (1,.5) circle (1.5pt);
		}
			\ar[r, "\phi" ,thick, Rightarrow, shift left = 3em,  ]
			\ar[r, ,thick, Rightarrow, "\diamondsuit"]
			\ar[r, thick, Rightarrow, shift right = 3em, "(\delta^{-1})_f" ]
			\ar[r, "{+}"{description}, shift left = 2em, draw=none]
			\ar[r, "{+}"{description},  shift right = 1em, draw=none]
			& 
			\tikzmath{
				\draw[thick] (0,-1) -- (0, 2) arc (+180:0:.5cm) -- (1,.5)  arc (-180:0:.5cm) -- (2, 4.5);
				\roundNbox{fill=white}{(.5,3)}{.3}{0}{0}{${\eta}$};
				% additional draw args, center, radius, additional left x-space, additional right x-space, contents
				\filldraw[black] (1,2) circle (1.5pt);
				\filldraw[black] (1,.5) circle (1.5pt);
				\roundNbox{fill=white}{(1,1.1)}{.3}{.2}{.2}{${\cF(^*f)}$};
			} 
			\ar[r, thick, Rightarrow, "\delta_f"]
			& 
			\tikzmath{
					\draw[thick] (0,-1) -- (0, 2) arc (+180:0:.5cm) -- (1,.5)  arc (-180:0:.5cm) -- (2, 4.5);
	\roundNbox{fill=white}{(.5,3)}{.3}{0}{0}{${\eta}$};
	% additional draw args, center, radius, additional left x-space, additional right x-space, contents

	\filldraw[black] (1,2) circle (1.5pt);
				\filldraw[black] (1,1.5) circle (1.5pt);
				\roundNbox{fill=white}{(1,1)}{.3}{.2}{.2}{${^*\cF f}$};
			} 
			\ar[r, "\diamondsuit", Rightarrow]
			& 
			\tikzmath{
							\draw[thick] (0,-1) -- (0, 2) arc (+180:0:.5cm) -- (1,.5)  arc (-180:0:.5cm) -- (2, 4.5);
			\roundNbox{fill=white}{(.5,3)}{.3}{0}{0}{${\eta}$};
			% additional draw args, center, radius, additional left x-space, additional right x-space, contents
			\filldraw[black] (1,2) circle (1.5pt);
				\filldraw[black] (1,1.5) circle (1.5pt);
				\roundNbox{fill=white}{(2,1)}{.3}{.05}{.05}{${\cF f}$};
			} 
			\ar[r, "\phi", Rightarrow]
			& 
			\tikzmath{
				\draw[thick] (0,-1) -- (0, 2) arc (+180:0:.5cm) -- (1,.5)  arc (-180:0:.5cm) -- (2, 4.5);
\roundNbox{fill=white}{(.5,3)}{.3}{0}{0}{${\eta}$};
% additional draw args, center, radius, additional left x-space, additional right x-space, contents

\filldraw[black] (1,2) circle (1.5pt);
				\filldraw[black] (1,.5) circle (1.5pt);
				\roundNbox{fill=white}{(2,3.5)}{.3}{.05}{.05}{${\cF f}$};
			} 
			\ar[d,thick, Rightarrow, "\cF(\text{cusp})"]
			\\
			%%LINE 6
			%%%%%%%%%%%%%%%%%%%%%%%%%%%%%%%%%%%%%%%%%%%%%%%%%%%%%%%%%%%%%%%%%%%%%%%%%%%%%%%%%%%%%%%%%%%%%%%%% 
		    \tikzmath{
	\draw[thick] (1,0) -- (1, 2);
	\roundNbox{fill=white}{(1,.5)}{.3}{.05}{.05}{${\cF f}$};
	\roundNbox{fill=white}{(0,1.5)}{.3}{0}{0}{${\eta}$};
}
		\ar[rrrrr, "\phi", thick, Rightarrow]&~&~&~&~&
		    \tikzmath{
		    \draw[thick] (1,0) -- (1, 2);
		    \roundNbox{fill=white}{(1,1.5)}{.3}{.05}{.05}{${\cF f}$};
		    \roundNbox{fill=white}{(0,.5)}{.3}{0}{0}{${\eta}$};
	    	}
		\end{tikzcd}}; 
\end{tikzpicture}\]

The morphisms $\diamondsuit_\star$ are the images of the 2-morphism moving $f$ around a cup in $\cC$. 

The top-right path of this diagram is the naturator for $\mu \circ (1 \boxtimes \cS) \circ \Delta(\eta)$, followed by $\Omega_X$, and the other outside path is $\Omega_Y$ followed by the interchanger, i.e the naturator for $\iota \circ \epsilon(\eta)$. We now turn our attention to why this diagram commutes. 

All regions in the between the first and second row commute by naturality of the interchanger and the tensorator for $\cQ$; some even by locality. The passage from the second line to the third may be initially disorienting; the definition of ${\cQ_J(\eta)}$ was used to simplify some of the resulting diagrams. In any case, the first, third, fourth and fifth regions between these rows commute by locality. The second follows from the definitions of $\clubsuit$ and $\diamondsuit_\star$. The first square between the third and 4th row commutes trivially, the second by naturality of $J$ and $\eta$, and the remainder by locality.  The regions between fourth and fifth row commute for the same reasons as those between the second and third. The bottom region commutes upon canceling the morphisms $\delta_f$, expanding the definition of $\diamondsuit$, and using naturality of the interchanger. 
\end{proof}
\begin{prop} \label{2AntipodeAxiomNaturality}
	The modifications $\Omega(\eta)$ are natural in $\eta$, i.e form a natural transformation $\mu \circ (1 \boxtimes \cS) \circ \delta \Rightarrow \iota \circ \epsilon(\eta)$. 
\end{prop}
\begin{proof}
We have the diagram: 
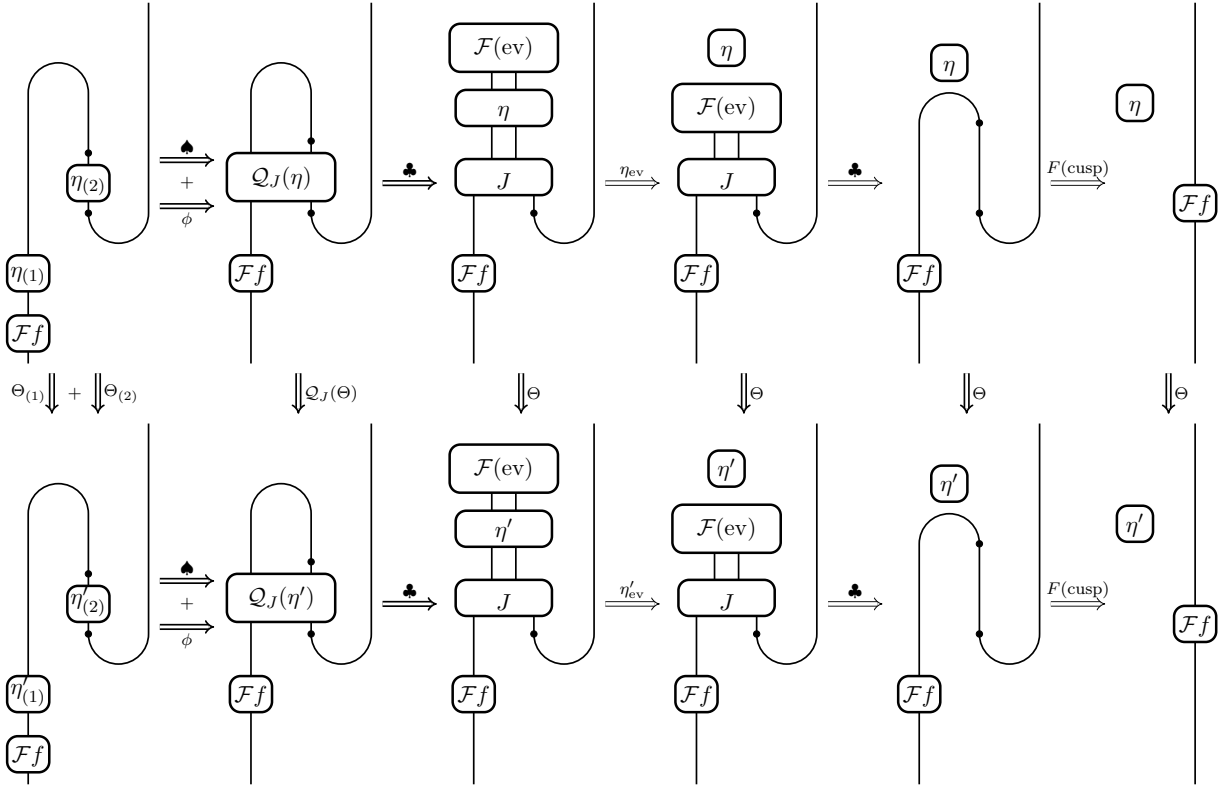
\[\hspace{-1.25cm}\begin{tikzpicture}[baseline= (a).base] \node[scale=.8] (a) at (0,0){
		\begin{tikzcd}
				\tikzmath{
				\draw[thick] (0,-2) -- (0, 2.5)  arc (180:0:.5cm) -- (1,.5)  arc (-180:0:.5cm) -- (2, 4);
				% additional draw args, center, radius, additional left x-space, additional right x-space, contents
				\roundNbox{fill=white}{(1,1)}{.3}{.05}{.05}{${\eta_{(2)}}$};
				\roundNbox{fill=white}{(0,-.5)}{.3}{.05}{.05}{${\eta_{(1)}}$};
				\roundNbox{fill=white}{(0,-1.5)}{.3}{.05}{.05}{${\cF f}$};
				\filldraw[black] (1,1.5) circle (1.5pt);
				\filldraw[black] (1,.5) circle (1.5pt);
			}
			\ar[r, "\spadesuit" ,thick, Rightarrow, shift left = 1em]
			\ar[r, "{+}"{description}, draw=none]
			\ar[r, swap, "\phi",thick, Rightarrow, shift right = 1em]
			\ar[d,thick, Rightarrow, shift left = 1em,  "\Theta_{(2)}"]
			\ar[d, "{+}"{description}, draw=none]
			\ar[d,thick, Rightarrow, shift right = 1em, swap, "\Theta_{(1)}"]
			&
			\tikzmath{
				\draw[thick] (0,-2) -- (0, 2.5)  arc (180:0:.5cm) -- (1,.5)  arc (-180:0:.5cm) -- (2, 4);
				% additional draw args, center, radius, additional left x-space, additional right x-space, contents
				\roundNbox{fill=white}{(.5,1.1)}{.4}{.5}{.5}{${{\cQ_J(\eta)}}$};
				\roundNbox{fill=white}{(0,-.5)}{.3}{.05}{.05}{${\cF f}$};
				\filldraw[black] (1,1.7) circle (1.5pt);
				\filldraw[black] (1,.5) circle (1.5pt);
			}
			\ar[r, "\clubsuit" ,thick, Rightarrow]
			\ar[d ,thick, Rightarrow, "\cQ_J(\Theta)"]
			& 
			\tikzmath{
				\draw[thick] (0,-2) -- (0, 1);
				\draw[thick] (1, 1)-- (1,.5)  arc (-180:0:.5cm) -- (2, 4);
				\draw[thick] (.3, 1) -- (.3,3);
				\draw[thick] (.7, 1) -- (.7,3);
				\roundNbox{fill=white}{(.5,2.25)}{.3}{.5}{.5}{${\eta}$};
				\roundNbox{fill=white}{(.5,3.25)}{.4}{.5}{.5}{${\cF(\ev)}$};
				\roundNbox{fill=white}{(.5,1.1)}{.3}{.5}{.5}{${J}$};
				% additional draw args, center, radius, additional left x-space, additional right x-space, contents
				\roundNbox{fill=white}{(0,-.5)}{.3}{.05}{.05}{${\cF f}$};
				\filldraw[black] (1,.5) circle (1.5pt);
			}
		\ar[r, Rightarrow, "\eta_{\ev}"]
					\ar[d ,thick, Rightarrow, "\Theta"]
		&
		\tikzmath{
			\draw[thick] (0,-2) -- (0, 1);
			\draw[thick] (1, 1)-- (1,.5)  arc (-180:0:.5cm) -- (2, 4);
			\draw[thick] (.3, 1) -- (.3,2);
			\draw[thick] (.7, 1) -- (.7,2);
			\roundNbox{fill=white}{(.5,3.25)}{.3}{0}{0}{${\eta}$};
			\roundNbox{fill=white}{(.5,2.25)}{.4}{.5}{.5}{${\cF(\ev)}$};
			\roundNbox{fill=white}{(.5,1.1)}{.3}{.5}{.5}{${J}$};
			% additional draw args, center, radius, additional left x-space, additional right x-space, contents
			\roundNbox{fill=white}{(0,-.5)}{.3}{.05}{.05}{${\cF f}$};
			\filldraw[black] (1,.5) circle (1.5pt);
		} 
		\ar[r, "\clubsuit", Rightarrow]
							\ar[d ,thick, Rightarrow, "\Theta"]
		&
		\tikzmath{
			\draw[thick] (0,-2) -- (0, 2) arc (+180:0:.5cm) -- (1,.5)  arc (-180:0:.5cm) -- (2, 4);
			\roundNbox{fill=white}{(.5,3)}{.3}{0}{0}{${\eta}$};
			% additional draw args, center, radius, additional left x-space, additional right x-space, contents
			\roundNbox{fill=white}{(0,-.5)}{.3}{.05}{.05}{${\cF f}$};
			\filldraw[black] (1,2) circle (1.5pt);
			\filldraw[black] (1,.5) circle (1.5pt);
		} 
		\ar[r, Rightarrow, "F(\text{cusp})"]
							\ar[d ,thick, Rightarrow, "\Theta"]
		& 
		\tikzmath{
			\draw[thick] (1, -2) -- (1, 4);
			\roundNbox{fill=white}{(0,2.333)}{.3}{0}{0}{${\eta}$};
			\roundNbox{fill=white}{(1,.667)}{.3}{.05}{.05}{${\cF f}$};
		}
	   \ar[d ,thick, Rightarrow, "\Theta"]
		\\ 
		%%%%%%%%%%%%%%%%
		%%%%LINE 2
	   \tikzmath{
	   	\draw[thick] (0,-2) -- (0, 2.5)  arc (180:0:.5cm) -- (1,.5)  arc (-180:0:.5cm) -- (2, 4);
	   	% additional draw args, center, radius, additional left x-space, additional right x-space, contents
	   	\roundNbox{fill=white}{(1,1)}{.3}{.05}{.05}{${\eta'_{(2)}}$};
	   	\roundNbox{fill=white}{(0,-.5)}{.3}{.05}{.05}{${\eta'_{(1)}}$};
	   	\roundNbox{fill=white}{(0,-1.5)}{.3}{.05}{.05}{${\cF f}$};
	   	\filldraw[black] (1,1.5) circle (1.5pt);
	   	\filldraw[black] (1,.5) circle (1.5pt);
	   }
	   \ar[r, "\spadesuit" ,thick, Rightarrow, shift left = 1em]
	   \ar[r, "{+}"{description}, draw=none]
	   \ar[r, swap, "\phi",thick, Rightarrow, shift right = 1em]
	   &
	   \tikzmath{
	   	\draw[thick] (0,-2) -- (0, 2.5)  arc (180:0:.5cm) -- (1,.5)  arc (-180:0:.5cm) -- (2, 4);
	   	% additional draw args, center, radius, additional left x-space, additional right x-space, contents
	   	\roundNbox{fill=white}{(.5,1.1)}{.4}{.5}{.5}{${\cQ_J(\eta')}$};
	   	\roundNbox{fill=white}{(0,-.5)}{.3}{.05}{.05}{${\cF f}$};
	   	\filldraw[black] (1,1.7) circle (1.5pt);
	   	\filldraw[black] (1,.5) circle (1.5pt);
	   }
	   \ar[r, "\clubsuit" ,thick, Rightarrow]
	   & 
	   \tikzmath{
	   	\draw[thick] (0,-2) -- (0, 1);
	   	\draw[thick] (1, 1)-- (1,.5)  arc (-180:0:.5cm) -- (2, 4);
	   	\draw[thick] (.3, 1) -- (.3,3);
	   	\draw[thick] (.7, 1) -- (.7,3);
	   	\roundNbox{fill=white}{(.5,2.25)}{.3}{.5}{.5}{${\eta'}$};
	   	\roundNbox{fill=white}{(.5,3.25)}{.4}{.5}{.5}{${\cF(\ev)}$};
	   	\roundNbox{fill=white}{(.5,1.1)}{.3}{.5}{.5}{${J}$};
	   	% additional draw args, center, radius, additional left x-space, additional right x-space, contents
	   	\roundNbox{fill=white}{(0,-.5)}{.3}{.05}{.05}{${\cF f}$};
	   	\filldraw[black] (1,.5) circle (1.5pt);
	   }
	   \ar[r, Rightarrow, "\eta'_{\ev}"]
	   &
	   \tikzmath{
	   	\draw[thick] (0,-2) -- (0, 1);
	   	\draw[thick] (1, 1)-- (1,.5)  arc (-180:0:.5cm) -- (2, 4);
	   	\draw[thick] (.3, 1) -- (.3,2);
	   	\draw[thick] (.7, 1) -- (.7,2);
	   	\roundNbox{fill=white}{(.5,3.25)}{.3}{0}{0}{${\eta'}$};
	   	\roundNbox{fill=white}{(.5,2.25)}{.4}{.5}{.5}{${\cF(\ev)}$};
	   	\roundNbox{fill=white}{(.5,1.1)}{.3}{.5}{.5}{${J}$};
	   	% additional draw args, center, radius, additional left x-space, additional right x-space, contents
	   	\roundNbox{fill=white}{(0,-.5)}{.3}{.05}{.05}{${\cF f}$};
	   	\filldraw[black] (1,.5) circle (1.5pt);
	   } 
	   \ar[r, "\clubsuit", Rightarrow]
	   &
	   \tikzmath{
	   	\draw[thick] (0,-2) -- (0, 2) arc (+180:0:.5cm) -- (1,.5)  arc (-180:0:.5cm) -- (2, 4);
	   	\roundNbox{fill=white}{(.5,3)}{.3}{0}{0}{${\eta'}$};
	   	% additional draw args, center, radius, additional left x-space, additional right x-space, contents
	   	\roundNbox{fill=white}{(0,-.5)}{.3}{.05}{.05}{${\cF f}$};
	   	\filldraw[black] (1,2) circle (1.5pt);
	   	\filldraw[black] (1,.5) circle (1.5pt);
	   } 
	   \ar[r, Rightarrow, "F(\text{cusp})"]
	   & 
	   \tikzmath{
	   	\draw[thick] (1, -2) -- (1, 4);
	   	\roundNbox{fill=white}{(0,2.333)}{.3}{0}{0}{${\eta'}$};
	   	\roundNbox{fill=white}{(1,.667)}{.3}{.05}{.05}{${\cF f}$};
	   }
	   \\ 
		\end{tikzcd}}; 
\end{tikzpicture}\]
Here we have been moderately abusive by denoting all the local applications of $\Theta$ with the same label. The leftmost vertical arrow is $\mu \circ (1 \boxtimes \cS) \circ \Delta (\Theta)$, and the rightmost is $\iota \circ \epsilon(\Theta)$. The first region commutes by naturality of the Sweedler isomorphism and the interchanger. The second region commutes by locality and the definition of $\cQ_J(\Theta) := \cQ(J, \otimes)(\Theta)$. The last three regions are commutative since $\Theta$ is a modification. 
\end{proof}
% \begin{prop}\ref{EndSplitsAcrossTensor}
%     Let $(\cC, \cF)$, $(\cC', \cF') \in \TwoSlice$. Then there is an equivalence of monoidal categories 
%     \[
%     \End(\cF) \boxtimes \End(\cF') \simeq \End(\cF \boxdot \cF') 
%     \]
%     given by $\eta \boxtimes \tau \mapsto \eta \boxdot \tau $.
% \end{prop}
% \begin{proof}
%    Let $c, c'$ be generators of $\cC, \cC'$. Then $c \boxdot c'$ generates $\cC \boxdot \cC'$.We have \nn{Thibualt2Tensor, 4.1} an equivalence of categories
%     \[
%     \Hom_{\TwoVec \boxtimes \TwoVec}(\cF c \boxdot \cF'c', \cF c \boxdot \cF'c') \simeq \Hom_\TwoVec(\cF c, \cF c) \boxtimes \Hom_\TwoVec(\cF' c', \cF' c')
%     \]
%     , commuting with composition up to natural isomorphism, i.e a monoidal equivalence. We recover the desired statement since natural transformations are determined by their action on generators. 
% \end{proof}
% \bibliographystyle{plain}
% \bibliography{bibliography}
\clearpage
\printbibliography
\end{document}